%% file: main.tex
\PassOptionsToPackage{dvipsnames}{xcolor}
\documentclass{article}

\usepackage{amsmath,amsthm,amsfonts}
\usepackage{authblk}
\usepackage{colonequals}
\usepackage{todonotes}
\usepackage{mathtools}
\usepackage{latexsym}
\usepackage{tikz,tikz-cd}
 \usepackage{quiver}
\usepackage{comment}
\usepackage{tikzit}				
	\input{markov2.tikzstyles}
	\usetikzlibrary{decorations.pathreplacing}	
\usepackage{enumitem}
	\setlist[enumerate]{label=(\roman*)}  
	\setlist[enumerate,2]{label=(\alph*)}  
\usepackage[T1,T2A]{fontenc}			
\usepackage[noadjust]{cite}
\usepackage{tikz-cd}
\usepackage{amssymb,xparse}		
\usepackage{scalerel}
\usepackage{geometry}
\usepackage{bm}				
\usepackage{centernot}
\usepackage[nottoc]{tocbibind}
\usepackage{needspace} 

\allowdisplaybreaks

\input{cache}

\definecolor{myurlcolor}{rgb}{0,0,0.3}
\definecolor{mycitecolor}{rgb}{0,0.3,0}
\definecolor{myrefcolor}{rgb}{0.3,0,0}
\usepackage[pagebackref,draft=false]{hyperref}
\hypersetup{colorlinks,
	linkcolor=myrefcolor,
	citecolor=mycitecolor,
urlcolor=myurlcolor}

\usepackage[noabbrev,capitalize]{cleveref}	

\newtheorem{theorem}{Theorem}[section]
\newtheorem{proposition}[theorem]{Proposition}
\newtheorem{lemma}[theorem]{Lemma}
\newtheorem{corollary}[theorem]{Corollary}
\newtheorem{definition}[theorem]{Definition}

\theoremstyle{definition}
\newtheorem{example}[theorem]{Example}

\newtheorem{remark}[theorem]{Remark}

\numberwithin{equation}{section}

\let\originalleft\left
\let\originalright\right
\renewcommand{\left}{\mathopen{}\mathclose\bgroup\originalleft}
\renewcommand{\right}{\aftergroup\egroup\originalright}

\newcommand{\N}{\mathbb{N}}

\newcommand{\R}{\mathbb{R}}
\newcommand{\E}{\mathbb{E}}	
\renewcommand{\P}{\mathbb{P}}	
\newcommand{\e}{\varepsilon}

\newcommand{\cat}[1]{{\mathrm{#1}}} 
\newcommand{\op}{\mathrm{op}}
\newcommand{\id}{\mathrm{id}} 		

\newcommand{\SI}{\cat{SI}}	
\newcommand{\DSI}{\cat{DSI}}	
\newcommand{\de}{\mathrm{de}}	
\newcommand{\dm}{\mathrm{dm}}	
\newcommand{\ess}{\mathrm{ess}} 

\newcommand{\newterm}[1]{\textbf{#1}}

\title{Convergence of martingales via enriched dagger categories}

\author{Paolo Perrone and Ruben Van Belle}
\affil{University of Oxford}
\date{}

\begin{document}

\maketitle

\begin{abstract}
We provide a categorical proof of convergence for martingales and backward martingales in mean, using enriched category theory. 
The enrichment we use is in topological spaces, with their canonical closed monoidal structure, which encodes a version of pointwise convergence.

We work in a topologically enriched dagger category of probability spaces and Markov kernels up to almost sure equality. In this category we can describe conditional expectations exactly as dagger-split idempotent morphisms, and filtrations can be encoded as directed nets of split idempotents, with their canonical partial order structure. As we show, every increasing (or decreasing) net of idempotents tends topologically to its supremum (or infimum).

Random variables on a probability space form contravariant functors into categories of Hilbert and Banach spaces, which we can enrich topologically  using the $L^p$ norms. Martingales and backward martingales can be defined in terms of these functors. Since enriched functors preserve convergence of nets, we obtain convergence in the $L^p$ norms. 
The convergence result for backward martingales indexed by an arbitrary net, in particular, seems to be new.

By changing the functor, one can describe more general notions of conditional expectations and martingales, and if the functor is enriched, one automatically obtains a convergence result.
For instance, one can recover the Bochner-based notion of vector-valued conditional expectation, and the convergence of martingales with values in an arbitrary Banach space.

This work seems to be the first application of topologically enriched categories to analysis and probability in the literature. We hope that this enrichment, so often overlooked in the past, will be used in the future to obtain further convergence results categorically.
\end{abstract}

\tableofcontents

\section{Introduction}

This work is part of an effort, started in the last few years, to express and to generalize the main results of probability theory using the language of category theory. 
Several results of probability theory and related fields have been recovered in this way, and sometimes extended. Among these, the Kolmogorov and Hewitt-Savage zero-one laws~\cite{fritzrischel2019zeroone,ensarguet2023ergodic}, de Finetti's theorem~\cite{fritz2021definetti,moss2022probability}, the Carathéodory extension theorem~\cite{vanbelle2022caratheodory}, the Radon-Nikodym theorem~\cite{vanbelle2023martingales}, and different versions of the ergodic decomposition theorem~\cite{moss2022ergodic,ensarguet2023ergodic}.
The main categorical structures used to achieve these purposes are \emph{probability monads}~\cite{giry}, \emph{Markov categories}~\cite{chojacobs2019strings,fritz2019synthetic}, and particular \emph{dagger categories}~\cite{wayofdagger}.

In this work, we turn to another cornerstone result of probability theory, the martingale convergence theorem.
It builds on previous work on martingales through the categorical lens of random variable functors \cite{adachi2018,dahlqvist2018borel,vanbelle2023martingales}, and on a categorical treatment of conditional expectations as particular idempotent morphisms \cite{fritz2023supports,ensarguet2023ergodic}.

The idea that martingale convergence is a limit or colimit-like property has already appeared in \cite{kozen2016martingales}, where it is argued that
it is like a universal property, but where \emph{certain key properties hold only up to a null set}. 
In this work we follow that intuition intuition using a category, first defined in \cite{dahlqvist2018borel}, where the morphisms are taken up to almost sure equality.

\paragraph{Probability spaces, kernels, and random variables.}
Among the most important structures of probability theory are probability spaces and random variables. In this work we give a categorical account of both: probability spaces as a category, and random variables on them as a functor.

In probability theory it is customary to fix a single probability space, the outcome space, and to work only on that space.
However, especially in the context of stochastic processes, one considers several sigma-algebras on that space at once, usually forming a filtration.
Technically, these sigma-algebras give rise to different, but related, probability spaces.
The categorical description that we use formalizes this intuition: we want probability spaces, with fixed sigma-algebras, to be the objects of a category, and refinements and coarse-grainings of these spaces to be the morphisms.
Concretely we work with a category whose objects are probability spaces, and whose morphisms are measure-preserving Markov kernels quotiented under a version of almost sure equality. This category, in the standard Borel case, was first defined in \cite{dahlqvist2018borel}.

Given a probability space $(X,\mathcal{A},p)$, one can consider the real-valued random variables on it which are measurable for $\mathcal{A}$ and integrable for $p$. 
Choosing a different probability space (for example, a different sigma-algebra), one obtains a different set of random variables. Moreover, refinements and coarse-grainings of the probability spaces have an effect on random variables, for example via conditional expectations.
This is encoded, category-theoretically, by modeling random variables as a \emph{functor} on the category of probability spaces. This idea was previously developed in \cite{vanbelle2023martingales,dahlqvist2018borel,adachi2018}, and here we give a further generalization of it. 
For an example of how this works, consider a probability space $(X,\mathcal{A},p)$ and a sub-sigma-algebra $\mathcal{B}\subseteq\mathcal{A}$. 
We have arrows between probability spaces as in the diagram below on the left,
\[
\begin{tikzcd}
  (X,\mathcal{A},p) \ar[shift right]{d}[swap]{\pi} \\
  (X,\mathcal{B},p) \ar[shift right]{u}[swap]{\pi^+}
\end{tikzcd}
\qquad\qquad
\begin{tikzcd}
  L^1(X,\mathcal{A},p) \ar[shift right,leftarrow]{d}[swap]{\pi^*} \\
  L^1(X,\mathcal{B},p) \ar[shift right,leftarrow]{u}[swap]{(\pi^+)^*}
\end{tikzcd}
\]
where we denote the restriction of $p$ to $\mathcal{B}$ again by $p$. 
The map $\pi$ is the (kernel induced by the) set-theoretical identity $X\to X$, which coarse-grains the sigma-algebra from $\mathcal{A}$ to $\mathcal{B}$. The kernel $\pi^+$ is the Bayesian inverse of $\pi$, which is guaranteed to exist, for example, if $(X,\mathcal{A})$ is a standard Borel space.
Random variables are now a contravariant functor. This means that we can assign to each probability space, for example $(X,\mathcal{A},p)$, the Banach space $L^1(X,\mathcal{A},p)$ of integrable random variables on it, again quotiented under almost sure equality. Moreover, we can assign to each morphism, i.e.~Markov kernel, a corresponding map between Banach spaces, in the opposite direction, as in the diagram above on the right. For example, 
\begin{itemize}
    \item The map $\pi^*:L^1(X,\mathcal{B},p)\to L^1(X,\mathcal{A},p)$, induced by the set-theoretical identity, is basically an inclusion map, saying that all $\mathcal{B}$-measurable random variables are $\mathcal{A}$-measurable as well;
    \item The map $(\pi^+)^*:L^1(X,\mathcal{A},p)\to L^1(X,\mathcal{B},p)$, induced by the Bayesian inverse, takes an $\mathcal{A}$-measurable random variable and gives its conditional expectation given $\mathcal{B}$.
\end{itemize}
In other words, the relations between the different probability spaces (for example, in a filtration) and their effects on random variables are captured categorically by means of a category and a functor.
For more details, see \Cref{sec_krnRV}.

This approach also lets us consider more general notions of random variables, for example vector-valued ones, integrable in the Bochner sense. 
These will simply form a different functor on our category. For more details, see \Cref{sec_V}.

\paragraph{Filtrations and martingales.}
A \emph{filtration} on a probability space $(X,\mathcal{A},p)$ is a sequence, or a net, of sub-sigma-algebras $\mathcal{B}_i\subseteq\mathcal{A}$ which are increasingly finer, that is, $\mathcal{B}_i\subseteq\mathcal{B}_j$ whenever $i\le j$.
Usually the filtration is indexed by a total order, which can be interpreted as time, so that as time progresses, our knowledge about the system progresses as well: we are able to make finer and finer distinctions, and to assess the probability of more and more possible events.
One can then form the \emph{join} sigma algebra $\mathcal{B}_\infty = \bigvee_i\mathcal{B}_i$, which can be interpreted as encoding all the knowledge that one can possibly learn from the process, if one had infinite time.
From the point of view of category theory, this sigma-algebra can be obtained as a particular directed limit, i.e. colimit of subobjects.

A \emph{martingale} on the filtration, following this interpretation, is a collection of random variables $f_i$ which, intuitively, follow the refinement of the filtration. 
More in detail, each $f_i$ needs to be measurable for the sigma-algebra $\mathcal{B}_i$, i.e.~it respects the state of knowledge that we have at ``time'' $i$. Moreover, for each $i\le j$, we have that $f_i$ is the conditional expectation of $f_j$. In other words, $f_i$ is a ``coarse graining'' of $f_j$ which ``forgets'' or ``averages over'' all the distinctions that we would be able to make at time $j$, but that we cannot yet make at time $i$. 
Dually, we can see $f_j$ as a ``refinement'' of $f_i$, which incorporates the additional ``knowledge'' that $\mathcal{B}_j$ has over $\mathcal{B}_i$. 

The \emph{martingale convergence theorem} says that, under some conditions,
\begin{itemize}
    \item The $f_i$ admits a common universal refinement $f$, measurable for the join sigma algebra $\mathcal{B}_\infty$, and moreover,
    \item The $f_i$ tend to $f$ topologically.
\end{itemize}
In other words, as our knowledge increases, our refinements become better and better ``approximations of the real thing''.
This ``approximation'' part is achieved, categorically, by means of a topological enrichment.

\paragraph{Topological enrichment.}
It is sometimes said that algebra concerns itself with equations and analysis concerns itself with inequalities.
While this is of course a simplification, it is true that in ordinary category theory it is not obvious how to express approximations and convergence, and the latter are extremely important in analysis and in probability theory.
Luckily, however, category theory has \emph{enriched} versions, enhancements of the notion of category where the sets of arrows are not just sets, but come equipped with additional structures, compatibly with composition. 
In particular, one can equip these sets with a metric or with a topology. This allows us to talk about convergence of measures and random variables categorically, and in a certain way, it adds approximations and inequalities to an otherwise purely algebraic, equational environment.

In \cite{vanbelle2023martingales}, a metric enrichment on a category of probability spaces and measure-preserving functions was used to study certain properties of martingales.
Also, in \cite{ours_entropy}, a metric enrichment on a category of kernels was used to recover classical concepts of information theory, such as entropy and data processing inequalities. 
In this work we bring the two ideas together, using enriched categories to study martingales, but using categories of kernels. Moreover, instead of metrics we endow our sets of arrows with more general topologies. The reason is that the notion of convergence between arrows that one obtains from the metric enrichment is a version of \emph{uniform} convergence, and this is too strong for our purposes. 
We are interested in a version of \emph{pointwise} convergence, and this is achieved by using a topological enrichment, where the sets of arrows have a topology which in general not metrizable.\footnote{The topologies we consider, however, can be induced by \emph{families} of metrics, and so we expect our categories to be enriched not only in topological spaces, but also in uniform spaces~\cite{uniform}. We will not pursue this idea in the present work.}
In the literature, this enrichment has not been used very often. In particular, it seems that it was never used to express, in terms of category theory, convergence results in probability or analysis. In that regard, this work might be the first one of its kind, where topological convergence and category theory are used together in this way.

Let's see how this works. First of all, one actually \emph{can} express certain ideas of approximation and convergence in ordinary categories: via particular limits and colimits (hence the name \emph{limit}). For example, the diagram in the category of sets formed by finite sets and inclusion maps
\[
\begin{tikzcd}
\{0\} \ar[hookrightarrow]{r}
 & \{0,1\} \ar[hookrightarrow]{r} 
 & \{0,1,2\} \ar[hookrightarrow]{r} 
 & \{0,1,2,3\} \ar[hookrightarrow]{r} 
 & \dots
\end{tikzcd}
\]
has as colimit the set of natural numbers $\mathbb{N}$. 
Similarly, and dually, one can express an infinite cartesian product as a limit of its finite factors. (This is particularly important in probability theory, as it is related to Kolmogorov's extension theorem, see \cite{fritzrischel2019zeroone}.)

The martingale convergence theorem, categorically, can be expressed as the fact that \emph{if we express an object as a categorical limit} (a probability space as a limit of increasingly finer sigma-algebras), \emph{we also have a topological convergence associated with it} (the conditional expectations of the random variables converge to a common refinement).

This correspondence between categorical limits and topological ones is a phenomenon that, in different ways, has been noticed before, for example in the category of Hilbert spaces \cite{dimeglio2024hilbert}, as well as for martingales \cite{kozen2016martingales}.
In this work, seemingly for the first time in the literature, we make this correspondence mathematically precise, by means of what we call the \emph{idempotent Levi property} of a topologically enriched category.
The details are explained in \Cref{sec_convergence}.

\paragraph{Our results.}
The main result of this paper is a version of the convergence theorem in mean for martingales and for backward martingales, which we prove for filtrations indexed by directed nets of arbitrary cardinality (\Cref{martin_up,martin_down}), and for random variables taking values in an arbitrary Banach space (\Cref{marti_bochner}).

Along the way, there are a number of additional results proven in this paper, which can be of independent interest.
First of all, we prove that random variables, including vector-valued ones, are functorial on probability spaces (\Cref{RV,RVV}).
We also prove that sub-sigma-algebras of a standard Borel space, up to null sets, are in bijection with almost surely idempotent kernels by means of taking conditional expectations and invariant sigma-algebras, and that the order of idempotents coincide with the (almost surely) inclusion order of sub-sigma-algebras (\Cref{krnsplit}).
We then give a characterization of martingales in terms of idempotents (\Cref{preserves_optima}), and in terms of limits and colimits of Banach and Hilbert spaces (\Cref{marti_ban,marti_hilb}).
The last step to prove martingale convergence is what we call the \emph{Levi property} of a topologically enriched category, which says that directed nets of idempotents converge topologically. We prove that such a property holds in the category of Hilbert spaces (\Cref{up_levi_hilb,down_levi_hilb}), and in our category of probability spaces (\Cref{levi_krn}).

As it is well known, there are results on convergence of martingales not only in mean, but also almost surely, at least for the case of sequential filtrations.
In this work we focus on convergence in mean, which implies convergence in probability, and leave the case of almost sure convergence to future work.
Since almost sure convergence does not arise from a topology, we expect that a categorical treatment of such a convergence will require much more refined enrichments.

\needspace{3\baselineskip}
\paragraph{Outline.}
\begin{itemize}
    \item In \Cref{sec_background} we give an overview of the main categorical concepts we need in this paper, dagger categories, and idempotents and their splittings. We show how, in categories of vector spaces, idempotents are projections onto subspaces. We conclude the section by looking at filtrations of subspaces and their limits.
    \item In \Cref{sec_krnRV} we turn to the main categories of interest for this work, of probability spaces and Markov kernels between them up to almost sure equality. We also define functors of random variables, and show how conditioning equips our categories with a dagger structure.
    \item In \Cref{sec_idempsigma} we join the structures introduced in the previous two sections, studying split idempotents in categories of kernels. We show that they are tightly related to sub-sigma-algebras and to conditional expectations of random variables.
    In particular they allow to give categorical descriptions of filtrations and of martingales.
    \item In \Cref{sec_enrichment} we provide our categories of kernels and of vector spaces with a topological enrichment, i.e.~a way to talk about convergence of morphisms in a ``pointwise-like'' sense. 
    We also show that the random variables functors are enriched, meaning that the functorial assignment on morphisms is continuous.
    \item In \Cref{sec_convergence} we finally show the main result of this work, namely convergence in mean martingales and backward martingales, for arbitrary nets, using categorical arguments. The main idea is what we call the \emph{Levi property}, which roughly says that convergence in the order implies topological convergence. The convergence of backward martingales indexed by arbitrary nets seems to be a new result.
    \item In \Cref{sec_V} we show that our convergence result can be extended to much more general notions of random variables, provided they form an enriched functor. We give the example of martingales with values in an arbitrary Banach space, using a Bochner-based approach to vector valued conditional expectations.
    \item Finally, in \Cref{sec_top} we give some background on topologically enriched categories, and on the closed monoidal structure of the enriching category $\cat{Top}$.
\end{itemize}

\paragraph{Acknowledgements.}
We want to thank Sam Staton and the members of his group for the enlightening discussions, and for providing a greatly supportive and inspiring research environment. 
We also want to thank Matthew Di Meglio for the helpful pointers on the category of Hilbert spaces.

Research for both authors is funded by Sam Staton's ERC grant ``BLaSt -- a Better Language for Statistics''.

\section{Categorical background}
\label{sec_background}

\subsection{Dagger categories}
\label{sec_dagger}

A dagger category can be interpreted as an undirected version of a category, similarly to directed and undirected graphs.
Just as for undirected graphs, in a dagger category if two objects are connected by a morphism $X\to Y$, then they are also connected by a morphism $Y\to X$, and we can think of these two arrows as a single entity which we can ``walk either way''. 
More specifically, to each morphism $f:X\to Y$ we assign a morphism $f^+:Y\to X$, not necessarily invertible, in a way which preserves identity and composition.
Here is the precise definition.

\begin{definition}
	A \newterm{dagger category} is a category $\cat{C}$ together with a functor $(-)^+:\cat{C}^\op\to\cat{C}$ which is
	\begin{itemize}
		\item The identity on objects, i.e.~$X^+=X$ for all $X$ of $\cat{C}$;
		\item Involutive, i.e.~for all morphisms, $(f^+)^+=f$.
	\end{itemize}
\end{definition}

Very often, dagger structures are used to encode ``Euclidean'' geometric structures, or notions of ``orthogonality'', as the following examples show.

\begin{example}
    The category $\cat{Euc}$ has as objects Euclidean spaces (finite-dimensional real vector spaces with an inner product), and as morphisms linear maps, or equivalently matrices.
    The transposition of matrices gives a dagger, as $(AB)^t=B^tA^t$. (Note that the inner product is necessary in order to have a coordinate-independent transposition.)
\end{example}

\begin{example}
    The category $\cat{Hilb}$ of real Hilbert spaces and bounded linear maps can be given a dagger structure via the \emph{adjoint}.
    Given Hilbert spaces $X$ and $Y$ and a bounded linear map $f:X\to Y$, its adjoint is the unique bounded linear map $f^+:Y\to X$ satisfying 
    \[
    \langle f(x), y \rangle  =  \langle x, f^+(y) \rangle 
    \]
    for all $x\in X$ and $y\in Y$. The brackets in the equation above are, in order, the inner product of $Y$, and the one of $X$.
\end{example}

\begin{definition}
	A morphism $f:X\to Y$ in a dagger category is called 
 	\begin{itemize}
 		\item a \newterm{dagger monomorphism} or \newterm{isometry} if $f^+\circ f = \id_X$;
 		\item a \newterm{dagger epimorphism} or \newterm{co-isometry} if $f\circ f^+=\id_Y$;
 		\item a \newterm{dagger isomorphism} or \newterm{unitary} if it is invertible, and $f^{-1}=f^+$;
        \item \newterm{self-dual} or \newterm{self-adjoint} if $X=Y$ and $f^+=f$.
 	\end{itemize}
\end{definition}

For brevity we will just say ``dagger mono'' for ``dagger monomorphism'', as well as ``dagger monic'', et cetera.
Note that a dagger mono is in particular a split mono, a dagger epi is in particular a split epi, and a dagger iso is in particular an isomorphism.
Equivalently, a dagger isomorphism is a morphism which is both dagger monic and dagger epic.

\begin{example}
    In $\cat{Euc}$,
    \begin{itemize}
        \item The dagger monomorphisms are exactly the linear isometries, or linear isometric embeddings, namely the matrices which preserve the Euclidean norms and inner products of vectors:
        \[
        \langle f(x) , f(x') \rangle = \langle x, f^+ f(x') \rangle = \langle x, x' \rangle .
        \]
        Dagger epimorphisms are their transposes;
        \item The dagger isomorphisms are exactly the orthogonal matrices;
        \item The self-dual morphisms are exactly the symmetric matrices.
    \end{itemize}
\end{example}

\begin{example}
    In $\cat{Hilb}$,
    \begin{itemize}
        \item The dagger monomorphisms are exactly the linear isometries (and their duals are sometimes called \emph{co-isometries} of Hilbert spaces);        
        \item The dagger isomorphisms are exactly the unitary maps;
        \item The self-dual morphisms are exactly the self-adjoint operators.
    \end{itemize}
\end{example}

Dagger monos and dagger epis are respectively closed under composition. In what follows we will write
\begin{itemize}
	\item $\cat{C}_\dm$ for the wide subcategory of dagger monos of $\cat{C}$;
	\item $\cat{C}_\de$ for the wide subcategory of dagger epis of $\cat{C}$. 
\end{itemize}

\subsection{Idempotents and their splittings}
\label{sec_idemp}

Let's now look at idempotents. We first start with idempotents in general, and then turn to the dagger setting.
An \newterm{idempotent} morphism in a category $\cat{C}$ is an endomorphism $e:X\to X$ such that $e\circ e=e$. 

\begin{definition}
	A \newterm{splitting} of an idempotent $e:X\to X$ consists of an object $A$ and morphisms $\iota:A\to X$ and $\pi:X\to A$ such that $\pi\circ\iota=\id_A$ and $\iota\circ\pi=e$. 
\end{definition}

Note that this implies that $\iota$ is split monic, and $\pi$ is split epic.

Very often, split idempotents encode the idea of ``projections onto a subspace'', as in the following examples.

\begin{example}\label{id_split_euc}
    In $\cat{Euc}$, all idempotents split. 
    Given an idempotent linear map $e$ on a Euclidean space $X$, denote by $A\subseteq X$ its image, and by $\iota:A\to X$ the inclusion. By definition of image, for all $x\in X$ we have that $e(x)\in A$, so $e$ specifies a linear map $\pi:X\to A$ by $x\mapsto e(x)$. 
    This way, $e(x)=\iota(\pi(x))$, and for all $y\in A$, for some $x$, $y=e(x)$, and so by idempotency, $\pi(\iota(y))=e(e(x))=e(x)=y$.
    Therefore $\pi\circ\iota=\id_A$ and $\iota\circ\pi=e$. 
    In other words, we can see the idempotent $e$ as a projection onto a the subspace $A$.
\end{example}
Let $\cat{Ban}$ be the category of Banach spaces and \emph{bounded} linear maps and let $\cat{Ban}_{\le 1}$ be the category of Banach spaces and \emph{1-Lipschitz} linear maps.
\begin{example}\label{id_split_ban}
    In $\cat{Ban}$ all idempotents split as well, and the intuition is similar.
    Let $e$ be an idempotent bounded linear map on a Banach space $X$. Let $A\subseteq X$ be its image, and denote by $\iota:A\to X$ be the inclusion map. 
    Note first of all that $A$ is a \emph{closed} subspace: suppose that for a sequence $(x_n)$ in $X$ we have that $e(x_n)\to y$. Then since $e$ is continuous and idempotent we have that
    \[
    e(y) = \lim_{n\to\infty} e(e(x_n)) = \lim_{n\to\infty} e(x_n) = y .
    \]
    So $y\in A$, which means that $A$ is closed, and so it is itself a Banach space.
    Just as in the example above, $e$ defines a map $\pi:X\to A$, and we have that $\pi\circ\iota=\id_A$ and $\iota\circ\pi=e$. 
    Once again, we can view the idempotent $e$ as a projection onto the closed subspace $A$.
    Note that not every closed subspace of a Banach space admits a (bounded) idempotent operator projecting onto it, and even when it exists, it may not be unique.

    In $\cat{Hilb}$, the situation is completely analogous, with the additional condition that \emph{every closed subspace admits an idempotent}. Uniqueness still fails, but there is a canonical choice, see \Cref{orthogonal_projection}. 
\end{example}

Equivalently, one can construct $A$ as either the set of fixed points of $e$, or as the quotient set of fibers of $e$.
Categorically, $A$ can be constructed as either the equalizer or the coequalizer of the parallel pair $e,\id:X\to X$, and it is preserved by every functor~\cite[Proposition~1]{cauchycompletion}.

Let's now turn to the dagger case.

\begin{definition}
	Let $\cat{C}$ be a dagger category.
	\begin{itemize}
	    \item A \newterm{dagger idempotent} or \newterm{projector} is a self-dual idempotent morphism. That is, a morphism $e:X\to X$ such that $e\circ e=e=e^+$;
        \item A \newterm{dagger splitting} of an idempotent $e:X\to X$ is a splitting $(A,\iota,\pi)$ such that $\iota=\pi^+$. 
	\end{itemize}
\end{definition}

Note that in the definition above, we necessarily have that $\iota$ is dagger monic and that $\pi$ is dagger epic. 
Note also that if an idempotent has a dagger splitting, then it is a dagger idempotent:
\[
e^+ \;=\; (\iota\circ\pi)^+ \;=\; \pi^+\circ\iota^+ \;=\; \iota\circ\pi \;=\; e .
\]

We have seen that split idempotents often encode ``projections onto subspaces'', and that dagger structures often encode notions of ``orthogonality''. Following this intuition, dagger idempotents often encode \emph{orthogonal projections} onto subspaces:

\begin{example}\label{orthogonal_projection}
    In $\cat{Euc}$ as well as in $\cat{Hilb}$, consider a dagger idempotent $e:X\to X$ and its splitting $(A,\iota,\pi)$ as in \Cref{id_split_euc,id_split_ban}.
    We have that for every $x\in X$ and $y\in A$,
    \begin{align*}
    \langle x - e(x), y \rangle &= \langle x , y \rangle - \langle e(x), y \rangle \\
    &= \langle x , y \rangle - \langle x, e(y) \rangle \\
    &= \langle x , y \rangle - \langle x, y \rangle \\
    &= 0 .
    \end{align*}
    Therefore $e$ is an orthogonal projector onto $A$. To see that our splitting is dagger, notice that for all $x\in X$ and $y\in A$,
    \[
    \langle \pi(x), y \rangle = \langle e(x), y \rangle = \langle x, e(y) \rangle = \langle x, y \rangle = \langle x, \iota(y) \rangle .
    \]
    Therefore $\iota=\pi^+$. 
    Conversely, by taking the adjoint of the inclusion map, every closed subspace admits a (unique) orthogonal projector.
\end{example}

\subsection{The order of split idempotents}
\label{sec_order}

Subspaces of a vector space form naturally a partial order under inclusion. The same, dually, can be said about quotients. More generally, subobjects and quotient objects form partial orders.
Retracts, or split idempotents, can be seen in this light, and inherit a partial order with interesting properties.

\begin{proposition}\label{eqcondorder}
Let $e_1$ and $e_2$ be idempotents $X\to X$ in a category, with splittings $(A_1,\pi_1,\iota_1)$ and $(A_2,\pi_2,\iota_2)$ respectively.
\[
\begin{tikzcd}
 & X \ar[shift right]{dl}[swap]{\pi_1} \ar[shift left]{dr}{\pi_2} \\
 A_1 \ar[shift right]{ur}[swap]{\iota_1} && A_2 \ar[shift left]{ul}{\iota_2}
\end{tikzcd}
\]
The following conditions are equivalent. 
\begin{enumerate}
 \item\label{eep} $e_1\circ e_2 = e_2\circ e_1 = e_1$;
 \item\label{pinv} $e_2\circ\iota_1=\iota_1$ and $\pi_1\circ e_2=\pi_1$;
 \item\label{f} There exists $f:A_1\to A_2$ and $g:A_2\to A_1$ such that the following diagrams commute.
\end{enumerate}
\[
\begin{tikzcd}
 & X \\
 A_1 \ar{ur}{\iota_1} \ar[dashed]{rr}[swap]{f} && A_2 \ar{ul}[swap]{\iota_2}
\end{tikzcd}
\qquad\qquad
\begin{tikzcd}
 & X \ar{dl}[swap]{\pi_1} \ar{dr}{\pi_2} \\
 A_1 && A_2 \ar[dashed]{ll}{g}
\end{tikzcd}
\]
Moreover, in the situation above, $f$ and $g$ are unique, and $g\circ f=\id_{A_1}$.

If in addition we are in a \emph{dagger} category, and $(\iota_1,\pi_1)$ and $(\iota_2,\pi_2)$ are \emph{dagger} splittings, then $f=g^+$.
\end{proposition}

\begin{proof}
 $\ref{eep}\Rightarrow\ref{pinv}$:
 \[
  e_2\circ\iota_1 \;=\; e_2\circ e_1\circ\iota_1 \;=\; e_1\circ\iota_1 \;=\; \iota_1 ;
 \]
 \[
  \pi_1\circ e_2 \;=\; \pi_1\circ e_1\circ e_2 \;=\; \pi_1\circ e_1 \;=\; \pi_1 .
 \]

 $\ref{pinv}\Rightarrow\ref{f}$:
 Notice that $\iota_2\circ f=\iota_1$ implies necessarily that $f=\pi_2\circ\iota_1$, so that $f$ is unique.
 This way,
 \[
  \iota_2\circ f \;=\; \iota_2\circ\pi_2\circ\iota_1 \;=\; e_2\circ\iota_1 \;=\; \iota_1 .
 \]
 Similarly, if we want that $g\circ\pi_2=\pi_1$, necessarily $g=\pi_1\circ\iota_2$.
 Now 
 \[
 g\circ\pi_2 \;=\; \pi_1\circ\iota_2\circ\pi_2 \;=\; \pi_1\circ e_2 \;=\; \pi_1 .
 \]
 To show that $g$ is a retraction,
 \[
 g\circ f \;=\; \pi_1\circ\iota_2\circ \pi_2\circ\iota_1 \;=\; \pi_1\circ e_2\circ\iota_1 \;=\; \pi_1\circ\iota_1 = \id_{A_1} . 
 \]

 $\ref{f}\Rightarrow\ref{eep}$:
 \[
  e_1\circ e_2 \;=\; \iota_1\circ\pi_1\circ \iota_2\circ \pi_2 \;=\; \iota_1\circ g\circ \pi_2\;=\; \iota_1\circ\pi_1 \;=\; e_1. 
 \]
 \[
  e_2\circ e_1 \;=\; \iota_2\circ\pi_2\circ \iota_1\circ \pi_1 \;=\; \iota_2\circ f\circ \pi_1\;=\; \iota_1\circ\pi_1 \;=\; e_1. 
 \]

 If we are in a dagger category and $\pi_1=\iota_1^+$ and $\pi_2=\iota_2^+$, then 
 \[
 g^+ \;=\; (\pi_1\circ\iota_2)^+ \;=\; \iota_2^+\circ\pi_1^+ \;=\; \pi_2\circ\iota_1 \;=\; f . \qedhere
 \]
\end{proof}

\begin{definition}
 In the hypotheses of the previous proposition, we say that $e_1\le e_2$ if any (hence all) the conditions are satisfied.
\end{definition}

\begin{remark}
In the dagger case it is even simpler to check that $e_1\le e_2$: by taking the dagger,
\begin{itemize}
    \item $e_1\circ e_2=e_1$ if and only if $e_2\circ e_1=e_1$;
    \item $e_2\circ\iota_1=\iota_1$ if and only if $\pi_1\circ e_2=\pi_1$;
    \item $f$ as in \Cref{eqcondorder} exists if and only if $g$ exists.
\end{itemize}
\end{remark}

\begin{example}
    In $\cat{Euc}$, $\cat{Ban}$ and $\cat{Hilb}$, in the hypotheses of the previous proposition, we have that $e_1\le e_2$ if and only if, as subspaces, $A_1\subseteq A_2$. The equivalent conditions read:
    \begin{enumerate}
        \item Projecting first onto $A_2$ and then onto $A_1$ is the same as just projecting onto $A_1$ directly;
        \item Projecting an element of $A_1$ onto $A_2$ leaves it unchanged;
        \item The projection $X\to A_1$ and the inclusion $A_1\to X$ both factor through $A_2$.
    \end{enumerate}
\end{example}

\begin{proposition}
 The relation $\le$ on the set of split idempotents on $X$ is a partial order.
\end{proposition}

\begin{proof}
 For reflexivity, notice that $e\circ e=e$ by idempotency.

 For transitivity, suppose that $e_1\circ e_2=e_1$ and $e_2\circ e_3=e_2$. Then
 \[
  e_1\circ e_3 \;=\; e_1\circ e_2\circ e_3 \;=\; e_1\circ e_2 \;=\; e_1 .
 \]

 For antisymmetry, notice that $e_1=e_1\circ e_2=e_2$ implies $e_1=e_2$.
\end{proof}

We can equivalently define the order of idempotents as the following category.
\begin{definition}
    Let $X$ be an object of a category $\cat{C}$.
The category $\SI(X)$ has:
\begin{itemize}
	\item As objects, objects $A$ of $\cat{C}$ together with a split idempotent $(\iota:A\to X,\pi:X\to A)$;
	\item As morphisms $(A_1,\iota_1,\pi_1)\to(A_2,\iota_2,\pi_2)$, pairs of morphisms $f_{12}:A_1\to A_2$ and $g_{21}:A_2\to A_1$ and making the top and bottom triangles in the following diagram commute.
	\[
\begin{tikzcd}[column sep=small]
    & X \ar{dl}[swap]{\pi_1} \ar{dr}[near end]{\pi_2} \\
     A_1 \ar{dr}[swap]{\iota_1} \ar[shift right]{rr}[swap]{f_{12}} 
     && A_2 \ar{dl}[near start]{\iota_2} \ar[shift right]{ll}[swap]{g_{21}} \\
	& X
\end{tikzcd}
\]
\end{itemize}
\end{definition}
This category is a preorder, equivalent to the poset of split idempotents on $X$.

In the dagger case, we can define things similarly. 
\begin{definition}
    Let $X$ be an object of a dagger category $\cat{C}$.

    The category $\DSI(X)$ has:
\begin{itemize}
	\item As objects, objects $A$ of $\cat{C}$ together with a dagger-split idempotent $(\iota:A\to X,\pi:X\to A)$;
	\item As morphisms $(A_1,\iota_1,\pi_1)\to(A_2,\iota_2,\pi_2)$, pairs of morphisms $f_{12}:A_1\to A_2$ and making the top and bottom triangles in the following diagram commute.
	\[
\begin{tikzcd}[column sep=small]
    & X \ar{dl}[swap]{\pi_1} \ar{dr}[near end]{\pi_2} \\
     A_1 \ar{dr}[swap]{\iota_1} \ar[shift right]{rr}[swap]{f_{12}} 
     && A_2 \ar{dl}[near start]{\iota_2} \ar[shift right]{ll}[swap]{f_{12}^+} \\
	& X
\end{tikzcd}
\]
\end{itemize}
\end{definition}

In the dagger case, thanks to the self-duality, can have two equivalent descriptions of the $\DSI(X)$.
The slice category $\cat{C}_\dm/X$ has:
\begin{itemize}
	\item As objects, objects $A$ of $\cat{C}$ together with a dagger monomorphism $\iota:A\to X$;
	\item As morphisms, dagger monomorphisms $f_{12}:A_1\to A_2$ making the following triangle commute.
	\[
	\begin{tikzcd}[row sep=small, column sep=tiny]
		A_1 \ar{dr}[swap]{\iota_1} \ar{rr}{f_{12}} && A_2 \ar{dl}{\iota_2} \\
		& X 
	\end{tikzcd}
	\]
\end{itemize}
Dually, the coslice category $X/\cat{C}_\de$ has:
 \begin{itemize}
  \item As objects, objects $Y$ of $\cat{C}$ together with a dagger epimorphism $\pi:X\to Y$;
  \item As morphisms, dagger epimorphisms $g_{21}:Y_2\to Y_1$ making the following triangle commute.
  \[
  \begin{tikzcd}[row sep=small, column sep=tiny]
   & X \ar{dl}[swap]{\pi_2} \ar{dr}{\pi_1} \\
   Y_2 \ar{rr}[swap]{g_{21}} && Y_1
  \end{tikzcd}
  \]
\end{itemize}
The category $\cat{C}_\dm/X$ is a preorder, equivalent to the poset of dagger-split idempotents on $X$, and the category $X/\cat{C}_\de$ is equivalent to the \emph{opposite} of the poset of dagger-split idempotents on $X$.

\subsection{Sequential and filtered suprema and infima}
\label{sec_supinf}

Consider an ascending chain $e_1\le e_2\le\dots$ of idempotents on $X$. Equivalently it is an inductive sequence in $\SI(X)$, as follows:
\[
\begin{tikzcd}[row sep=huge, column sep=small]
    &&&& X \ar{dllll}[swap]{\pi_1} \ar{dl}{\pi_2} \ar{dr}[near end]{\pi_3} \ar{drrrr} \\
    A_1 \ar{drrrr}[swap]{\iota_1} \ar[shift right]{rrr}[swap, near end]{f_{12}} 
     &&& A_2 \ar{dr}{\iota_2} \ar[shift right]{rr}[swap]{f_{23}} \ar[shift right]{lll}[swap, near start]{g_{21}}
     && A_3 \ar{dl}[near start]{\iota_3} \ar[shift right]{rrr} \ar[shift right]{ll}[swap]{g_{32}}
     &&& \cdots \ar{dllll} \ar[shift right]{lll} \\
	&&&& X
\end{tikzcd}
\]

The sequential colimit $A_\infty$ of the $A_i$ in $\SI(X)$, if it exists, is equivalently the supremum in the order of idempotents.
The same thing is true for increasing (directed) nets instead of sequences.

Similarly, consider a descending chain $e_1\ge e_2\ge\dots$ of idempotents. This is equivalently a coinductive sequence in $\SI(X)$, as follows:
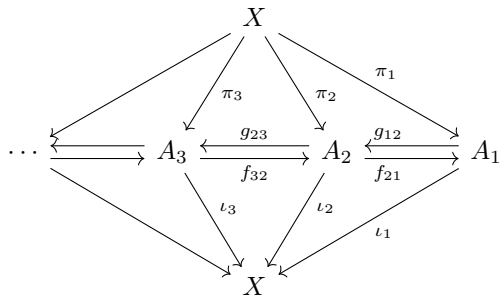
\[
\begin{tikzcd}[row sep=huge, column sep=small]
    &&&& X \ar{dllll} \ar{dl}{\pi_3} \ar{dr}[near end]{\pi_2} \ar{drrrr}{\pi_1} \\
    \cdots \ar{drrrr} \ar[shift right]{rrr}
     &&& A_3 \ar{dr}{\iota_3} \ar[shift right]{rr}[swap]{f_{32}} \ar[shift right]{lll}
     && A_2 \ar{dl}[near start]{\iota_2} \ar[shift right]{rrr}[swap, near start]{f_{21}} \ar[shift right]{ll}[swap]{g_{23}}
     &&& A_1 \ar{dllll}{\iota_1} \ar[shift right]{lll}[swap, near end]{g_{12}} \\
	&&&& X
\end{tikzcd}
\]

The sequential limit $A_\infty$ of the $A_i$ in $\SI(X)$, if it exists, is equivalently the infimum in the order of idempotents.
The same is true for decreasing nets instead of sequences.

Similar remarks can be made for dagger-split idempotents. 

\begin{example}
    In $\cat{Euc}$ and $\cat{Hilb}$, increasing and decreasing sequences of dagger idempotents are increasing and decreasing sequences of (closed) subspaces.
    Their suprema are given by (the closure of) their union, and their infima are given by their intersection.
    The same is true for nets.
\end{example}

\begin{example}
 In $\cat{Ban}$, not every closed subspace gives rise to an idempotent, and when it does it may not be unique, so we need a little care.
 What we can say is that
    \begin{enumerate}
        \item Given an increasing net of idempotents $e_\lambda:X\to X$ projecting onto closed subspaces $A_\lambda\subseteq X$, denote by $A_\infty$ the closure of their union. Then if the supremum of the $e_\lambda$ in the idempotent order exists, it projects onto $A_\infty$.
        
        \item Given a decreasing net of idempotents $e_\lambda:X\to X$ projecting onto closed subspaces $A_\lambda\subseteq X$, denote by $A_\infty$ their intersection. Then if the infimum of the $e_\lambda$ in the idempotent order exists, it projects onto $A_\infty$.
    \end{enumerate}
\end{example}

Moreover, in the category $\cat{Ban}_{\le 1}$ of Banach spaces and \emph{1-Lipschitz} linear maps (not just bounded), 
\begin{itemize}
    \item nets of increasing subspaces have the closure of their union as colimit of the inclusion maps;
    \item nets of decreasing subspaces have their intersection as limit of the inclusion maps.
\end{itemize}
This is shown in the following propositions. 
The same statements will also hold in the subcategory $\cat{Hilb}_{\le 1}$ of Hilbert spaces and 1-Lipschitz linear maps, and we can interpret them in terms of the category $\DSI$. 

\begin{proposition}\label{banupi}
    Let $X$ be a Banach space. Let $(A_\lambda)_{\lambda\in\Lambda}$ be an increasing net of closed subspaces of $X$, i.e.~such that for all all $\lambda\le\mu$, $A_\lambda\subseteq A_{\mu}$. Let 
    \[
    A_\infty = \mathrm{cl} \left( \bigcup_\lambda A_\lambda \right) .
    \]
    For $\lambda\le\mu$, denote by $\iota_{\lambda,\mu}:A_\lambda\to A_\mu$ the 1-Lipschitz inclusion.
    Then $A_\infty$ is the (filtered) colimit in $\cat{Ban}_{\le 1}$ of the diagram formed by the $\iota_{\lambda,\mu}$.
\end{proposition}

\begin{proof}
    Let $B$ be a Banach space, and consider a cone of $1$-Lipschitz linear maps $c_\lambda:A_\lambda\to B$ making the following diagram commute for each $\mu\ge\lambda$.
    \begin{equation}\label{coneah}
    \begin{tikzcd}[row sep=small]
        A_\lambda \ar{dd}[swap]{\iota_{\lambda,\mu}} \ar{dr}{c_\lambda} \\ 
        & B \\
        A_\mu \ar{ur}[swap]{c_\mu}
    \end{tikzcd}
    \end{equation}
    We have to show that there is a unique map $c:A_\infty\to B$ making the following diagram commute.
    \[
    \begin{tikzcd}
        A_\lambda \ar{dr}[swap]{\iota_\lambda} \ar{dd}[swap]{\iota_{\lambda,\mu}} \ar{drr}{c_\lambda} \\ 
        & A_\infty \ar[dashed]{r} & B \\
        A_\mu \ar{ur}{\iota_\mu} \ar{urr}[swap]{c_\mu}
    \end{tikzcd}
    \]
    Now define first the map 
    \[
    c' : \bigcup_{\lambda\in\Lambda} A_\lambda \to B 
    \]
    by $c'(a_\lambda)=c_\lambda(a_\lambda)$ for all $a_\lambda\in A_\lambda$. Note that since \eqref{coneah} commutes, this map does not depend on the choice of $\lambda$. 
    Since all the $c_\lambda$ are 1-Lipschitz, so is $c'$, and so by density $c'$ admits a unique (linear) extension to $A_\infty$. 
\end{proof}

\begin{proposition}\label{bandowni}
    Let $X$ be a Banach space. Let $(A_\lambda)_{\lambda\in\Lambda}$ be a decreasing net of closed subspaces of $X$, i.e.~such that for all all $\lambda\le\mu$, $A_\lambda\supseteq A_{\mu}$. Let 
    \[
    A_\infty =  \bigcap_\lambda A_\lambda .
    \]
    For $\lambda\le\mu$, denote by $\iota_{\mu,\lambda}:A_\mu\to A_\lambda$ the inclusion.
    Then $A_\infty$ is the (cofiltered) limit in $\cat{Ban}$ and $\cat{Ban}_{\le 1}$ of the diagram formed by the $\iota_{\mu,\lambda}$.
\end{proposition}

\begin{proof}
    Let $B$ be a Banach space, and consider a cone consisting of bounded linear maps $c_\lambda:B\to A_\lambda$ (for each $\lambda$) making the following diagram commute for each $\lambda\le\mu$.
    \[
    \begin{tikzcd}[row sep=small]
        & A_\mu \ar{dd}{\iota_{\mu,\lambda}} \\
        B \ar{ur}{c_\mu} \ar{dr}[swap]{c_\lambda} \\
        & A_\lambda
    \end{tikzcd}
    \]
    The commutativity of the diagram above says that for every $b\in B$, and for every $\mu\ge\lambda$, we have that $c_\lambda(b)\in A_\mu$.
    Since any $\lambda,\lambda'\in\Lambda$ admit an upper bound $\mu$, for every $\lambda,\lambda'\in\Lambda$, $c_\lambda(b)\in A_\mu\subseteq A_{\lambda'}$. Therefore, for every $\lambda$, $c_\lambda(b)\in\bigcap_{\lambda'} A_{\lambda'}=A_\infty$. 
    In other words, all the $c_\lambda$ factor uniquely through a map $c:B\to A_\infty$ as follows,
    
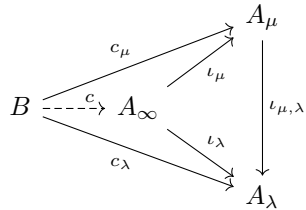
\[
    \begin{tikzcd}
        && A_\mu \ar{dd}{\iota_{\mu,\lambda}} \\
        B \ar{urr}{c_\mu} \ar{drr}[swap]{c_\lambda} \ar[dashed]{r}[near end]{c} & A_\infty \ar{ur}[swap]{\iota_\mu} \ar{dr}{\iota_\lambda} \\
        && A_\lambda 
    \end{tikzcd}
    \]
    where $\iota_\lambda$ and $\iota_\mu$ are the inclusions.
    That is, $A_\infty$ is the limit of the cofiltered diagram.
\end{proof}

Since these statements also hold in $\cat{Hilb}_{\le 1}$, which is a dagger category, by dualizing we get the following statements, involving the projections instead of the inclusions:

\begin{corollary}\label{hilbupp}
    Let $X$ be a Hilbert space. Let $(A_\lambda)_{\lambda\in\Lambda}$ be an increasing net of closed subspaces of $X$, i.e.~such that for all all $\lambda\le\mu$, $A_\lambda\subseteq A_{\mu}$. Let 
    \[
    A_\infty = \mathrm{cl} \left( \bigcup_\lambda A_\lambda \right) .
    \]
    For $\lambda\le\mu$, denote by $\pi_{\mu,\lambda}:A_\mu\to A_\lambda$ the orthogonal projection.
    Then $A_\infty$ is the (cofiltered) limit in $\cat{Hilb}_{\le 1}$ of the diagram formed by the $\pi_{\mu,\lambda}$.
\end{corollary}

\begin{corollary}\label{hilbdownp}
    Let $X$ be a Hilbert space. Let $(A_\lambda)_{\lambda\in\Lambda}$ be a decreasing net of closed subspaces of $X$, i.e.~such that for all all $\lambda\le\mu$, $A_\lambda\supseteq A_{\mu}$. Let 
    \[
    A_\infty = \bigcap_\lambda A_\lambda .
    \]
    For $\lambda\le\mu$, denote by $\pi_{\lambda,\mu}:A_\lambda\to A_\mu$ the projection.
    Then $A_\infty$ is the (filtered) colimit in $\cat{Hilb}_{\le 1}$ of the diagram formed by the $\pi_{\lambda,\mu}$.
\end{corollary}

Note that the same statements does not hold more in general for Ban.
A counterexample for the analogue of \Cref{hilbupp} in the Banach case will be given in \Cref{non_int_mart}.
A counterexample for the analogue for \Cref{hilbdownp} in the Banach case will be given in \Cref{non_lim_ban}.

Let's now interpret this statement, for $\cat{Hilb}$, in terms of the category $\DSI$. 
Note first of all that the dagger monic and the dagger epic morphisms of $\cat{Hilb}$ are 1-Lipschitz.
Now, it is well known that for every category $\cat{C}$ and object $X$, the forgetful functor 
\[
\begin{tikzcd}[row sep=0]
    X/\cat{C} \ar{r} & C \\
    (Y,X\xrightarrow{f} Y) \ar[mapsto]{r} & Y
\end{tikzcd}
\]
creates all limits and all connected colimits which exist in $\cat{C}$~\cite[Proposition~3.3.8]{riehl2016category}.
Similarly, the forgetful functor $\cat{C}/X\to\cat{C}$ creates all colimits and all connected limits which exist in $\cat{C}$.
Setting now $\cat{C}=\cat{Hilb}_\dm$, we have that the forgetful functor $\DSI(X)\cong \cat{Hilb}_\dm/X\to\cat{Hilb}_\dm$ creates all existing connected limits and colimits, in particular filtered colimits and cofiltered limits. 
This, together with the previous statements, gives an alternative proof of the fact that intersections and closures of unions are the directed infima and suprema in the order of dagger idempotents of $\cat{Hilb}$.

\section{Kernels and random variables, categorically}
\label{sec_krnRV}

We now turn our attention to probability theory. We define a (dagger) category of probability spaces, with Markov kernels as morphisms, and functors encoding random variables. 

\subsection{Categories of Markov kernels}
\label{sec_krn}

We now present categories, $\cat{Krn}$, $\cat{GKrn}$ and $\cat{EKrn}$, of probability spaces and Markov kernels up to almost sure equality. $\cat{Krn}$ was first defined in \cite{dahlqvist2018borel}.
These categories should not be confused with the category, often denoted by $\cat{Stoch}$, of \emph{measurable} spaces and Markov kernels, without quotienting under almost sure equality.
The two constructions are however related, as shown in the abstract setting of \cite[Section~13]{fritz2019synthetic}.
In \cite{ensarguet2023ergodic} (where $\cat{GKrn}$ is called $\cat{PS(Stoch)}$ and $\cat{GKrn}$ is called $\cat{PS(BorelStoch)}$) some of the properties of these categories were studied, which are quite relevant for the present work. 

\begin{definition}\label{defkern}
    Let $(X,\mathcal{A})$ and $(Y,\mathcal{B})$ be measurable spaces. A \newterm{Markov kernel} $k:X\to Y$ is an assignment 
    \[
    \begin{tikzcd}[row sep=0]
        X\times \mathcal{B} \ar{r} & {[0,1]} \\
        (x,B) \ar[mapsto]{r} & k(B|x)
    \end{tikzcd}
    \]
    which for every $B$ is measurable as a function of $x$, and for every $x$ is a probability measure in $B$.

    Given Markov kernels $k:X\to Y$ and $\ell:Y\to Z$ we define the Markov kernel $\ell\circ k:X\to Z$ as follows:
    \[
    (\ell\circ k) (C|x) \coloneqq \int_Y \ell(C|y)\,k(dy|x) =\int_Y \ell(C|\cdot)\text{d}k(\cdot |x) .
    \]
    Given probability spaces $(X,\mathcal{A},p)$ and $(Y,\mathcal{B},q)$, a Markov kernel $k:X\to Y$ is called \newterm{measure-preserving} if and only if for every $B\in\mathcal{B}$,
    \[
    \int_X k(B|x) \,p(dx) = q(B) .
    \]

    Given probability spaces $(X,\mathcal{A},p)$ and $(Y,\mathcal{B},q)$, two Markov kernels $k,h:X\to Y$ are \newterm{$p$-almost surely equal} if and only if for every measurable $B\in\mathcal{B}$, for $p$-almost all $x\in X$,
    \[
    k(B|x) = h(B|x) .
    \]
\end{definition}

\begin{definition}
    The category $\cat{GKrn}$ has 
    \begin{itemize}
        \item As objects, probability spaces;
        \item As morphisms, measure-preserving Markov kernels, quotiented under almost-sure equality.
    \end{itemize}
\end{definition}

The ``G'' in ``$\cat{GKrn}$'' stands for ``general'', as we also consider the following subcategories:

\begin{definition}
    The category $\cat{Krn}$ is the full subcategory of $\cat{GKrn}$ whose objects are probability spaces $(X,\mathcal{A},p)$, where the measurable space $(X,\mathcal{A})$ is standard Borel;
\end{definition}

\begin{definition}
    We call a probability space \newterm{essentially standard Borel} if and only if it is isomorphic in $\cat{GKrn}$ to a standard Borel space. The category $\cat{EKrn}$ is the full subcategory of $\cat{GKrn}$ whose objects are essentially standard Borel. 
\end{definition}

\begin{remark}
    The inclusion functor $\cat{Krn}\to\cat{EKrn}$ is an equivalence of categories (but not an isomorphism, since it's only \emph{essentially} surjective on objects).
    An example of a measure space which is in $\cat{EKrn}$ but not in $\cat{Krn}$ is is given by $\R$ with the codiscrete sigma-algebra $\{\R,\varnothing\}$, and equipped with the unique probability measure it admits:\footnote{We thank an anonymous reviewer for suggesting this example.}
    \begin{itemize}
        \item It is not an object of $\cat{Krn}$, as the measurable space is not standard Borel (the sigma-algebra does not even separate points);
        \item However, as an object of $\cat{GKrn}$, it is isomorphic to the one-point space (which is standard Borel), and hence it is an object of $\cat{EKrn}$.
    \end{itemize}
    Similar examples are given by non-countably generated sub-sigma-algebras of standard Borel spaces, which do not give rise to standard Borel spaces (since they are not countably generated), but are still in $\cat{EKrn}$ thanks to the results of \cite[Section~3]{ensarguet2023ergodic}. (See also \Cref{sec_condexp} of the current paper.)
\end{remark}

In $\cat{Krn}$ (and $\cat{EKrn}$, as we will see), regular conditionals always exist. Therefore, for those categories, we can equivalently take as morphisms $(X,\mathcal{A},p)\to (Y,\mathcal{B},q)$ probability measures on $X\times Y$ with marginals $p$ and $q$, i.e.~the so-called \emph{couplings} or \emph{transport plans}. 
Given a coupling $c$, we obtain a Markov kernel $X\to Y$ (unique up to $p$-almost-sure equality) by conditioning.
Conversely, given a Markov kernel $k:X\to Y$, we get a coupling by defining
\[
c(A\times B) = \int_A k(B|x) \; p(dx) .
\]

It was shown in \cite[Corollary~3.18]{ensarguet2023ergodic} that given a standard Borel probability space $(X,\mathcal{A},p)$ and any sub-sigma-algebra $\mathcal{B}\subseteq\mathcal{A}$, the probability space $(X,\mathcal{B},p)$ (where we denote the restriction of $p$ to $\mathcal{B}$ again by $p$) is essentially standard Borel.

Every measurable function defines a particular ``deterministic'' kernel as follows. (We will define determinism precisely in \Cref{defdet}.)

\begin{definition}
    Let $f:(X,\mathcal{A},p)\to (Y,\mathcal{B},p)$ be a measurable, measure-preserving function (i.e.\ for all $B\in\mathcal{B}$, $f^{-1}(B)\in\mathcal{A}$ and $p(f^{-1}(B))=q(B)$).
    The function $f$ defines a kernel $(X,\mathcal{A},p)\to (Y,\mathcal{B},q)$ as follows.
    \[
    \delta_f(B|x) = 1_B(x) = \begin{cases}
        1 & f(x)\in B ; \\
        0 & f(x)\notin B .
    \end{cases}    
    \]    
\end{definition}

If we denote by $\cat{Prob}$ the category of probability spaces and measure-preserving maps, we have an identity-on-objects functor $\delta:\cat{Prob}\to\cat{GKrn}$.

In particular, given a probability space $(X,\mathcal{A},p)$ and a sub-sigma-algebra $\mathcal{B}\subseteq\mathcal{A}$, the set-theoretic identity defines a kernel $\delta_\id:(X,\mathcal{A},p)\to(X,\mathcal{B},p)$ given by 
\begin{equation}\label{defdeltaf}
\delta_\id(B|x) = 1_B(x) 
\end{equation}
for all $B\in\mathcal{B}$ and $x\in X$. Note that since $\mathcal{B}\subseteq\mathcal{A}$, this assignment is $\mathcal{B}$-measurable.

\subsection{Functors of random variables}
\label{sec_RV}

Random variables, up to almost sure equality, are functorial on $\cat{EKrn}$.
Covariant functors of random variables were introduced in \cite{adachi2018} and independently \cite[Section~4]{dahlqvist2018borel} and \cite{vanbelle2023martingales}. 
Here we use a contravariant version, first defined in \cite[Section~4]{dahlqvist2018borel} in the Borel case, which can be turned covariant by means of the dagger structure of $\cat{EKrn}$.

As usual, a \newterm{random variable} (\newterm{RV}) on a probability space $(X,\mathcal{A},p)$ is an $\mathcal{A}$-measurable function $f:X\to\R$. 
We call $f$
\begin{itemize}
    \item \newterm{integrable} if and only if 
    \[
    \int_X |f(x)|\,p(dx) < \infty ;
    \]
    \item \newterm{$n$-integrable} if and only if 
    \[
    \int_X |f(x)|^n\,p(dx) < \infty ;
    \]
    \item \newterm{almost surely bounded} if and only if there exists $A\in\mathcal{A}$ of $p$-measure one such that $f$ restricted to $A$ is bounded.    
\end{itemize}

Unless otherwise stated, we will not consider non-integrable RVs in the present work. 

Again as usual, we say that two RVs $f$ and $g$ on $(X,\mathcal{A},p)$ are \newterm{almost surely equal} if there exists $A\in\mathcal{A}$ of $p$-measure one such that the restrictions of $f$ and $g$ to $A$ are equal. 
We denote by 
\begin{itemize}
    \item $L^1(X,\mathcal{A},p)$ the set of integrable RVs on $(X,\mathcal{A},p)$ quotiented under almost sure equality;
    \item $L^n(X,\mathcal{A},p)$ the set of $n$-integrable RVs on $(X,\mathcal{A},p)$ quotiented under almost sure equality;\footnote{Here we use $L^n$ rather than $L^p$ in order to avoid confusion with the probability measure $p$.}
    \item $L^\infty(X,\mathcal{A},p)$ the set of almost surely bounded RVs on $(X,\mathcal{A},p)$ quotiented under almost sure equality.
\end{itemize}
Since $p$ is a probability measure, every almost surely bounded RV on $(X,\mathcal{A},p)$ is $n$-integrable for every $n$. Also, for $m\ge n\ge 1$, every $m$-integrable function is $n$-integrable: for $x\ge 0$, the function $x\mapsto x^{m/n}$ is convex, and so by Jensen's inequality,
\[
\left( \int_X |f(x)|^n\,p(dx) \right)^{m/n} \;\le\;\int_X |f(x)|^m\,p(dx) ,
\]
so that 
\begin{equation}
    \| f \|_{L^n} \le \| f \|_{L^m} .
\end{equation}
In particular, if the right-hand side is finite, so is the left-hand side.
So, for each probability space, $L^\infty \subseteq\dots \subseteq L^n \subseteq L^{n-1}\subseteq\dots\subseteq L^1$.

The spaces $L^n$ of random variables form contravariant functors on $\cat{GKrn}$. 
Denote by $\cat{Ban}$ the category of Banach spaces and bounded linear maps. 
Equip $L^n(X,\mathcal{A},p)$ and $L^\infty(X,\mathcal{A},p)$ with their usual norms,
\[
\|f\|_{L^n} \coloneqq \left(\int_X |f(x)|^n\,p(dx)\right)^{1/n} \qquad\mbox{and}\qquad \|f\|_{L^\infty} \coloneqq \ess\sup |f| .
\]
We have functors $L^n:\cat{GKrn}^\op\to\cat{Ban}$, one for each $n$, which work as follows. 
\begin{itemize}
    \item On objects, we map a probability space $(X,\mathcal{A},p)$ to the Banach space $L^n(X,\mathcal{A},p)$;
    \item On morphisms, given a measure-preserving kernel $k:(X,\mathcal{A},p)\to(Y,\mathcal{B},q)$, we get a bounded linear map $k^*:L^n(Y,\mathcal{B},q)\to L^n(X,\mathcal{A},p)$ acting on random variables $g\in L^n (Y,\mathcal{B},q)$ by
    \begin{equation}\label{defkstar}
    k^*g (x) \coloneqq \int_Y g(y)\,k(dy|x) .
    \end{equation}
\end{itemize}

Here is the precise statement.

\begin{proposition}\label{RV}
    Let $k:(X,\mathcal{A},p)\to(Y,\mathcal{B},q)$ be a measure-preserving kernel, and let $g\in L^n(Y,\mathcal{B},q)$, for $n$ finite or infinite. Then the assignment 
    \[
    x\longmapsto k^*g (x) \coloneqq \int_Y g(y)\,k(dy|x) 
    \]
    is a well-defined element of $L^n(X,\mathcal{A},p)$. 

    Moreover, the assignment $k^*:L^n(Y,\mathcal{B},q)\to L^n(X,\mathcal{A},p)$ given by $g\mapsto k^*g$ is linear and 1-Lipschitz, so that $L^n$ is a functor $\cat{GKrn}^\op\to\cat{Ban}$.
\end{proposition}

\begin{proof}[Proof for finite $n$]
Since $k$ is measure-preserving, by approximating via simple functions,
\begin{align*}
    \int_X \int_Y |g(y)|^n \,k(dy|x)\,p(dx) \;=\; \int_Y |g(y)|^n\,q(dy) \;<\;\infty ,
\end{align*}
and so the integrals
\[
\int_Y |g(y)|^n \,k(dy|x) 
\]
are finite for $p$-almost all $x\in X$. Even if $n>1$, again by Jensen's inequality we have that 
\[
\left( \int_Y |g(y)| \,k(dy|x) \right)^n \;\le\; \int_Y |g(y)|^n \,k(dy|x) \;<\; \infty .
\]
Therefore the integral
\[
x\longmapsto \int_Y g(y) \,k(dy|x) ,
\]
is defined for $p$-almost all $x\in X$, and the assignment specifies a unique element of $L^n(X,\mathcal{A},p)$ (measurability can be shown by means of the usual approximation via simple functions). 
We denote this element by $k^*g$.
In other words, given $k:(X,\mathcal{A},p)\to(Y,\mathcal{B},q)$, we get a well-defined function $k^*:L^n(Y,\mathcal{B},q)\to L^n(X,\mathcal{A},p)$. 
given by $g\mapsto k^*g$.  This function is linear by linearity of integration.
To see that it is 1-Lipschitz, once again by Jensen,\begin{align*}
(\| k^*g \|_{L^n})^n \;&=\; \int_X \left| \int_Y g(y) \,k(dy|x) \right|^n \, p(dx) \\
&\le\; \int_X \left( \int_Y |g(y)| \,k(dy|x) \right)^n \, p(dx) \\
&\le\; \int_X \int_Y |g(y)|^n \,k(dy|x)\,p(dx) \\
&=\; \int_Y |g(y)|^n\,q(dy) \\
&=\; (\|g\|_{L^n})^n . \qedhere
\end{align*}
\end{proof}

Before looking at the case of $n=\infty$, let's prove the following auxiliary statement, which can be consider a sort of ``unitality'' property of the map $k^*$. 

\begin{lemma}\label{preserveone}
    Let $k:(X,\mathcal{A},p)\to(Y,\mathcal{B},q)$ be a measure-preserving kernel. 
    Then for every subset $B\in\mathcal{B}$ of $q$-measure one there exists a subset $A\in\mathcal{A}$ of $p$-measure one such that for all $x\in A$, $k(B|x)=1$.
\end{lemma}
\begin{proof}[Proof of \Cref{preserveone}]
    Consider $B\in\mathcal{B}$ of measure one, and denote its complement by $\bar{B}$. Then since $k$ is measure-preserving,
    \[
    0 \;=\; q(\bar{B}) \;=\; \int_Y k(\bar{B}|x)\,p(dx) ,
    \]
    which means that for $p$-almost all $x\in X$, $k(\bar{B}|x)=0$, i.e.~$k(B|x)=1$. 
\end{proof}

\begin{proof}[Proof of \Cref{RV} for infinite $n$]
Let $g\in L^\infty(Y,\mathcal{B},q)$. By definition there exists $B\in\mathcal{B}$ of $q$-measure one such that $g$ restricted to $B$ is bounded. Let $K$ be an upper bound.
Since $k$ is measure-preserving, just as for finite $n$,
\begin{align*}
    \int_X \int_Y |g(y)| \,k(dy|x)\,p(dx) \;=\; \int_Y |g(y)|\,q(dy) \;=\; \int_B |g(y)|\,q(dy) \;\le\; K .
\end{align*}
Therefore the integrals
\[
\int_Y |g(y)| \,k(dy|x) 
\]
are finite for $p$-almost all $x\in X$, and so the integral
\[
x\longmapsto \int_Y g(y) \,k(dy|x) ,
\]
is defined for $p$-almost all $x\in X$. Denote this value by $k^*g(x)$.

Moreover, by \Cref{preserveone}, there exists $A\in\mathcal{A}$ of $p$-measure one such that for all $x\in A$, $k(B|x)=1$.
Therefore, for all $x\in A$,
\[
|k^*g(x)| \;=\; \left| \int_Y g(y) \,k(dy|x) \right| \;\le\; \int_Y |g(y)| \,k(dy|x) \;=\;\int_B |g(y)| \,k(dy|x) \;\le\; K .
\]
In other words, the restriction of $k^*g$ to $A$ is also bounded by $K$. So we get a well-defined function $k^*:L^\infty(Y,\mathcal{B},q)\to L^\infty(X,\mathcal{A},p)$. Now since every essential upper bound for $|g|$ is an essential upper bound for $|k^*g|$, we have that 
\[
\|k^*g\|_{L^\infty} \;=\; \mathrm{ess}\sup |k^*g| \;\le\; \mathrm{ess}\sup |g| \;=\; \|g\|_{L^\infty} . \qedhere
\]
\end{proof}

Moreover, the inclusions $L^n(X,\mathcal{A},p)\subseteq L^m(X,\mathcal{A},p)$ for $m\le n$ form a natural transformation $L^n\Rightarrow L^m$.

Let's now see two particular cases of this functorial action.
\begin{example}
    A special case of a Markov kernel is a probability measure, which can be seen as a Markov kernel from the one-point space $1$. Denote by $(1,\mathcal{N},\delta_u)$ be the unique probability space on the one-point set $1$, where $u$ is the unique element of $1$. Given a probability space $(X,\mathcal{A},p)$, the only measure-preserving kernel $(1,\mathcal{N},\delta_u)\to (X,\mathcal{A},p)$ is given by $\tilde{p}$, defined in terms of $p$ as follows: for every $A\in\mathcal{A}$,
    \[
    \tilde{p}(A|u) = p(A) .
    \]
    Now given any $g\in L^1(X,\mathcal{A},p)$ (or $L^n$), we have that $L^1(1,\mathcal{N},\delta_u)\cong\R$, and
    \[
    \tilde{p}^*g (u) \;=\; \int_X g(x) \, \tilde{p}(dx|u) \;=\; \int_X g(x) \,p(dx) \;=\;\E[g] .
    \]
    In other words, for those kernels which are probability measures, the functorial action of $L^1$ (and $L^n$) is exactly giving the expectation values of random variables. 
\end{example}

\begin{example}
    Let's now consider the case of ``deterministic'' Markov kernels defined by functions. Let $f:(X,\mathcal{A},p)\to(Y,\mathcal{B},q)$ be a measure-preserving function, and consider the induced kernel $\delta_f$ as in \eqref{defdeltaf}. 
    Given a random variable $g\in L^1(Y,\mathcal{B},q)$ (or $L^n$), we have for every $x\in X$,
    \[
    (\delta_f)^*g(x) \;=\; \int_Y g(y) \,\delta_f(dy|x) \;=\; g(f(x)) .
    \]
    In other words, $(\delta_f)^*g=g\circ f$. That is, for the kernels in the form $\delta_f$ for some measurable function $f$, the functorial action of $L^1$ (and $L^n$) is exactly precomposition of functions.
    We can view $f:(X,\mathcal{A},p)\to(Y,\mathcal{B},q)$ as a ``reparametrization'' or ``refinement'' of the outcome space on which our RVs are defined: any RV on $Y$ defines a RV on $X$, but in general the sigma-algebra $\mathcal{A}$ may be finer than (the one induced by) $\mathcal{B}$.
\end{example}

In other words, the functorial action of $L^1$ (and $L^n$) on kernels is a common generalization of expectations and of refinements of the outcome space.

\begin{proposition}
    The functors $L^n:\cat{GKrn}^\op\to\cat{Ban}$ are faithful. 
    In other words, for measure-preserving kernels $k,h:(X,\mathcal{A},p)\to(Y,\mathcal{B},q)$ we have that $k=h$ $p$-almost surely if and only if for every random variable $g\in L^n(Y,\mathcal{B},q)$, $k^*g=h^*g$ $p$-almost surely.
\end{proposition}
\begin{proof}
    It suffices to test the a.s.\ equality of $k$ and $h$ on the indicator functions of sets of $\mathcal{B}$, which are in $L^n(Y,\mathcal{B},q)$ for all $n$ (including $n=\infty$).
\end{proof}

Before we leave this section, let's also look at how to talk about \emph{single} random variables categorically, as opposed to spaces of random variables. 
To this end, notice that a point of a Banach space $B$ is encoded exactly by a (bounded) linear map $f:\R\to B$ (by looking at where $f$ maps $1\in\R$).
Therefore an $L^n$ random variable on $(X,\mathcal{A},p)$ can be described categorically as a morphism $f:\R\to L^n(X,\mathcal{A},p)$ of $\cat{Ban}$.
Its expectation is the composition
\[
\begin{tikzcd}[row sep=0]
    \R \ar{r}{f} & L^n(X,\mathcal{A},p) \ar{r}{\tilde{p}^*} & L^n(1,\mathcal{N},\delta_u) \cong \R .
\end{tikzcd}
\]

\subsection{Bayesian inversions, conditioning, and dagger structures}
\label{sec_cond}

The categories $\cat{Krn}$ and $\cat{EKrn}$ come equipped with a dagger structure which is very relevant for probability, given by Bayesian inverses. Here is the definition.

\begin{definition}
    Let $k:(X,\mathcal{A},p)\to(Y,\mathcal{B},q)$ be a measure-preserving kernel. A \newterm{Bayesian inverse} of $k$ is a measure-preserving kernel $k^+:(Y,\mathcal{B},q)\to(X,\mathcal{B},p)$ such that for every $A\in\mathcal{A}$ and $B\in\mathcal{B}$,
    \begin{equation}\label{defbinv}
        \int_A k(B|x)\,p(dx) \;=\; \int_B k^+(A|y)\,q(dy) . 
    \end{equation}
\end{definition}

When a Bayesian inverse exists, it is unique almost surely. 
It is easy to see that the identity kernel is its own Bayesian inverse, and that Bayesian inverses are closed under composition. 
Moreover, every invertible kernel (invertible in $\cat{GKrn}$) has as inverse its Bayesian inverse. 

In what follows it is useful to set some notation for conditional expectations. 
Given a probability space $(X,\mathcal{A},p)$, a sub-sigma-algebra $\mathcal{B}\subseteq\mathcal{A}$ and an integrable RV $f\in L^1(X,\mathcal{A},p)$, recall that a \newterm{conditional expectation of $f$ given $\mathcal{B}$} is a RV $g\in L^1(X,\mathcal{B},p)$ such that for every $B\in\mathcal{B}$,
\[
\int_B g(x)\,p(dx) = \int_B f(x)\,p(dx) .
\]
Such a conditional expectation, if it exists, is unique almost surely. 
We denote it by
\[
x\longmapsto \E[f|\mathcal{B}](x) .
\]
We also write, for $A\in\mathcal{A}$, the shorthand
\[
\P[A|\mathcal{B}](x) \coloneqq \E[1_A|\mathcal{B}](x) ,
\]
where $1_A$ is the indicator function. 

We will use the following classic measure-theoretic result, which we restate in our notation:

\begin{theorem}[Rokhlin's disintegration theorem]\label{rokhlin}
    Let $(X,\mathcal{A},p)$ be a standard Borel probability space. 
    Consider a sub-sigma algebra $\mathcal{B}\subseteq\mathcal{A}$, and the kernel $\pi\coloneqq\delta_{\id}:(X,\mathcal{A},p)\to(X,\mathcal{B},p)$ induced by the set-theoretical identity of $X$ via \eqref{defdeltaf}.
    Then $\pi$ has a Bayesian inverse (or \newterm{disintegration}) $\pi^+:(X,\mathcal{B},p)\to(X,\mathcal{A},p)$ such that for all $A\in\mathcal{A}$, for $p$-almost all $x\in X$,
    \[
    \pi^+(A|x) = \P[A|\mathcal{B}](x) .
    \]
\end{theorem}
This is Theorem 4 in \cite{Tao}. 

On random variables, the functorial action of kernel $\pi^+:(X,\mathcal{B},p)\to(X,\mathcal{A},p)$ takes an $\mathcal{A}$-measurable random variable $g$ and returns its conditional expectation with respect to $\mathcal{B}$:
\begin{equation}\label{actioncexp}
    (\pi^+)^*g \;=\; ((\delta_\id)^+)^*g \;=\; \E[g|\mathcal{B}] .
\end{equation}

Here is a well known, important consequence of the theorem:

\begin{proposition}\label{bayesexists}
    Every measure-preserving kernel between standard Borel spaces admits a Bayesian inverse.
\end{proposition}

\begin{proof}
    Let $(X,\mathcal{A})$ and $(Y,\mathcal{B})$ be standard Borel spaces, and let $k:(X,\mathcal{A},p)\to(Y,\mathcal{B},q)$ be a measure-preserving kernel. Form the joint probability measure $r$ on $(X\times Y,\mathcal{A}\otimes\mathcal{B})$ specified on rectangles by
    \[
    r(A\times B) \coloneqq \int_A k(B|x)\,p(dx) .
    \]
    Let's show that the projection to the second marginal $\pi_2:(X\times Y,\mathcal{A}\otimes\mathcal{B},r)\to (Y,\mathcal{B},q)$ has a Bayesian inverse.
    By \cite[Proposition~2.7]{ensarguet2023ergodic}, we can rewrite $(Y,\mathcal{B},q)$, up to isomorphism of $\cat{GKrn}$, as the set $X\times Y$ equipped with the sigma-algebra $\pi_2^{-1}(\mathcal{B})$ induced by the projection $\pi_2$, and the measure $r$ restricted to $\pi_2^{-1}(\mathcal{B})$. Therefore $\pi_2$ can be equivalently written as a deterministic kernel $(X\times Y, \mathcal{A}\otimes\mathcal{B},r)\to(X\times Y, \pi_2^{-1}(\mathcal{B}),r)\cong (Y,\mathcal{B},q)$.
    By \Cref{rokhlin}, we then have a Bayesian inverse $\pi_2^+:(Y,\mathcal{B},q)\cong(X\times Y, \pi_2^{-1}(\mathcal{B}),r)\to (X\times Y, \mathcal{A}\otimes\mathcal{B},r)$ as follows,
    \[
    \pi_2^+ (A\times B| y) = \E\big(1_A\cdot1_B | \pi_2^{-1}(\mathcal{B}) \big)(x,y) 
    \]  
    for $A\in\mathcal{A}$ and $B\in\mathcal{B}$, and for $x\in X$ and $y\in Y$ (the quantity, almost surely, does not depend on $x$). 
    Composing it with the projection $\pi_1:X\times Y\to X$ we obtain a kernel $h:Y\to X$ given by 
    \[
    h(A|y) = \E\big(1_A| \pi_2^{-1}(\mathcal{B}) \big)(x,y) ,
    \]
    which again, almost surely, does not depend on $x$. 
    So now for $A\in\mathcal{A}$ and $B\in\mathcal{B}$,
    \begin{align*}
    \int_B h(A|y)\,q(dy) &= \int_{\pi_2^{-1}(B)} \E\big(1_A| \pi_2^{-1}(\mathcal{B}) \big)(x,y) \, r(dx\,dy) \\
    &= \int_{\pi_2^{-1}(B)} 1_A(x) \, r(dx\,dy) \\
    &= r(A\times B) \\
    &= \int_A k(B|x)\,p(dx) ,
    \end{align*}
    so that $h$ is a Bayesian inverse of $k$.
\end{proof}

Bayesian inversion is therefore defined on all morphisms, and it makes $\cat{Krn}$ a dagger category (see also \cite[Theorem~2.10]{dahlqvist2018borel}). Since isomorphisms of $\cat{GKrn}$ admit Bayesian inverses as well, this also implies that every kernel between \emph{essentially} standard Borel spaces admits a Bayesian inverse as well, and so $\cat{EKrn}$ is a dagger category as well. 

\begin{remark}
Since $\cat{Krn}$ and $\cat{EKrn}$ are dagger categories, we can also define a \emph{covariant} version of the random variable functors, $\cat{Krn}\to\cat{Ban}$, by precomposing with the dagger:
\[
\begin{tikzcd}
    \cat{Krn} \ar{r}{(-)^+} & \cat{Krn}^\op \ar{r}{L^n} & \cat{Ban}
\end{tikzcd}
\]
We can restrict this functor to the deterministic kernels in the form $\delta_f$. If we denote by $\cat{BorelMeas}$ the category of standard Borel probability spaces and measure-preserving functions, the resulting functor $\cat{BorelMeas}\to\cat{Ban}$ acts as follows. 
First of all, given a standard Borel space $(X,\mathcal{A},p)$ and a sub-sigma-algebra $\mathcal{B}\subseteq\mathcal{A}$, the map $L^n(X,\mathcal{A},p)\to L^n(X,\mathcal{B},A)$ takes an $\mathcal{A}$-measurable RV and gives its conditional expectation, as in \eqref{actioncexp}.
More generally, given a measure-preserving function $f:(X,\mathcal{A},p)\to(Y,\mathcal{B},q)$, the map $L^n(X,\mathcal{A},p)\to L^n(Y,\mathcal{B},q)$ takes a RV $g$ on $X$ and gives the following RV on $Y$:
\[
y \longmapsto \E[g|f^{-1}(\mathcal{B})](x) \qquad\mbox{for } x\in f^{-1}(y) ,
\]
where $f^{-1}(\mathcal{B})\subseteq\mathcal{A}$ is the pullback sigma-algebra induced by $f$, and where the quantity above does not depend on the choice of $x\in f^{-1}(y)$.
This functor was defined in \cite{adachi2018} with the name $\mathcal{E}$ and in \cite{vanbelle2023martingales} with the name $RV$.
\end{remark}

Let's now focus on the functor $L^2$.
Given a probability space $(X,\mathcal{A},p)$, the space $L^2(X,p)$ is a Hilbert space, with inner product given by 
\[
\langle f, g \rangle \;=\; \int_X f(x)\,g(x)\,p(dx) .
\]

Therefore we can consider $L^2$ as a functor $\cat{GKrn}\to\cat{Hilb}$.

\begin{proposition}
    The functor $L^2:\cat{EKrn}\to\cat{Hilb}$ is a dagger functor, meaning that the following diagram commutes.
    \[
    \begin{tikzcd}
        \cat{EKrn}^\op \ar{d}{(-)^+} \ar{r}{L^2} & \cat{Hilb} \ar{d}{(-)^+} \\
        \cat{EKrn}^\op \ar{r}{L^2} & \cat{Hilb}
    \end{tikzcd}
    \]
\end{proposition}
\begin{proof}
    On objects, the assert is trivial. 
    On morphisms, consider a measure-preserving kernel $k:(X,\mathcal{A},p)\to (Y,\mathcal{B},q)$. 
    The commutativity of the diagram says that $(k^*)^+=(k^+)^*$
    In other words, given random variables $f\in L^2(X,\mathcal{A},p)$ and $g\in L^2(Y,\mathcal{B},q)$,
    \[
    \langle f, k^* g \rangle  \;=\;  \langle (k^+)^*f, g \rangle .
    \]
    In terms of integrals, equivalently,
    \[
    \int_X \int_Y f(x)\,g(y)\, k(dy|x)\,p(dx) \;=\; \int_X \int_Y f(x)\,g(y)\, k^+(dx|y)\,q(dy) . 
    \]
    This now follows from \eqref{defbinv} by approximating $f$ and $g$ via simple functions.
\end{proof}

One of the first probabilistic concepts that we can express in terms of the dagger structure is (almost sure) determinism.
Let's start with a definition, which follows \cite[Section~10 and Section~13]{fritz2019synthetic}.

\begin{definition}\label{defdet}
    A measure-preserving kernel $(X,\mathcal{A},p)\to (Y,\mathcal{B},q)$ is called \newterm{deterministic} if for every $x\in X$ and for every $B\in\mathcal{B}$,
    \[
    k(B|x) = 0 \quad\mbox{or}\quad k(B|x) = 1 .
    \]
    
    The kernel $k$ is called \newterm{$p$-almost surely deterministic} if for every $B\in\mathcal{B}$, the set of $x\in X$ where $k(B|x)\in\{0,1\}$ has $p$-measure one.\footnote{Note that this may be more general than requiring that there exists a measure-one subset independent of $B$ on which $k(B|x)\in\{0,1\}$. If $\mathcal{B}$ is countably generated, for example if $(Y,\mathcal{B})$ is standard Borel, the two notions coincide.}
\end{definition}

Every kernel in the form $\delta_f$ for a measurable function $f$ is deterministic. The converse in general is not true, since the sigma-algebra on the codomain may fail to separate points. However, every kernel between standard Borel spaces is deterministic if and only if it is in the form $\delta_f$~\cite[Example~10.5]{fritz2019synthetic}.

Almost sure determinism, in $\cat{Krn}$ and $\cat{EKrn}$, can be expressed in terms of the dagger structure as follows:

\begin{proposition}[{\cite[Proposition~2.5]{ensarguet2023ergodic}}]\label{asdet}
    A Markov kernel $k$ admitting a Bayesian inverse $k^+$ is almost surely deterministic if and only if $k\circ k^+=\id$.
    In particular, a morphism of $\cat{EKrn}$ is almost surely deterministic if and only if it is a dagger epi.
\end{proposition}

Since every isomorphism of $\cat{EKrn}$ is almost surely deterministic, it follows that every isomorphism of $\cat{EKrn}$ is a dagger iso.

\section{Idempotent kernels and sub-sigma-algebras}
\label{sec_idempsigma}

We now study (dagger) idempotents in the category $\cat{EKrn}$. 
We will show that these correspond to sub-sigma-algebras up to almost sure equality, and that their functorial action on random variables will be given by conditional expectation.
Similarly, the functorial action of a filtration of sub-sigma-algebras will give martingales.

\subsection{Conditional expectation operators}
\label{sec_condexp}

Let's look at the idempotent morphisms of $\cat{Krn}$ and $\cat{EKrn}$. Since everything is defined up to almost sure equality, from the measure-theoretic point of view the kernels we are interested in are \emph{almost surely} idempotent, in the following sense.

\begin{definition}
A measure-preserving kernel $e:(X,\mathcal{A},p)\to(X,\mathcal{A},p)$ is \newterm{$p$-almost surely idempotent} if and only if for every $B\in\mathcal{A}$ and for $p$-almost all $x\in X$, we have
\[
\int_X e(B|x')\,e(dx'|x) = e(B|x) .
\]
\end{definition}
Equivalently, if for all $A,B\in\mathcal{A}$, 
\[
\int_A\int_X e(B|x')\,e(dx'|x)\,p(dx) = \int_A e(B|x)\,p(dx) .
\]
The almost sure equivalence classes of these kernels, between standard Borel probability spaces, are precisely the idempotent morphisms of $\cat{GKrn}$.
In terms of random variables, equivalently, a kernel $e$ is a.s.\ idempotent if and only if for every RV $f$, $e^*e^*f=e^*f$ almost surely. That is, if the corresponding operator $e^*$ on $L^n(X,\mathcal{A},p)$ is idempotent. 

Here is our main example of idempotent kernel. As we will show, all idempotent kernels can be written in this form.

\begin{definition}\label{subtokern}
    Let $(X,\mathcal{A},p)$ be a standard Borel probability space, and consider a sub-sigma-algebra $\mathcal{B}\subseteq\mathcal{A}$.
    The \newterm{conditional expectation operator} is the equivalence class of Markov kernels $e_{\mathcal{B}}:(X,\mathcal{A},p)\to(X,\mathcal{A},p)$ given by
    \[
    e_{\mathcal{B}}(A|x) \coloneqq \P[A|\mathcal{B}](x) = \E[1_A|\mathcal{B}](x) ,
    \]
    for each $A\in\mathcal{A}$ and for $p$-almost all $x\in X$.
\end{definition}

Note that by the disintegration theorem (\Cref{rokhlin}) we know that these conditional expectations can indeed be assembled to a kernel $(X,\mathcal{B},p)\to(X,\mathcal{A},p)$. Since $\mathcal{B}\subseteq\mathcal{A}$, this assignment is also measurable as a kernel $(X,\mathcal{A},p)\to(X,\mathcal{A},p)$.

The action of $e_\mathcal{B}$ on random variables, almost surely, is conditional expectation (hence the name):
\begin{align*}
(e_\mathcal{B})^*f (x) \;=\; \int_X f(x')\,e_\mathcal{B}(dx'|x) \;=\; \E[f|\mathcal{B}](x) .
\end{align*}
By idempotency of conditional expectations, we then have that $(e_\mathcal{B})^*:L^n(X,\mathcal{A},p)\to L^n(X,\mathcal{A},p)$ is idempotent, and hence so is $e_\mathcal{B}$. 

The idempotent $e_\mathcal{B}$ is split by $\mathcal{B}$:

\begin{lemma}\label{Bsplits}
    Let $(X,\mathcal{A},p)$ be a standard Borel probability space, and let $\mathcal{B}\subseteq\mathcal{A}$ be a sub-sigma-algebra.
    A splitting of the idempotent $e_\mathcal{B}:(X,\mathcal{A},p)\to(X,\mathcal{A},p)$ in $\cat{GKrn}$ is given by $(X,\mathcal{B},p)$, together with the map $\pi=\delta_\id:(X,\mathcal{A},p)\to(X,\mathcal{B},p)$ induced by the set-theoretic identity, and its Bayesian inverse $\pi^+=(\delta_\id)^+:(X,\mathcal{B},p)\to(X,\mathcal{A},p)$ (which exists by \Cref{rokhlin}).
\end{lemma}
\begin{proof}
    First of all, we have that $\pi\circ\pi^+=\id_{(X,\mathcal{B},p)}$ almost surely, since $\pi$ is deterministic (using \Cref{asdet}).
    To show that $\pi^+\circ\pi=e_\mathcal{B}$ almost surely, notice that for every $A\in\mathcal{A}$ and for almost all $x\in X$, using \Cref{rokhlin},
    \begin{align*}
        (\pi^+\circ\pi)(A|x) \;=\; \int_X \pi^+(A|x')\,\pi(dx'|x) \;=\; \int_X \P[A|\mathcal{B}](x')\,\pi(dx'|x) \;=\; \P[A|\mathcal{B}](x) \;=\; e_\mathcal{B}(A|x) . 
    \end{align*} 
\end{proof}

In \cite[Theorem~3.14]{ensarguet2023ergodic} it was proven that all idempotents of $\cat{Krn}$ split, and since $\cat{Krn}$ is a full subcategory of $\cat{GKrn}$, the splitting in $\cat{GKrn}$ is going to be isomorphic to the one in $\cat{Krn}$, i.e., it is in $\cat{EKrn}$.
In this work we prefer to use sub-sigma-algebras, which are technically not standard Borel (only essentially so), but easier and more explicit to work with. This is why we are interested in the category $\cat{EKrn}$.

It turns out that all idempotent morphisms of $\cat{Krn}$ can be written in the form $e_\mathcal{B}$ for some sub-sigma-algebra $\mathcal{B}$, and that this gives a splitting of the idempotent (in $\cat{EKrn}$).
This sub-sigma-algebra is the \emph{invariant sigma-algebra}, defined here:

\begin{definition}\label{definv}
    Let $k:(X,\mathcal{A},p)\to(X,\mathcal{A},p)$ be a measure-preserving kernel.
    A measurable subset $B\in\mathcal{A}$ is called \newterm{$p$-almost surely invariant} under $k$ if and only if for $p$-almost all $x\in X$, we have 
    \[
    k(B|x) = 1_B(x) .
    \]
\end{definition}
Equivalently, if for all $A\in\mathcal{A}$,
\[
\int_A k(B|x) \,p(dx) = \int_A 1_B(x)\,p(dx) = p(A\cap B) .
\]
Almost surely invariant sets form a sub-sigma-algebra of $\mathcal{A}$, which we denote by $\mathcal{I}_k$, and call the \newterm{invariant sigma-algebra}. 

\begin{remark}\label{complete}
    For every measure-preserving kernel $k:(X,\mathcal{A},p)\to(X,\mathcal{A},p)$, the invariant sigma-algebra $\mathcal{I}_k$ contains all the null sets.
    Indeed, suppose that for $B\in\mathcal{A}$, $p(B)=0$. Then for all $A\in\mathcal{A}$,
    \[
    \int_A k(B|x) \,p(dx) \le \int_X k(B|x) \,p(dx) = p(B) = 0 = p(A\cap B) .
    \]
    In particular, if $(X,\mathcal{A},p)$ is a complete measure space, so is $(X,\mathcal{I}_k,p)$.
    Similarly, $\mathcal{I}_k$ also contains all the full-measure sets. 
\end{remark}

Invariant sigma-algebras split idempotents as follows.

\begin{theorem}\label{krnsplit}
    Let $e:(X,\mathcal{A},p)\to(X,\mathcal{A},p)$ be an idempotent morphism of $\cat{Krn}$. 
    Consider the invariant sigma-algebra $\mathcal{I}_e$. Form the kernel $\pi:(X,\mathcal{A},p)\to (X,\mathcal{I}_e,p)$ induced by the set-theoretic identity, as in \Cref{rokhlin}, and consider its Bayesian inverse $\pi^+:(X,\mathcal{I}_e,p)\to (X,\mathcal{A},p)$ (which exists by \Cref{rokhlin}).
    Then $\big((X,\mathcal{I}_e,p),\pi^+,\pi\big)$ is a splitting of $e$ in $\cat{GKrn}$. 
\end{theorem}

We will use the following auxiliary statements. 

\begin{lemma}[{\cite[Lemmas 3.7 and 3.9]{ensarguet2023ergodic}}]\label{harm}
    Let $k:(X,\mathcal{A},p)\to(X,\mathcal{A},p)$ be a measure-preserving kernel. 
    A bounded $\mathcal{A}$-measurable function $f:X\to\R$ is $\mathcal{I}_k$-measurable if and only if for $p$-almost all $x$,
    \[
    k^*f(x) = f(x) .
    \]
    (Such functions are sometimes called \emph{harmonic}.)
\end{lemma}

\begin{lemma}\label{pos}
    Let $k:(X,\mathcal{A},p)\to(X,\mathcal{A},p)$ be a measure-preserving kernel, and let $A\in\mathcal{I}_k$.
    Then for every $B\in\mathcal{A}$, for $p$-almost all $x\in X$,
    \[
    k(A|x)\,k(B|x) = k(A\cap B|x) .
    \]
\end{lemma}

This lemma can be considered as an instance of \emph{relative positivity in a Markov category}, see \cite[Example~13.19]{fritz2019synthetic} and \cite[Section~2.5]{fritz2022dilations}.

\begin{proof}[Proof of \Cref{pos}]
    Define the sets
    \[
    X_0 \coloneqq \{x\in X : k(A|x)=0\} ,
    \qquad
    X_1 \coloneqq \{x\in X : k(A|x)=1\} .
    \]
    Since $A\in\mathcal{I}_k$, we have that $k(A|x)=1_A(x)$ $p$-almost surely, and so $p(X_0\cup X_1)=1$. 
    Now let $B\in\mathcal{A}$. For $x\in X_0$,
    \[
    k(A|x)\,k(B|x) = 0 = k(A\cap B|x) .
    \]
    Similarly, for $x\in X_1$,
    \[
    k(A|x)\,k(B|x) = k(B|x) = k(A\cap B|x) ,
    \]
    as the intersection of measure-one sets has measure one.
\end{proof}

We are now ready to prove the theorem.

\begin{proof}[Proof of \Cref{krnsplit}]
    By \Cref{Bsplits}, it suffices to show that $e=e_{\mathcal{I}_e}$ (almost surely).
    So let $A\in\mathcal{A}$. For $p$-almost all $x\in X$,
    \[
    e_{\mathcal{I}_e}(A|x) \;=\; \P[A|\mathcal{I}_e](x) \;=\; \E[1_A|\mathcal{I}_e](x) .
    \]
    To show that this conditional expectation is almost surely equal to $e(A|x)$, it suffices to show that $e(A|x)$ is a conditional expectation in the form above. This means that:
    \begin{enumerate}
        \item\label{mble} The function $x\mapsto e(A|x)$ is $\mathcal{I}_e$-measurable;
        \item\label{onset2} For every $B\in\mathcal{I}_e$, 
        \[
        \int_B e(A|x)\,p(dx) = \int_B 1_A(x)\,p(dx) = p(A\cap B) .
        \]
    \end{enumerate}
    To prove \ref{mble}, notice that since $e$ is idempotent, the function $x\mapsto e(A|x)$ satisfies
    \[
    \int_X e(A|x')\,e(dx'|x) =  e(A|x) \qquad\mbox{for }p\mbox{-almost all }x .
    \]
    Therefore it is harmonic in the sense of \Cref{harm}, and so it is $\mathcal{I}_e$-measurable.
    To prove \ref{onset2} we use, in order, the fact that $B\in\mathcal{I}_e$, \Cref{pos}, and the fact that $e$ is measure-preserving:
    \begin{align*}
        \int_B e(A|x)\,p(dx) &= \int_X e(A|x)\,e(B|x)\,p(dx) \\
        &= \int_X e(A\cap B|x)\,p(dx) \\
        &= p(A\cap B) .
    \end{align*}
    Therefore $e(A|x)=e_{\mathcal{I}_e}(A|x)$, and by \Cref{Bsplits}, $\mathcal{I}_e$ splits $e$. 
\end{proof}

As we remarked, $(X,\mathcal{I}_e,p)$ is in $\cat{EKrn}$, which is a dagger category, and so this splitting, given by $\pi$ and its Bayesian inverse, is a dagger splitting. 

\begin{corollary}[{\cite[end of Section 3]{ensarguet2023ergodic}}]\label{selfadj}
    Every idempotent morphism of $\cat{EKrn}$ is a dagger idempotent.
\end{corollary}

This can be interpreted as a detailed balance condition: every measure-preserving, idempotent Markov kernel induces a reversible process.
Also, since we have seen that the functor $L^2:\cat{EKrn}^\op\to\cat{Hilb}$ preserves the dagger, every idempotent kernel induces a self-adjoint, idempotent operator $e^*:L^2(X,\mathcal{A},p)\to L^2(X,\mathcal{A},p)$, i.e.~an orthogonal projector.
(Conversely, a kernel induces an orthogonal projector if and only if it is idempotent.)

\begin{remark}
Since idempotent splittings are preserved by every functor, and so we also have that for every $n$, the idempotent $e^*:L^n(X,\mathcal{A},p)\to L^n(X,\mathcal{A},p)$ is split by $L^2(X,\mathcal{I}_e,p)$. Mind that the functors $L^n$ reverse the direction of the arrows, so that the idempotent operator $e^*$ factors as follows,
\[
\begin{tikzcd}[row sep=0]
    L^n(X,\mathcal{A},p) \ar{r}{(\pi^+)^*} & L^n(X,\mathcal{I}_e,p) \ar{r}{\pi^*} & L^n(X,\mathcal{A},p) \\
    f \ar[mapsto]{r} & \E[f|\mathcal{I}_e] \ar[mapsto]{r} & \E[f|\mathcal{I}_e] .
\end{tikzcd}
\]
The map $\pi^*$ is the inclusion of $\mathcal{I}_e$-measurable RVs into all $\mathcal{A}$-measurable ones, and the map $(\pi^+)^*$ is the projection of $\mathcal{A}$-measurable RVs onto $\mathcal{I}_e$-measurable ones given by taking conditional expectations.
As it is well known, for the $L^2$ case this projection is orthogonal. 
\end{remark}

\subsection{The preorder of sub-sigma-algebras relative to a measure}
\label{sec_subsigma}
The assignments $e\mapsto \mathcal{I}_e$ and $\mathcal{B}\mapsto e_\mathcal{B}$ between idempotents and sub-sigma-algebras are almost inverse to each other. We will make this now precise by showing that these form a Galois connection, which induces an order-preserving bijection between idempotents and $p$-complete sub-sigma-algebras.

\begin{proposition}
For idempotents $e\leq e'$, we have that $\mathcal{I}_e\subseteq \mathcal{I}_{e'}$, i.e. the assignment $e\mapsto \mathcal{I}_e$ is order-preserving.
\end{proposition}
\begin{proof}
    For $B\in \mathcal{I}_e$, we find for $p$-almost every $x$ in $X$ that $$e'(B\mid x)=\int 1_B(y)e'(\text{d}y\mid x)=\int e(B\mid y)e'(\text{d}y\mid x) = e(B\mid x)=1_B(x).$$
Here we used that $e(B\mid y)=1_B(y)$ for $p$-almost every $y$ in $X$ and that $e\circ e'=e$. It follows now that $B\in \mathcal{I}_{e'}$.
\end{proof}

\begin{proposition}

For sub-sigma-algebras $\mathcal{B}\subseteq \mathcal{B}'$, we have that $e_\mathcal{B}\leq e_{\mathcal{B}'}$, i.e. the assignment $\mathcal{B}\mapsto e_\mathcal{B}$ is order-preserving.
\end{proposition}
\begin{proof}
We want to show that $e_\mathcal{B}\circ e_{\mathcal{B}'}=e_\mathcal{B}$. For $A\in \mathcal{A}$ we have $p$-almost surely that 
$$e_\mathcal{B}\circ e_{\mathcal{B}'}(A\mid -) = \mathbb{E}[\mathbb{E}[1_A\mid \mathcal{B}]\mid \mathcal{B}']=\mathbb{E}[1_A\mid \mathcal{B}]=e_\mathcal{B}(A\mid -).$$
Here we used the tower property of conditional expectation.
\end{proof}
\begin{theorem}
For an idempotent $e$ and a sub-sigma-algebra $B$, we have that $$\mathcal{B}\subseteq \mathcal{I}_e \quad \text{if and only if }\quad e_\mathcal{B}\leq e,$$ i.e. the order-preserving assignment $\mathcal{B}\mapsto e_\mathcal{B}$ is left adjoint to the order-preserving assignment $e\mapsto \mathcal{I}_e$.
\end{theorem}
\begin{proof}
Suppose that $\mathcal{B}\subseteq \mathcal{I}_e$. We want to show that $e_\mathcal{B}\circ e = e_\mathcal{B}$. For $A\in \mathcal{A}$ we have $p$-almost surely that $$e_\mathcal{B}\circ e(A\mid -) = \int \mathbb{E}[1_A\mid \mathcal{B}]e(\text{d}y\mid -)=\mathbb{E}[1_A\mid \mathcal{B}]=e_\mathcal{B}(A\mid -).$$
In the second equality we use the fact that $\mathbb{E}[1_A\mid \mathcal{B}]$ is $\mathcal{I}_e$-measurable. 

Suppose now that $e_\mathcal{B}\leq e$ (i.e. $e_\mathcal{B}\circ e = e_\mathcal{B}$) and that $B\in \mathcal{B}$. Then for $p$-almost every $x$ in $X$ it holds that $$e(B\mid x)=\int 1_B(y)e(\text{d}y\mid x) = \int e_\mathcal{B}(B\mid y)e(\text{d}y\mid x)=e_\mathcal{B}(1_B\mid x)=1_B(x).$$
We can conclude that $B\in \mathcal{I}_e$, and therefore $\mathcal{B}\subseteq \mathcal{I}_e$.
\end{proof}

Every Galois connection restricts to an order-preserving bijection. In our particular Galois connection, this means that the assignments $e\mapsto \mathcal{I}_e$ and $\mathcal{B}\mapsto e_\mathcal{B}$ restrict to an order-preserving bijection between: $$\left\{e\mid e_{\mathcal{I}_e}=e\right\} \quad \text{ and }\quad \left\{\mathcal{B}\mid \mathcal{I}_{e_\mathcal{B}}=\mathcal{B}\right\}.$$
\begin{proposition}
For every idempotent $e$ it holds that $e_{\mathcal{I}_e}=e$, i.e. the preorder $\left\{e\mid e_{\mathcal{I}_e}=e\right\}$ consists of all idempotents.
\end{proposition}
\begin{proof}
 We want to show that for every $A\in\mathcal{A}$, $e(A\mid -)=\mathbb{E}[1_A\mid \mathcal{I}_e]$ $p$-almost surely. We prove this by showing that the right hand side satisfies the defining property of conditional expectation. For $B\in\mathcal{I}_e$, using Lemma 4.8 we find that $$\int e(A\mid-)1_B\text{d}p = \int e(A\cap B\mid -)\text{d}p=p(A\cap B)=\int 1_A1_B\text{d}p.$$  In the second equality we also used the fact that $e$ is measure-preserving.
\end{proof}

\begin{proposition}
The set $\{\mathcal{B}\mid \mathcal{I}_{e_\mathcal{B}}=\mathcal{B}\}$ consists precisely of the $p$-complete sub-sigma-algebras of $\mathcal{A}$.
\end{proposition}
\begin{proof}
We will first show that $\mathcal{I}_e$ is complete for any idempotent $e$. Let $A\in \mathcal{A}$ be a measurable subset such that $p(A)=0$. This means that $1_A=0$ $p$-almost surely. Since $e$ is measure-preserving, we also have that $$\int e(A\mid x)p(\text{d}x)=p(A)=0,$$ which implies that $e(A\mid -)=0$ $p$-almost surely. It follows now that $A\in \mathcal{I}_e$, and therefore $\mathcal{I}_e$ is $p$-complete. In particular, for every $\mathcal{B}$ such that $\mathcal{I}_{e_\mathcal{B}}=\mathcal{B}$, we can conclude that $\mathcal{B}$ is $p$-complete.

Let $\mathcal{B}$ be a $p$-complete sub-sigma-algebra. For $B\in \mathcal{B}$, $\mathbb{E}[1_B\mid \mathcal{B}]=1_B$ $p$-almost surely, hence $\mathcal{B}\subseteq \mathcal{I}_{e_\mathcal{B}}.$ For $A\in \mathcal{I}_{e_\mathcal{B}}$, we know that $\mathcal{E}[1_A\mid \mathcal{B}]=1_A$ $p$-almost surely. For a representative $g$ of $B:=\mathbb{E}[1_A\mid \mathcal{B}]$, the set $B$ is $\mathcal{B}$-measurable. We also have that $$p(A\Delta B)=0.$$ Because $\mathcal{B}$ is $p$-complete, this implies that $A\setminus B$ and $B\setminus A$ are also in $\mathcal{B}$. 
Because we can write $$A=(B\cup(A\setminus B))\setminus (B\setminus A),$$ we can conclude that $A\in \mathcal{B}$, which means that $\mathcal{I}_{e_\mathcal{B}}\subseteq \mathcal{B}$.
\end{proof}

Combining the above results leads to the following theorem, which states that idempotent measure-preserving kernels correspond to complete sub-sigma-algebras in an order-preserving way.

\begin{theorem}
The assignments $e\mapsto \mathcal{I}_e$ and $\mathcal{B}\mapsto e_\mathcal{B}$ define order-preserving inverse maps between the partial order of measure-preserving idempotent kernels and the partial order of complete sub-sigma-algebras.
\end{theorem}

\subsection{Filtrations and martingales}
\label{sec_martingales}

Let's now turn to martingales. We first need to introduce filtrations, which in our categorical setting correspond to filtered diagrams of dagger-split idempotents. 

\begin{definition}
    Let $(X,\mathcal{A},p)$ be a probability space, and let $(\mathcal{B}_\lambda)_{\lambda\in\Lambda}$ be a net (or sequence) of sub-sigma-algebras of $\mathcal{A}$. 
    We say that $(\mathcal{B}_\lambda)$ is 
    \begin{itemize}
        \item An \newterm{increasing filtration} if for all $\lambda\le\mu$, $\mathcal{B}_\lambda\subseteq\mathcal{B}_{\mu}$;
        \item A \newterm{decreasing filtration} if for all $\lambda\le\mu$, $\mathcal{B}_\lambda\supseteq\mathcal{B}_{\mu}$;
        \item An \newterm{almost surely increasing filtration} if for all $\lambda\le\mu$, $\mathcal{B}_\lambda\lesssim\mathcal{B}_{\mu}$;
        \item An \newterm{almost surely decreasing filtration} if for all $\lambda\le\mu$, $\mathcal{B}_\lambda\gtrsim\mathcal{B}_{\mu}$.
    \end{itemize}
\end{definition}

For sub-sigma-algebras of a standard Borel space, whenever $\mathcal{B}\subseteq\mathcal{C}$, we have a canonical almost surely deterministic morphism $(X,\mathcal{C},p)\to(X,\mathcal{B},p)$ (note the direction of the arrow).
Therefore every almost surely increasing filtration gives rise to a cofiltered (projective) net of measure spaces and almost surely deterministic kernels, which in the sequential case looks as follows.
\begin{equation}\label{backward}
\begin{tikzcd}
   (X,\mathcal{B}_0,p)
    & (X,\mathcal{B}_1,p) \ar{l}
    & (X,\mathcal{B}_2,p) \ar{l} 
    & \dots \ar{l} 
    & (X,\mathcal{B}_n,p) \ar{l} 
    & \dots \ar{l} 
\end{tikzcd}
\end{equation}
(We draw the arrow in this way to keep the filtration in the ``forward'' direction.)
Similarly, an almost surely decreasing filtration gives rise to a filtered (inductive) net, which in the sequential case looks as follows.
\begin{equation}\label{forward}
\begin{tikzcd}
   \dots
    & (X,\mathcal{B}_n,p) \ar{l} 
    & \dots \ar{l} 
    & (X,\mathcal{B}_2,p) \ar{l} 
    & (X,\mathcal{B}_1,p) \ar{l}
    & (X,\mathcal{B}_0,p) \ar{l}     
\end{tikzcd}
\end{equation}
If the filtrations are strictly increasing or decreasing, meaning not just almost surely, the morphisms can be taken to be deterministic, not just almost surely.
On a standard Borel space there is essentially no difference between ``increasing'' and ``almost surely increasing''. On one hand, every increasing filtration is almost surely increasing. A weak converse is given by the following proposition.

\begin{proposition}\label{asincreasing}
    Let $(X,\mathcal{A},p)$ be standard Borel, and let $(\mathcal{B}_\lambda)_{\lambda\in\Lambda}$ be an almost surely increasing filtration. 
    Then there exists an increasing (not just almost surely) filtration $(\mathcal{B}'_\lambda)_{\lambda\in\Lambda}$ such that for all $\lambda$, $\mathcal{B}'_\lambda\simeq\mathcal{B}_\lambda$, and moreover for each $\lambda\le\mu$ each of the following squares commute,
    \[
    \begin{tikzcd}
        (X,\mathcal{B}_\mu,p) \ar{d}{\cong} \ar{r} & (X,\mathcal{B}_{\lambda},p) \ar{d}{\cong} \\
        (X,\mathcal{B}'_\mu,p) \ar{r} & (X,\mathcal{B}'_{\lambda},p)
    \end{tikzcd}
    \]
    where the maps are the ones canonically induced by coarse-graining.

    The same can be said for the decreasing case. 
\end{proposition}

\begin{proof}
    Take $\mathcal{B}'_\lambda=\mathcal{I}_{e_{\mathcal{B}_\lambda}}$. 
\end{proof}

We know by the results of \Cref{sec_supinf} that the limit of an increasing filtration (in the subcategory of a.s.\ deterministic kernels, i.e.~the dagger epis) is given by the supremum in the almost-surely-coarse preorder ($\lesssim$), and the colimit of a decreasing filtration is given by the infimum.
However, relating the suprema and infima for the almost-surely-coarse preorder ($\lesssim$) and for the inclusion order ($\subseteq$) between sigma-algebras requires a little care. 

First of all, call a sub-sigma-algebra \newterm{null-set-complete} if it contains all the null sets, or equivalently, if it is in the form $\mathcal{I}_k$ for some kernel $k$. It is easy to see that the joins (for the inclusion order) and the meets (i.e.~intersection) of null-set-complete sub-sigma-algebras are again null-set complete. 

Somewhat conversely, on a standard Borel space we can turn any sub-sigma-algebra into an almost surely equivalent, null-set-complete one via the assignment $\mathcal{B}\mapsto\mathcal{I}_{e_\mathcal{B}}$.
This assignment preserves joins (for the inclusion order). To see this, denote by $\mathcal{N}$ the sub-sigma-algebra generated by all null sets, i.e.~the one of sets $A\in\mathcal{A}$ for which either $p(A)=0$ or $p(A)=1$. Now we have that $\mathcal{I}_{e_\mathcal{B}}$ is the join of $\mathcal{B}$ and $\mathcal{N}$, and so the assignment $\mathcal{B}\mapsto\mathcal{B}\vee\mathcal{N}$ preserves all joins.

However, this assignment in general does not preserve meets (intersections). For an example, consider the set $\{x,y,z\}$ with the discrete sigma-algebra $\mathcal{A}$ and the measure $p(x)=p(z)=1/2$, $p(y)=0$. Consider now the sub-sigma-algebras $\mathcal{B}$ and $\mathcal{C}$ generated by the partitions $\{\{x,y\},\{z\}\}$ and $\{\{x\},\{y,z\}\}$.
We have that $\mathcal{B}\cap\mathcal{C}$ is the trivial sigma-algebra, and so $\mathcal{I}_{e_{\mathcal{B}\cap\mathcal{C}}}=\mathcal{N}$.
On the other hand $\mathcal{B}\simeq\mathcal{C}$, so that $\mathcal{I}_{e_\mathcal{B}}=\mathcal{I}_{e_\mathcal{C}}=\mathcal{I}_{e_\mathcal{B}}\cap\mathcal{I}_{e_\mathcal{C}}=\mathcal{A}$.

\begin{remark}
    For the readers familiar with order theory, we can model the situation as follows.
    Denote by $\Sigma X$ the set of sub-sigma algebras of $\mathcal{A}$, denote by $NC\Sigma X$ the set of null-set-complete sub-sigma-algebras, and denote by $I$ the set of idempotents on $(X,\mathcal{A},p)$ with the idempotent order.
    We have respectively a Galois connection, an isomorphism, and an equivalence of preorders as in the following diagram:
    \[
    \begin{tikzcd}[row sep=tiny]
        (\Sigma X, \subseteq) \ar[bend right]{r}[inner sep=1.8mm]{\top} & (NC\Sigma X, \subseteq) \ar[hook']{l} \ar[leftrightarrow]{r}{\cong} & (I,\le) \ar[leftrightarrow]{r}{\simeq} & (\Sigma X, \lesssim) \\
        \mathcal{B} \ar[mapsto]{r} & \mathcal{I}_{e_\mathcal{B}}
    \end{tikzcd}
    \]
    Now the assignment $\mathcal{B}\mapsto\mathcal{I}_{e_\mathcal{B}}$, being a lower (or left) adjoint, preserves all joins (but not all meets). 
\end{remark}

For sequential filtrations, in any case, the necessary meets are preserved, as the following proposition shows.

\begin{proposition}\label{sequentialinf}
    Let $(X,\mathcal{A},p)$ be standard Borel.
    The assignment $\mathcal{B}\mapsto\mathcal{I}_{e_\mathcal{B}}$ preserves intersections of decreasing filtrations, i.e.\ sequential infima for the inclusion relation ($\subseteq$).
\end{proposition}

\begin{proof}
    Let $(\mathcal{B}_n)_{n=0}^\infty$ be a decreasing filtration. We have to prove that 
    \[
    \bigcap_n \left( \mathcal{I}_{e_{\mathcal{B}_n}}\right) \;=\; \mathcal{I}_{e_{\left( \bigcap_n\mathcal{B}_n \right)}} .
    \]
    Since the assignment $\mathcal{B}\mapsto\mathcal{I}_{e_\mathcal{B}}$ is monotone, we already have that $\mathcal{I}_{e_{\left( \bigcap_n\mathcal{B}_n \right)}}\subseteq \bigcap_n \left( \mathcal{I}_{e_{\mathcal{B}_n}}\right)$. We have to prove the other direction.
    
    So let $A\in \bigcap_n \left( \mathcal{I}_{e_{\mathcal{B}_n}}\right)$. 
    This means that for every $n$, $A\in \mathcal{I}_{e_{\mathcal{B}_n}}$, which means that $A$ is $p$-almost surely in each $\mathcal{B}_n$. In other words, for each $n$ there exists $B_n\in\mathcal{B}_n$ such that $p(A\,\triangle\,B_n)=0$. 
    Now since for each $n$ we have that $\mathcal{B}_{n+1}\subseteq\mathcal{B}_n$, we have that for every $n$ and for every $k\ge n$, $B_k\in\mathcal{B}_n$ as well.
    Define now for each $n$ the set
    \[
    B'_n \coloneqq \bigcup_{k\ge n} B_k .
    \]
    For every $n$ we have that 
    \begin{itemize}
        \item $B'_n\in\mathcal{B}_n$;
        \item $B'_n\supseteq B'_{n+1}$;
        \item $p(B'_n\,\triangle\,A)=0$.
    \end{itemize}
    The first item implies, using again that $\mathcal{B}_{n+1}\subseteq\mathcal{B}_n$, that for every $n$ and for every $k\ge n$, $B'_k\in\mathcal{B}_n$ as well.
    The second item says exactly that the sequence $(B'_n)$ is nonincreasing. With these two remarks in mind, define now the set
    \[
    B \coloneqq \bigcap_{k=0}^\infty B'_k .
    \]
    Note that since the sequence $(B'_n)$ is nonincreasing, we could define $B$ equivalently as 
    \[
    B \coloneqq \bigcap_{k\ge n} B'_k 
    \]
    for any $n$. Therefore $B\in\mathcal{B}_n$ for every $n$, and hence $B\in\bigcap_n\mathcal{B}_n$. 
    Moreover, once again by construction, $p(A\,\triangle\,B)=0$. Therefore $A$ is almost surely in $\bigcap_n\mathcal{B}_n$, which means that $A\in\mathcal{I}_{e_{\left( \bigcap_n\mathcal{B}_n \right)}}$. 
\end{proof}

For non-sequential filtrations (nets), instead, one has to be careful, since the infimum in the idempotent order is given by the intersection of the \emph{null-set-completion} $\mathcal{I}_{e_{\left( \bigcap_n\mathcal{B}_\lambda \right)}}$ of the $\mathcal{B}_\lambda$.
This difficulty is avoided if one starts with a decreasing net of null-set-complete sub-sigma-algebras.
In what follows, to avoid ambiguity, we will denote by $\bigcap_\lambda\mathcal{B}_\lambda$ the ordinary intersection of sigma-algebras, and by $\bigwedge_\lambda\mathcal{B}_\lambda$ the intersection of the null-set completions (giving the meet in the idempotent order).

Let's now use the ideas of \Cref{sec_supinf}. Given a standard Borel space $(X,\mathcal{A},p)$, the category $\DSI(X,\mathcal{A},p)$ has (equivalently) as objects, sub-sigma-algebras $\mathcal{B}\subseteq\mathcal{A}$ up to almost sure equality, together with the maps 
\[
\begin{tikzcd}
    (X,\mathcal{A},p) \ar[shift left]{r}{\pi} & (X,\mathcal{B},p) \ar[shift left]{l}{\iota},
\end{tikzcd}
\]
where $\pi=\delta_\id$ and $\iota=\pi^+$ is the disintegration. As morphisms, corresponding to (almost sure) inclusions $\mathcal{B}_1\lesssim\mathcal{B}_2$, it has almost surely deterministic kernels $\pi_{2,1}:(X,\mathcal{B}_2,p)\to(X,\mathcal{B}_1,p)$ making the upper and lower triangles in the following diagram commute.
\[
\begin{tikzcd}[column sep=tiny]
    & (X,\mathcal{A},p) \ar{dl}[swap]{\pi_1} \ar{dr}{\pi_2} \\
    (X,\mathcal{B}_1,p) \ar{dr}[swap]{\iota_1} \ar[shift right]{rr}[swap]{\pi_{2,1}^+}
     && (X,\mathcal{B}_2,p) \ar{dl}{\iota_2} \ar[shift right]{ll}[swap]{\pi_{2,1}} \\
    & (X,\mathcal{A},p) 
\end{tikzcd}
\]
Equivalently, it can be seen as the slice category $\cat{Krn}_\dm/(X,\mathcal{A},p)$, or the opposite of $(X,\mathcal{A},p)/\cat{Krn}_\de$.

\begin{corollary}
    Let $(X,\mathcal{A},p)$ be standard Borel. In the category $\DSI(X,\mathcal{A},p)$ (equivalently, $\cat{Krn}_\dm/(X,\mathcal{A},p)$ and $(X,(\mathcal{A},p)/\cat{Krn}_\de)^\op$):
    \begin{itemize}
        \item The limit of an increasing filtration $(\mathcal{B}_\lambda)$ on $(X,\mathcal{A},p)$ exists, and it is given by the join sigma-algebra $\bigvee_\lambda\mathcal{B}_\lambda$, i.e.~the smallest (coarsest) sigma-algebra containing all the $\mathcal{B}_\lambda$;
        \item The colimit of a decreasing filtration $(\mathcal{B}_\lambda)$ on $(X,\mathcal{A},p)$ exists, and is given by the meet (intersection of the null-set completions) sigma-algebra $\bigwedge_\lambda\mathcal{B}_\lambda$.
    \end{itemize}
\end{corollary}

Let's now apply the random variable functors to filtrations. We focus on increasing filtrations, decreasing filtrations are analogous.
If we apply the $L^1$ (or $L^n$) functor to an increasing filtration, which in the sequential case looks as follows,
\[
\begin{tikzcd}[column sep=large]
   (X,\mathcal{B}_0,p)
    & (X,\mathcal{B}_1,p) \ar{l}[swap]{\pi_{1,0}}
    & (X,\mathcal{B}_2,p) \ar{l}[swap]{\pi_{2,1}} 
    & \dots \ar{l}[swap]{\pi_{3,2}} 
\end{tikzcd}
\]
where we denoted by $\pi_{n,n-1}$ the canonical map $(X,\mathcal{B}_n,p)\to (X,\mathcal{B}_{n-1},p)$,
we get a net with the arrows reversed:
\[
\begin{tikzcd}[column sep=large]
   L^1(X,\mathcal{B}_0,p) \ar{r}{(\pi_{1,0})^*}
    & L^1(X,\mathcal{B}_1,p) \ar{r}{(\pi_{2,1})^*} 
    & L^1(X,\mathcal{B}_2,p) \ar{r}{(\pi_{3,2})^*} 
    & \dots 
\end{tikzcd}
\]
For example, the map $(\pi_{2,1})^*:L^1(X,\mathcal{B}_1,p)\to L^1(X,\mathcal{B}_2,p)$ includes the $\mathcal{B}_1$-measurable random variables into the space of $\mathcal{B}_2$-measurable ones, which is possible since $\mathcal{B}_2$ is finer or equal to $\mathcal{B}_1$.
(Note that, technically, $L^1(X,\mathcal{B}_1,p)$ is not exactly a subset of $L^1(X,\mathcal{B}_2,p)$, since the latter may have more measure zero sets, but it defines a subset canonically.)

We can also look at the disintegrations $\pi_{n,n-1}^+:(X,\mathcal{B}_{n-1},p)\to(X,\mathcal{B}_n,p)$, fitting into the original diagram as follows.
\[
\begin{tikzcd}[column sep=large]
   (X,\mathcal{B}_0,p) \ar[shift left]{r}{\pi_{1,0}^+}
    & (X,\mathcal{B}_1,p) \ar[shift left]{l}{\pi_{1,0}} \ar[shift left]{r}{\pi_{1,2}^+}
    & (X,\mathcal{B}_2,p) \ar[shift left]{l}{\pi_{2,1}} \ar[shift left]{r}{\pi_{3,2}^+}
    & \dots \ar[shift left]{l}{\pi_{3,2}} 
\end{tikzcd}
\]
Applying the $L^1$ (or $L^n$) functor we get the following maps.
\[
\begin{tikzcd}[column sep=large]
   L^1(X,\mathcal{B}_0,p) \ar[leftarrow,shift left]{r}{(\pi_{1,0}^+)^*}
    & L^1(X,\mathcal{B}_1,p) \ar[leftarrow,shift left]{l}{(\pi_{1,0})^*} \ar[leftarrow,shift left]{r}{(\pi_{1,2}^+)^*}
    & L^1(X,\mathcal{B}_2,p) \ar[leftarrow,shift left]{l}{(\pi_{2,1})^*} \ar[leftarrow,shift left]{r}{(\pi_{3,2}^+)^*}
    & \dots \ar[leftarrow,shift left]{l}{(\pi_{3,2})^*} 
\end{tikzcd}
\]
The new maps encode conditional expectations: for example, the map $(\pi_{2,1}^+)^*:L^1(X,\mathcal{B}_2,p)\to L^1(X,\mathcal{B}_1,p)$ takes the conditional expectation of $\mathcal{B}_2$-measurable random variables given $\mathcal{B}_1$ (which is possible since $\mathcal{B}_2$ is finer or equal to $\mathcal{B}_1$).
Focusing on these maps, 
\[
\begin{tikzcd}[column sep=large]
   L^1(X,\mathcal{B}_0,p) \ar[leftarrow]{r}{(\pi_{1,0}^+)^*}
    & L^1(X,\mathcal{B}_1,p) \ar[leftarrow]{r}{(\pi_{1,2}^+)^*}
    & L^1(X,\mathcal{B}_2,p) \ar[leftarrow]{r}{(\pi_{3,2}^+)^*}
    & \dots 
\end{tikzcd}
\]
we can encode a \emph{martingale} as a collection of random variables, one for every space, compatible with the arrows:

\begin{definition}
    Let $(X,\mathcal{A},p)$ be a standard Borel probability space.
    
    A \newterm{(forward) martingale} on $(X,\mathcal{A},p)$ consists of 
    \begin{itemize}
        \item An increasing filtration $(\mathcal{B}_\lambda)_{\lambda\in\Lambda}$ on $(X,\mathcal{A},p)$;
        \item For each $\lambda\in\Lambda$, a random variable $f_\lambda\in L^1(X,\mathcal{B}_\lambda,p)$,
    \end{itemize}
    such that for all $\lambda\le\mu$, $f_\lambda$ is a conditional expectation of $f_{\mu}$ given $\mathcal{B}_\lambda$.
    
    A \newterm{backward martingale} on $(X,\mathcal{A},p)$ consists of 
    \begin{itemize}
        \item A decreasing filtration $(\mathcal{B}_\lambda)_{\lambda\in\Lambda}$ on $(X,\mathcal{A},p)$;
        \item For each $\lambda\in\Lambda$, a random variable $f_\lambda\in L^1(X,\mathcal{B}_\lambda,p)$,
    \end{itemize}
    such that for all $\lambda\le\mu$, $f_{\mu}$ is a conditional expectation of $f_{\lambda}$ given $\mathcal{B}_{\mu}$.
\end{definition}

Recall that we can encode a random variable as a bounded linear map $\R\to (X,\mathcal{A},p)$. Similarly, we can encode a martingale as a commutative diagram, which in the sequential case looks as follows,
\[
\begin{tikzcd}[column sep=small]
    &&& \R \ar{dlll}[swap]{f_0} \ar{dl}{f_1} \ar{dr}[swap]{f_2} \ar{drrr} \\
    L^1(X,\mathcal{B}_0,p)
     && L^1(X,\mathcal{B}_1,p) \ar{ll}{(\pi_{1,0}^+)^*}
     && L^1(X,\mathcal{B}_2,p) \ar{ll}{(\pi_{2,1}^+)^*} 
     && \cdots \ar{ll}{(\pi_{3,2}^+)^*} 
\end{tikzcd}
\]
and a backward martingale is described analogously.
Note that the $f_n$ commute with the maps $(\pi_{n,n-1}^+)^*$ (since they are conditional expectations of each other), but \emph{not} with the maps $(\pi_{n,n-1})^*$ in general: that would mean that all the $f_n$ are almost surely equal, and almost surely $\mathcal{B}_i$-measurable for all $i$.  

Note also that, by the results above, we can equivalently work with \emph{almost surely} increasing (resp.~decreasing) filtrations (but in the decreasing, non-sequential case, taking the infimum will require some care).

Let's now consider the elements $\mathcal{B}_\lambda$ of the filtration as sub-sigma-algebras of $\mathcal{A}$. We have canonical maps as follows.
\[
\begin{tikzcd}
    (X,\mathcal{A},p) \ar[shift left]{r}{\pi_\lambda} & (X,\mathcal{B}_\lambda,p) \ar[shift left]{l}{\pi_\lambda^+}
\end{tikzcd}
\]
where $\pi_\lambda:(X,\mathcal{A},p)\to(X,\mathcal{B}_\lambda,p)$ is the kernel induced by the set-theoretic identity, and $\pi^+_\lambda$ is the corresponding disintegration. 
The pair $(\pi_\lambda,\pi^+_\lambda)$ forms a (dagger-)split idempotent, and those are preserved by all functors.
Therefore we have again a split idempotent 
\[
\begin{tikzcd}
    L^1(X,\mathcal{A},p) \ar[leftarrow,shift left]{r}{(\pi_\lambda)^*} & L^1(X,\mathcal{B}_\lambda,p) \ar[leftarrow,shift left]{l}{(\pi_\lambda^+)^*} ,
\end{tikzcd}
\]
where the map $(\pi_\lambda)^*:L^1(X,\mathcal{B}_\lambda,p)\to L^1(X,\mathcal{A},p)$ is (almost) an inclusion, and the map $(\pi_\lambda^+)^*:L^1(X,\mathcal{A},p)\to L^1(X,\mathcal{B}_\lambda,p)$ forms conditional expectations. 

If we consider the functor $L^2:\cat{EKrn}\to\cat{Hilb}$, which preserves the dagger structure, we have that the maps $(\pi_\lambda)^*:L^2(X,\mathcal{B}_\lambda,p)\to L^2(X,\mathcal{A},p)$ and $(\pi_\lambda^+)^*:L^2(X,\mathcal{A},p)\to L^2(X,\mathcal{B}_\lambda,p)$ form moreover a \emph{dagger}-split idempotent, just like $\pi_\lambda$ and $\pi_\lambda^+$. 

We have seen above that, in the order of split idempotents of $(X,\mathcal{A},p)$, the join sigma-algebra is the supremum of an increasing filtration, and the meet (intersection of the null-set completions) sigma-algebra is the infimum of a decreasing filtration. 
It is natural to ask if the $L^n$ functors preserve these suprema and infima, and the answer is affirmative.

\begin{proposition}\label{preserves_optima}
    Let $(X,\mathcal{A},p)$ be a standard Borel probability space, and let $n\in\N$.

    \begin{enumerate}
        \item Given an increasing filtration $(\mathcal{B}_\lambda)_{\lambda\in\Lambda}$ with supremum $\mathcal{B}_\infty=\bigvee_{\lambda\in\Lambda}\mathcal{B}_\lambda$, the space $L^n(X,\mathcal{B}_\infty,p)$ is the supremum of the $L^n(X,\mathcal{B}_\lambda,p)$ in the order of split idempotents on $L^n(X,\mathcal{A},p)$ (equivalently, the colimit in $\SI(L^n(X,\mathcal{A},p))$).

        \item Given a decreasing filtration $(\mathcal{C}_\lambda)_{\lambda\in\Lambda}$ with infimum $\mathcal{C}_\infty=\bigwedge_{\lambda\in\Lambda}\mathcal{C}_\lambda$, the space $L^n(X,\mathcal{C}_\infty,p)$ is the infimum of the $L^n(X,\mathcal{C}_\lambda,p)$ in the order of split idempotents on $L^n(X,\mathcal{A},p)$ (equivalently, the limit in $\SI(L^n(X,\mathcal{A},p))$).
    \end{enumerate}
\end{proposition}

\begin{proof}
    First of all, without loss of generality, by possibly replacing $\mathcal{B}_\lambda$ and $\mathcal{B}_\infty$ with their null-set completions, we can assume that $\mathcal{B}_\lambda$- and $\mathcal{B}_\infty$-measurability and almost sure measurability coincide. (Same for $\mathcal{C}_\lambda$ and $\mathcal{C}_\infty$.) Moreover, this way the infimum will simply be given by the intersection of sigma-algebras.
    
    \begin{enumerate}
        \item Notice that for split idempotents of Banach spaces, in order for $L^n(X,\mathcal{B}_\infty,p)$ to be a supremum it suffices to check that it is the closure of the union of the $L^n(X,\mathcal{B}_\lambda,p)$, and that it admits a retract. As we know, the retract is given by the conditional expectation map $(\pi^+)^*:L^n(X,\mathcal{A},p)\to L^n(X,\mathcal{B}_\infty,p)$, where $\pi:(X,\mathcal{A},p)\to(X,\mathcal{B}_\infty,p)$ is the kernel induced by the set-theoretic identity.
        
        To prove that $L^n(X,\mathcal{B}_\infty,p)$ is the closure of the union, let first of all $B\in\mathcal{B}_\infty$. By \cite[Lemma~3.14]{vanbelle2023martingales}, there exists a sequence $(B_i)_{i=0}^\infty$ of sets from the union $\bigcup_{\lambda\in\Lambda}\mathcal{B}_\lambda$ such that $p(B\,\triangle\,B_i)\to 0$. Therefore $1_B$ is in the closure of the union of the $L^n(X,\mathcal{B}_\lambda,p)$. The same claim can now be extended to all $\mathcal{B}_\infty$-measurable function by approximating them via simple functions.
        Conversely, suppose $f$ is in the closure of the union of the $L^n(X,\mathcal{B}_\lambda,p)$. Then there exists a sequence $(f_i)_{i=0}^\infty$ in $\bigcup_{\lambda\in\Lambda}L^n(X,\mathcal{B}_\lambda,p)$ tending to $f$, and all the functions $f_i$ are also in $L^n(X,\mathcal{B}_\infty,p)$. But now since $L^n(X,\mathcal{B}_\infty,p)$ is complete, $f\in L^n(X,\mathcal{B}_\infty,p)$ as well. So $f\in L^n(X,\mathcal{B}_\infty,p)$ if and only if it is in the closure of the union of the $L^n(X,\mathcal{B}_\lambda,p)$.
        
        \item As above, notice that for split idempotents of Banach spaces, in order for $L^n(X,\mathcal{C}_\infty,p)$ to be an infimum it suffices to check that it is the intersection of the $L^n(X,\mathcal{C}_\lambda,p)$, and that it admits a retract. Again as above, the retract is given by the conditional expectation map $(\pi^+)^*:L^n(X,\mathcal{A},p)\to L^n(X,\mathcal{C}_\infty,p)$, where $\pi:(X,\mathcal{A},p)\to(X,\mathcal{C}_\infty,p)$ is the kernel induced by the set-theoretic identity.

        To prove that $L^n(X,\mathcal{C}_\infty,p)$ is the intersection, suppose that $f\in L^n(X,\mathcal{A},p)$ is $\mathcal{C}_\lambda$-measurable for all $\lambda$. Then for every Borel set $B\subseteq\R$, $f^{-1}(B)\in\mathcal{C}_\lambda$ for all $\lambda$, so that $f^{-1}(C)\in\mathcal{C}_\infty$, and hence $f$ is $\mathcal{C}_\infty$-measurable. 
        Conversely, if $f$ is $\mathcal{C}_\infty$-measurable, then it is $\mathcal{C}_\lambda$-measurable for all $\lambda$. So $f\in L^n(X,\mathcal{C},p)$ if and only if for all $\lambda$, $f\in L^n(X,\mathcal{C}_\lambda,p)$.  \qedhere
    \end{enumerate}
\end{proof}

\begin{corollary}\label{marti_ban}
    By \Cref{banupi} and \Cref{bandowni}, we have that
    \begin{itemize}
        \item In the increasing case, $L^n(X,\mathcal{B}_\infty,p)$ is the colimit in $\cat{Ban}_{\le 1}$ of the $L^n(X,\mathcal{B}_i,p)$ and the inclusion maps $(\pi_{i+1,i})^*:L^n(X,\mathcal{B}_{i},p)\to L^n(X,\mathcal{B}_{i+1},p)$.
        \item In the decreasing case, $L^n(X,\mathcal{B}_\infty,p)$ is the limit in $\cat{Ban}_{\le 1}$ of the $L^n(X,\mathcal{B}_i,p)$ and the inclusion maps $(\pi_{i,i+1})^*:L^n(X,\mathcal{B}_{i+1},p)\to L^n(X,\mathcal{B}_{i},p)$.
    \end{itemize}
\end{corollary}

Moreover, in the case of $L^2$, which is a Hilbert space, the dagger structure gives us the following dual statements:
\begin{corollary}\label{marti_hilb}
    By \Cref{hilbupp} and \Cref{hilbdownp}, we have that 
    \begin{itemize}
        \item In the increasing case, $L^2(X,\mathcal{B}_\infty,p)$ is the limit in $\cat{Hilb}_{\le 1}$ of the $L^2(X,\mathcal{B}_i,p)$ and the projection maps $(\pi_{i+1,i}^+)^*:L^2(X,\mathcal{B}_{i+1},p)\to L^2(X,\mathcal{B}_{i},p)$.
        \item In the decreasing case, $L^2(X,\mathcal{B}_\infty,p)$ is the colimit in $\cat{Hilb}_{\le 1}$ of the $L^2(X,\mathcal{B}_i,p)$ and the inclusion maps $(\pi_{i+1,i}^+)^*:L^2(X,\mathcal{B}_{i},p)\to L^2(X,\mathcal{B}_{i+1},p)$.
    \end{itemize}
\end{corollary}

The latter property says in particular that given an $L^2$-bounded martingale (for example, sequential) $(f_n)$ on a filtration $(\mathcal{B}_n)$, forming a cone as below,
\[
\begin{tikzcd}[column sep=small]
    &&& \R \ar{dlll}[swap]{f_0} \ar{dl}{f_1} \ar[dashed]{drrr}{f} \\
    L^2(X,\mathcal{B}_0,p) 
     && L^2(X,\mathcal{B}_1,p) \ar{ll}{(\pi_{1,0}^+)^*} 
     && \cdots \ar{ll}{(\pi_{2,1}^+)^*} 
     && L^2(X,\mathcal{B}_\infty,p) \ar{ll}
\end{tikzcd}
\]
there exists a unique random variable $f$ as above making the diagram commute, i.e.~a $\mathcal{B}_\infty$-measurable random variable $f$ such that such that for all $n$, $f_n=\E[f|\mathcal{B}_n]$. This is in general not the case for $L^1$, in particular if the $f_i$ are not uniformly integrable.

\begin{example}\label{non_int_mart}
    Let $X=[0,1]$ with its Borel sigma-algebra. For $n\in\N$, let $\mathcal{B}_n$ be the sigma-algebra generated by the inverse powers of 2, i.e. by the intervals $[0,1/2^n],(1/n,2/2^n],\dots,((2^n-1)/2^n, 1]$. 
    The functions $f_i$ given by
    \[
    f_n(x) \coloneqq \begin{cases}
        2^n & x \le 1/2^n \\
        0 & x > 1/2^n
    \end{cases}
    \]
    form a martingale on $(\mathcal{B}_n)_{n\in\N}$, and we have that for all $n$,
    \[
    \|f_n\|_{L^1} \;=\; \int_{[0,1]} |f_n(x)|\,dx \;=\; \int_{[0,1/2^n]} 2^n \,dx \;=\; 1 .
    \]
    However there does not exist an $f$ such that for all $n$, $\E[f|\mathcal{B}_n]=f_n$. 
\end{example}

What we \emph{always} get, by having a colimit in $\SI(L^1(X,\mathcal{A},p))$, is that if \emph{we have} a random variable $f\in L^1(X,\mathcal{A},p)$, and its conditional expectations $f_n$ with respects to a filtration $(\mathcal{B}_n)_{n=0}^\infty$, then the conditional expectation of $f$ given $\mathcal{B}_\infty=\bigvee_n\mathcal{B}_n$ is a $\mathcal{B}$-measurable random variable $f_\infty$ such that for all $n$, $f_n=\E[f_\infty|\mathcal{B}_n]$. 
In diagrams, given $f:\R\to L^1(X,\mathcal{A},p)$,
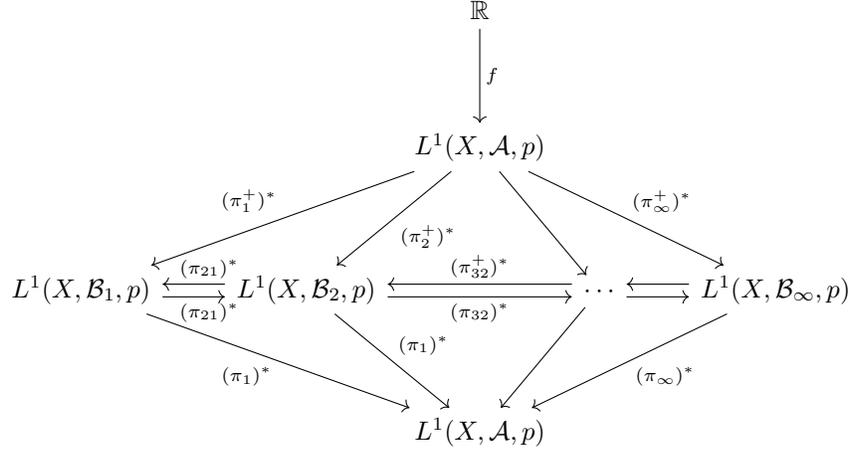
\[
\begin{tikzcd}[row sep=huge, column sep=tiny]
    &&&&& \R \ar{d}{f} \\
    &&&&& L^1(X,\mathcal{A},p) \ar{dlllll}[swap]{(\pi^+_1)^*} \ar{dl}{(\pi^+_2)^*} \ar{dr} \ar{drrrrr}{(\pi^+_\infty)^*} \\
    L^1(X,\mathcal{B}_1,p) \ar{drrrrr}[swap]{(\pi_1)^*} \ar[shift right]{rrrr}[swap, near end]{(\pi_{21})^*} 
     &&&& L^1(X,\mathcal{B}_2,p) \ar{dr}{(\pi_1)^*} \ar[shift right]{rr}[swap]{(\pi_{32})^*} \ar[shift right]{llll}[swap, near start]{(\pi_{21})^*}
     && \cdots \ar{dl} \ar[shift right]{rrrr} \ar[shift right]{ll}[swap]{(\pi^+_{32})^*}
     &&&& L^1(X,\mathcal{B}_\infty,p) \ar{dlllll}{(\pi_\infty)^*} \ar[shift right]{llll} \\
	&&&&& L^1(X,\mathcal{A},p)
\end{tikzcd}
\]
the conditional expectations $f_n$ can be expressed as $(\pi^+_n)^*f\in L^1(X,\mathcal{B}_n,p)$, and the map $f_\infty$ can be written as $(\pi^+_\infty)^*f\in L^1(X,\mathcal{B}_\infty,p)$. By further postcomposing with the maps $(\pi_n)^*$ and $(\pi_\infty)^*$, one can of course consider all these maps again as random variables on $(X,\mathcal{A},p)$. 
Analogous remarks can be made for decreasing filtrations.

The \emph{martingale convergence theorem} says that moreover, this map $f_\infty$ is the limit (in the topological sense) of the $f_n$. We can state and prove this categorically as well, but in order to talk about topological convergence we need an additional ingredient: a topological enrichment.

\section{Topological enrichment}
\label{sec_enrichment}

We now equip both our categories of kernels and our categories of vector spaces with a topological enrichment.
This amounts to equipping each hom-set $\cat{C}(X,Y)$ with a topology which makes the composition of morphisms continuous, and it allows us to talk about convergence of morphisms. When morphisms are functions, in our examples this convergence is the pointwise one.
More details on this enrichment are given in \Cref{sec_top}.

\subsection{The enrichments of Ban and Hilb}
\label{sec_banenrichment}

Let $X$ and $Y$ be Banach spaces. We equip the set of bounded linear maps $\cat{Ban}(X,Y)$ with the topology of pointwise convergence. That is, we say that a net $(f_\lambda:X\to Y)_{\lambda\in\Lambda}$ tends to $f:X\to Y$ if and only if for all $x\in X$, $f_\lambda(x)\to f(x)$ in $Y$, or equivalently,
\[
\| f(x) - f_\lambda(x) \|_Y \to 0 .
\]
Note that we are not requiring that the convergence be uniform in $x$. 

A subbasis of this topology is given by the sets
\[
\{f : \|f(x) - y\|_Y < r \}
\]
for $x\in X$, $y\in Y$, and $r>0$.

In the case of $X=Y$, this topology is sometimes called the \emph{strong operator topology}.

\begin{proposition}
    This choice of topology for each pair of Banach spaces $X$ and $Y$ makes $\cat{Ban}$ enriched in $\cat{Top}$, in the sense of \Cref{sec_top}.
\end{proposition}

The statement follows immediately since we can consider $\cat{Ban}$ a subcategory of $\cat{Top}$ (with the corresponding enrichment). We nevertheless provide a direct proof for clarity.

\begin{proof}
    We have to prove that for all Banach spaces $X$, $Y$ and $Z$, the composition map 
    \[
    \begin{tikzcd}[row sep=0]
        \cat{Ban}(X,Y) \times \cat{Ban}(Y,Z) \ar{r} & \cat{Ban}(X,Z) \\
        (f,g) \ar[mapsto]{r} & g\circ f
    \end{tikzcd}
    \]
    is separately continuous in both variables. 
    \begin{enumerate}
        \item To prove continuity in $f$ (of the postcomposition with $g$), suppose that $f_\lambda\to f:X\to Y$. 
        This means that for all $x\in X$, $f_\lambda(x)\to f(x)$ in $Y$. 
        Since $g$ is continuous, then $g(f_\lambda(x))\to g(f(x))$. 
        This holds for all $x$, which means exactly that $g\circ f_\lambda \to g\circ f$.

        \item To prove continuity in $g$ (of the precomposition with $f$), suppose that $g_\lambda\to g:Y\to Z$. 
        This means that for all $y\in Y$, $g_\lambda(y)\to g(y)$ in $Z$. 
        Setting $y=f(x)$, we see that for all $x\in X$, $g_\lambda(f(x))\to g(f(x))$, i.e.~$g_\lambda\circ f\to g\circ f$. \qedhere
    \end{enumerate}
\end{proof}

Let's now turn to Hilbert spaces. One can consider $\cat{Hilb}$ as a subcategory of $\cat{Ban}$, and inherit the enrichment accordingly. However, this enrichment is not compatible with the dagger structure, i.e.~it does not respect the symmetry: if $f_\lambda\to f:X\to Y$, it is not necessarily true that $f^+_\lambda\to f^+:Y\to X$. 
Therefore we define the following topology: we say that a net $(f_\lambda:X\to Y)_{\lambda=0}^\infty$ tends to $f:X\to Y$ if and only if 
\begin{enumerate}
    \item for all $x\in X$, $f_\lambda(x)\to f(x)$ in $Y$; and
    \item for all $y\in Y$, $f_\lambda^+(y)\to f^+(y)$ in $X$.
\end{enumerate}
Once again, we do not require that the convergence be uniform in $x$. 
A subbasis of this topology is given by the sets
\[
\{f : \|f(x) - y\|_Y < r, \| f^+(y) - x \|_X < r \} 
\]
for $x\in X$, $y\in Y$, and $r>0$.

In the case of $X=Y$, this topology is sometimes called the \emph{strong-* operator topology}.
Note that for self-adjoint operators, this topology coincide with the strong operator topology.

\begin{proposition}
    This choice of topology for each pair of Banach spaces $X$ and $Y$ makes $\cat{Hilb}$ enriched in $\cat{Top}$, in the sense of \Cref{sec_top}.
\end{proposition}

\begin{proof}
    We have to prove that for all Banach spaces $X$, $Y$ and $Z$, the composition map 
    \[
    \begin{tikzcd}[row sep=0]
        \cat{Hilb}(X,Y) \times \cat{Hilb}(Y,Z) \ar{r} & \cat{Hilb}(X,Z) \\
        (f,g) \ar[mapsto]{r} & g\circ f
    \end{tikzcd}
    \]
    is separately continuous in both variables. 
    
    To prove continuity in $f$ (of the postcomposition with $g$), suppose that $f_\lambda\to f:X\to Y$. 
    This means that for all $x\in X$, $f_\lambda(x)\to f(x)$ in $Y$, and for all $y\in Y$, $f_\lambda^+(y)\to f^+(y)$ in $X$. 
    Now since $g$ is continuous, then $g(f_\lambda(x))\to g(f(x))$ for all $x$.
    Moreover, setting $y=g^+(z)$, we have that for all $z\in Z$, $f_\lambda^+(g^+(z))\to f^+(g^+(z))$, i.e.~$(g\circ f_\lambda)^+(z) = (g\circ f)^+(z)$.
    Therefore $g\circ f_\lambda \to g\circ f$.

    Continuity in $g$ is completely analogous, and dual.
\end{proof}

This choice of topology makes the functor $(-)^+:\cat{Hilb}\to\cat{Hilb}$ an equivalence of $\cat{Top}$-enriched categories.
Therefore we have a \emph{topologically enriched dagger category}.

\subsection{The enrichment of Krn and EKrn}
\label{sec_krnenrichment}

A Markov kernel $(X,\mathcal{A},p)\to(Y,\mathcal{B},q)$, intuitively, is at the same time a measurable function on $X$, and a probability measure on $Y$. Because of this, we define a notion of convergence which is like the $L^1$ convergence of measures on $X$, and like the setwise convergence of probability measure on $Y$.
More in detail, consider probability spaces $(X,\mathcal{A},p)$ and $(Y,\mathcal{B},q)$, and let $(k_\lambda:(X,\mathcal{A},p)\to(Y,\mathcal{B},q))_{\lambda\in\Lambda}$ be a sequence of measure-preserving kernels (up to a.s.~equality).
\begin{enumerate}
    \item We say that $k_\lambda\to k$ in the \newterm{one-sided topology} if and only if for every every measurable set $B\in\mathcal{B}$, 
        \[
        \int_X \big| k(B|x) - k_\lambda(B|x) \big| \, p(dx) \;\to\; 0 ;
        \]
    \item If $(X,\mathcal{A},p)$ and $(Y,\mathcal{B},q)$ are essentially standard Borel, we say that $(k_\lambda)$ tends to $k$ in the \newterm{two-sided topology} if and only if both $k_\lambda\to k$ and $k_\lambda^+\to k^+$ in the one-sided topology.
\end{enumerate}
Just as for $\cat{Hilb}$, we will have that the two-sided topology makes $\cat{EKrn}$ an enriched \emph{dagger} category. 
On those kernels $e:X\to X$ which are their own Bayesian inverse (for example, idempotent kernels), the one-sided topology and the two-sided topology coincide.

\begin{theorem}\label{homeomorph}
    Let $(X,\mathcal{A},p)$ and $(Y,\mathcal{B},q)$ be probability spaces, and let $(k_\lambda:(X,\mathcal{A},p)\to(Y,\mathcal{B},q))_{\lambda\in\Lambda}$ be a sequence of measure-preserving kernels (up to a.s.~equality).
    Then the $k_\lambda$ tend to a kernel $k$ in the one-sided topology if and only if for any (hence all) $n$ (including $n=\infty$), the maps $(k_\lambda)^*:L^n(Y,\mathcal{B},q)\to L^n(Y,\mathcal{B},q)$ tend to $k^*$. 
\end{theorem}

In other words, the map 
\[
\begin{tikzcd}[row sep=0]
    \cat{EKrn}\big( (X,\mathcal{A},p), (Y,\mathcal{B},q) \big) \ar{r} & \cat{Ban}\big( L^n(Y,\mathcal{B},q), L^n(X,\mathcal{A},p) \big) \\
    k \ar[mapsto]{r} & k^*
\end{tikzcd}
\]
induces a homeomorphism onto its image.

Let's prove the theorem using the following auxiliary statement.

\begin{lemma}\label{densitylemma}
    Let $X$ and $Y$ be metric spaces, and let $D\subseteq X$ be a dense subset. Let $f:X\to Y$ be 1-Lipschitz, and $(f_\lambda:X\to Y)$ be a net of $1$-Lipschitz functions.
    Then $f_\lambda\to f$ pointwise if and only if for all $d\in D$, $f_\lambda(d)\to f(d)$. 
\end{lemma}
\begin{proof}[Proof of \Cref{densitylemma}]
    Clearly, if the $f_\lambda$ tend to $f$ pointwise, so do their restrictions to $D$.

    Conversely, suppose that for all $d\in D$, $f_\lambda(d)\to f(d)$. 
    Now let $x\in X$. By density we can find a sequence $(x_i)$ in $D$ tending to $x$. 
    For every $\e>0$ there exists $I_\e\in\N$ such that for all $i\ge I_\e$, $d(x,x_i)<\e$. 
    Moreover, for all $i\in\N$, since $x_i\in D$ we have that $f_\lambda(x_i)\to f(x_i)$. So for all $\e>0$ and for all $i\in\N$ we can find $\lambda_{\e,i}$ such that for all $\mu\ge \lambda_{\e,i}$, $d\big(f_\mu(x_i),f(x_i)\big)<\e$. 
    Given $\e>0$, take now $i\ge I_\e$ and $\mu\ge \lambda_{\e,i}$. Since $f$ and $f_\mu$ are $1$-Lipschitz,
    \begin{align*}
    d\big(f_\mu(x),f(x)\big) \;&\le\; d\big(f_\mu(x),f_\mu(x_i)\big) + d\big(f_\mu(x_i),f(x_i)\big) + d\big(f(x_i),f(x)\big) \\
    &\le\; d(x,x_i) + d\big(f_\mu(x_i),f(x_i)\big) + d(x_i,x) \\
    &< 3\e .
    \end{align*}
    Therefore $f_\lambda(x)\to f(x)$. 
\end{proof}

\begin{proof}[Proof of \Cref{homeomorph}]
    We can restate the theorem as the fact that the following conditions are equivalent for all $n$:
    \begin{enumerate}
        \item\label{onset} For all $B\in\mathcal{B}$, 
        \[
        \int_X \big| k(B|x) - k_\lambda(B|x) \big| \, p(dx) \to 0 ;
        \]
        \item\label{onRV} For all $g\in L^n(Y,\mathcal{B},q)$,
        \[
        \left\| k^*g - k_\lambda^*g \right\|_{L^n} \to 0 .
        \]
    \end{enumerate}
    Since the maps $k^*$ and $k_\lambda^*$ are 1-Lipschitz (\Cref{RV}), using \Cref{densitylemma} we can equivalently test condition \ref{onRV} on simple functions, i.e.~those in the form 
    \[
    g(x) = \sum_{i=1}^k \alpha_i\, 1_{B_i}(x) ,
    \]
    where $\alpha_i\in\R$ and $B_i\in\mathcal{B}$, which are dense in $L^n$. 
    We can then equivalently rewrite the integral in \ref{onRV} as
        \[
        \left\| \sum_{i=1}^k \alpha_i \, \big( k(B_i|x) - k_\lambda(B_i|x) \big) \right\|_{L^n} \;\le\quad \sum_{i=1}^k |\alpha_i| \,  \left\| \big( k(B_i|x) - k_\lambda(B_i|x) \big) \right\|_{L^n} ,
        \]
    and so we can rewrite condition \ref{onRV} equivalently as
    \begin{enumerate}\setcounter{enumi}{2}
        \item\label{onsetn} For all $B\in\mathcal{B}$, 
        \[
        \left\| \big( k(B|x) - k_\lambda(B|x) \big) \right\|_{L^n} \to 0 .
        \]
    \end{enumerate}

    With this in mind, let's distinguish the few cases.
    \begin{itemize}
        \item For $n=1$, conditions \ref{onset} and \ref{onsetn} are exactly the same.
        
        \item For $0<n<\infty$, notice that for all $f\in L^n(X,\mathcal{A},p)$,
        \begin{align*}
            (\|f\|_{L^n})^n \;&=\; \int_X |f|^n\,dp \\
            &=\; \int_X |f| \,|f|^{n-1}\,dp \\
            &\le\; \int_X |f| \,\ess\sup |f|^{n-1}\,dp \\
            &= \|f\|_{L^1} \, \|f\|_{L^\infty}^{n-1} .
        \end{align*}
        Setting now $f=k(B|x) - k_\lambda(B|x)$, so that $\|f\|_{L^\infty}\le 1$, we have that $(\|f\|_{L^n})^n\le (\|f\|_{L^1})$, and so $\ref{onset}\Rightarrow\ref{onsetn}$
        The converse statement $\ref{onsetn}\Rightarrow\ref{onset}$ follow by Jensen's inequality. 
        
        \item The $n=\infty$ case follows by letting $n\to\infty$ in the case above.  \qedhere
    \end{itemize}
\end{proof}

\begin{corollary}\label{hilb_homeo}
    Let $(X,\mathcal{A},p)$ and $(Y,\mathcal{B},q)$ be probability spaces, and consider a net of measure-preserving kernels $(k_\lambda:(X,\mathcal{A},p)\to(Y,\mathcal{B},q))_{\lambda\in\Lambda}$ (up to a.s.~equality).
    Then the $k_\lambda$ tend to a kernel $k$ in the two-sided topology if and only if the maps $(k_\lambda)^*:L^2(Y,\mathcal{B},q)\to L^2(X,\mathcal{A},p)$ tend to $k^*$ in the topology of $\cat{Hilb}$. 
\end{corollary}

\begin{proof}
    By \Cref{homeomorph}, we have that $k_\lambda\to k$ in the one-sided topology if and only if $(k_\lambda)^*:L^2(Y,\mathcal{B},q)\to L^2(X,\mathcal{A},p)$ tends to $k^*$ pointwise.
    Similarly, $k_\lambda^+\to k^+$ in the one-sided topology if and only if $(k_\lambda^+)^*:L^2(X,\mathcal{A},p)\to L^2(Y,\mathcal{B},q)$ tends to $(k^+)^*$ pointwise.
    Now $k_\lambda\to k$ in the two-sided topology if and only if $k_\lambda\to k$ and $k_\lambda^+\to k^+$ in the one-sided topology, which happens if and only if $(k_\lambda)^*\to k^*$ and $(k_\lambda^+)^*\to(k^+)^*$ pointwise. The latter condition is precisely convergence in the topology of $\cat{Hilb}$. 
\end{proof}

\begin{proposition}
    The one-sided and two-sided topologies enrich $\cat{GKrn}$ in $\cat{Top}$, in the sense of \Cref{sec_top}.
\end{proposition}

\begin{proof}
    Let's first prove the assert for the one-sided topology.
    We have to prove that for all probability spaces $(X,\mathcal{A},p)$, $(Y,\mathcal{B},q)$ and $(Z,\mathcal{C},r)$, the composition map 
    \[
    \begin{tikzcd}[row sep=0]
        \cat{GKrn}\big((X,\mathcal{A},p),(Y,\mathcal{B},q)\big) \times \cat{GKrn}\big((Y,\mathcal{B},q),(Z,\mathcal{C},r)\big) \ar{r} 
         & \cat{GKrn}\big((X,\mathcal{A},p),(Z,\mathcal{C},r)\big) \\
        (k,h) \ar[mapsto]{r} & h\circ k
    \end{tikzcd}
    \]
    is separately continuous in both variables. 
    
    \begin{enumerate}
        \item\label{ki} To prove continuity in $k$ (of the postcomposition with $h$), suppose that $k_\lambda\to k$.
        Then for all $C\in\mathcal{C}$,
        \begin{align*}
        \int_X \big| (h\circ k)(C|x) - (h\circ k_\lambda)(C|x) \big| \, p(dx) 
         \;&=\; \int_X \left| \int_Y h(C|y)\,k(dy|x) - \int_Y h(C|y)\,k_\lambda(dy|x) \right| \, p(dx) \\
         \;&=\; \| k^*h_C - k_\lambda^*h_C \|_{L^1}
        \end{align*}
        where $h_C(y)=h(C|y)$. By \Cref{homeomorph}, we know that $\| k^*h_C - k_\lambda^*h_C \|_{L^1} \to 0$.
        This holds for all $C\in\mathcal{C}$, and so $h\circ k_\lambda\to h\circ k$. 
        
        \item\label{hi} To prove continuity in $h$ (of the precomposition with $k$), suppose that $h_\lambda\to h$, meaning that for all $C\in\mathcal{C}$, 
        \[
        \int_Y \big| h(C|y) - h_\lambda(C|y) \big| \, q(dy) \;\to\; 0 .
        \]
        Now for every $C\in\mathcal{C}$, by Jensen's inequality and the fact that $k$ is measure-preserving,
        \begin{align*}
        \int_X \big| (h\circ k)(C|x) - (h_\lambda\circ k)(C|x) \big| \, p(dx) 
         \;&=\; \int_X \left| \int_Y h(C|y)\,k(dy|x) - \int_Y h_\lambda(C|y)\,k(dy|x) \right| \, p(dx) \\ 
         \;&\le\; \int_X \int_Y  \big| h(C|y) - h_\lambda(C|y) \big| \,k(dy|x) \, p(dx) \\ 
         \;&=\; \int_Y \big| h(C|y) - h_\lambda(C|y) \big| \, q(dy) \;\to\; 0 .
        \end{align*}
        Therefore $h_\lambda\circ k\to h\circ k$.  
    \end{enumerate}

    Let's now turn to the two-sided topology.
    Expressing the two-sided topology in terms of the one-sided one, we equivalently have to prove that
    \begin{enumerate}
    \setcounter{enumi}{2}
        \item\label{kii} If $k_\lambda\to k$ and $k_\lambda^+\to k^+$, then $h\circ k_\lambda\to h\circ k$ and $k_\lambda^+\circ h^+\to k^+\circ h^+$;
        \item\label{hii} If $h_\lambda\to h$ and $h_\lambda^+\to h^+$, then $h_\lambda\circ k\to h\circ k$ and $k^+\circ h_\lambda^+\to k^+\circ h^+$.
    \end{enumerate}
    (All the convergences above are meant in the one-sided topology).
    To prove \ref{kii}, note that if $k_\lambda\to k$, then $h\circ k_\lambda\to h\circ k$ by \ref{ki}. Similarly, if $k_\lambda^+\to k^+$, then $k_\lambda^+\circ h^+\to k^+\circ h^+$ by \ref{hi} (after renaming). 
    Similarly, to prove \ref{hii}, note that if $h_\lambda\to h$, then $h_\lambda\circ k\to h\circ k$ by \ref{hi}, and if $h_\lambda^+\to h^+$, then $k^+\circ h_\lambda^+\to k^+\circ h^+$ by \ref{ki}.
\end{proof}

In particular, this way $\cat{EKrn}$ is an enriched dagger category, like $\cat{Hilb}$. 

\begin{corollary}\label{Ln_enriched}
    The following functors are enriched:
    \begin{itemize}
        \item $L^n:\cat{Krn}^\op\to\cat{Ban}$, where $\cat{Krn}$ has either the one-sided or two-sided topology;
        \item $L^2:\cat{Krn}^\op\to\cat{Hilb}$, where $\cat{Krn}$ has the two-sided topology;
        \item $L^n\circ(-)^+:\cat{Krn}\to\cat{Ban}$ and $L^2\circ(-)^+:\cat{Krn}\to\cat{Hilb}$, where $\cat{Krn}$ has the two-sided topology. These are the (covariant) functors analogous to the RV functors of \cite{vanbelle2023martingales}.
    \end{itemize} 
\end{corollary}

\begin{proposition}
    The composition of $1$-Lipschitz maps between metric spaces is jointly continuous for the pointwise order.
\end{proposition}

\begin{proof}
    Let $X$, $Y$ and $Z$ be metric spaces. Suppose that $(f_\lambda)$ is a net of $1$-Lipschitz functions $X\to Y$ tending pointwise to $f$, and that $(g_\lambda)$ is a net of $1$-Lipschitz functions $Y\to Z$ tending pointwise to $g$.
    Then for every $x\in X$, for every $\e>0$ we can find $\lambda\in\Lambda$ such that for all $\mu\ge\lambda$, both $d\big( f(x), f_\mu(x) \big)$ and $d\big( g(f(x)), g_\mu(f(x))$ are less than $\e$. Therefore
    \begin{align*}
        d\big( g(f(x)), g_\mu(f_\mu(x)) \big) \;&\le\; d\big( g(f(x)), g_\mu(f(x)) + d\big( g_\mu(f(x)), g_\mu(f_\mu(x)) \\
        &\le\; d\big( g(f(x)), g_\mu(f(x)) + d\big( f(x), f_\mu(x) \big) \\
        &<\; 2\,\e .
    \end{align*}
    Hence $g_\lambda\circ f_\lambda\to g\circ f$.
\end{proof}

\begin{corollary}\label{jointly}
    The composition in $\cat{GKrn}$ is jointly continuous, not just separately continuous. 
\end{corollary}

\subsection{The subspace of idempotent kernels}
\label{sec_idenrichment}

Let's now restrict to the case of almost surely idempotent kernels. 

\begin{proposition}
    Let $(X,\mathcal{A},p)$ be a probability space, and let $(e_\lambda)$ be a net of a.s.\ idempotent kernels on $(X,\mathcal{A},p)$ tending to a kernel $e$. Then $e$ is a.s.\ idempotent as well. 

    In other words, almost surely idempotent kernels form a closed subset of $\cat{GKrn}\big((X,\mathcal{A},p),(X,\mathcal{A},p)\big)$.
\end{proposition}
\begin{proof}
    Since composition in $\cat{GKrn}$ is jointly continuous (\Cref{jointly}), 
    \[
    e\circ e \;=\; \lim_\lambda (e_\lambda\circ e_\lambda) \;=\; \lim_\lambda e_\lambda \;=\; e.
    \]
    Therefore $e$ is idempotent.
\end{proof}

Now let's restrict to the almost surely idempotents on an essentially standard Borel space $(X,\mathcal{A},p)$.
Notice that since for all a.s.~idempotent kernels $e$ we have $e^+=e$ (almost surely), the one-sided and two-sided topologies coincide.

\begin{proposition}\label{closedorder}
    Let $(X,\mathcal{A},p)$ be a standard Borel probability space, and let $(e_\lambda)$ and $(f_\lambda)$ be nets of a.s.\ idempotent kernels on $(X,\mathcal{A},p)$, indexed by the same directed set, and tending respectively to kernels $e$ and $f$. 
    Suppose moreover that in the order of (split) idempotents, $e_\lambda\le f_\lambda$ for all $i$. Then $e\le f$.

    In other words, the order relation of idempotents on $(X,\mathcal{A},p)$ is closed.
\end{proposition}
\begin{proof}
    Since composition in $\cat{GKrn}$ is jointly continuous (\Cref{jointly}), 
    \[
    e\circ f \;=\; \lim_\lambda (e_\lambda\circ f_\lambda) \;=\; \lim_\lambda e_\lambda \;=\; e .
    \]
    The condition $f\circ e = e$ is proven similarly.
\end{proof}

We have seen in \Cref{sec_idempsigma} that idempotent kernels on a standard Borel probability space correspond to sub-sigma-algebras (up to almost sure equality). 
So equivalently, the topology we have on idempotent kernels can be seen as a topology on the set of (equivalence classes of) sub-sigma-algebras.
Concretely, we can say that a net $(\mathcal{B}_\lambda)$ of sub-sigma algebras of $(X,\mathcal{A},p)$ tends to a sub-sigma-algebra $\mathcal{B}\subseteq\mathcal{A}$ if and only if for the corresponding idempotents we have $e_\lambda\to e$, i.e.~for all measurable sets $A\in\mathcal{A}$,
\[
\int_X \big| \P[A|\mathcal{B}](x) - \P[A|\mathcal{B}_\lambda](x) \big|\,p(dx) \to 0 .
\]
Equivalently, if for all integrable random variables $f$,
\[
\int_X \big| \E[f|\mathcal{B}](x) - \E[f|\mathcal{B}_\lambda](x) \big|\,p(dx) \to 0 ,
\]
i.e.\ the conditional expectations $\E[f|\mathcal{B}_\lambda]$ tend to the conditional expectation $\E[f|\mathcal{B}]$ in $L^1$ (and also in $L^n$, when the resulting quantities are finite).

The martingale convergence theorems, which we will prove in the next section, will tell us that
\begin{itemize}
    \item For an increasing filtration $(\mathcal{B}_\lambda)$, the $\mathcal{B}_\lambda$ tend topologically to their supremum $\bigvee_\lambda \mathcal{B}_\lambda$;
    \item For a decreasing filtration $(\mathcal{C}_\lambda)$, the $\mathcal{C}_\lambda$ tend topologically to their infimum $\bigwedge_\lambda \mathcal{C}_\lambda$.
\end{itemize}

\section{Martingale convergence}
\label{sec_convergence}

We now have all the ingredients in place to talk about martingale convergence categorically:
\begin{itemize}
    \item A way to talk about conditional expectations (\Cref{sec_condexp});
    \item A way to talk about filtrations and martingales (\Cref{sec_martingales});
    \item A notion of convergence (the topological enrichment from \Cref{sec_enrichment}).
\end{itemize}
We will show that the topological convergence of martingales can be interpreted as the fact that \emph{convergence in the order implies convergence in the topology}, a phenomenon we call \emph{Levi property}. 

\subsection{Levi properties}
\label{sec_levi}

Recall Levi's theorem, which says that every bounded monotone real sequence converges topologically to its supremum. 
We can define the \emph{Levi property} of a poset analogously:

\begin{definition}
    Let $X$ be a topological space equipped with a partial order. We say that 
    \begin{itemize}
        \item $X$ has the \newterm{upward Levi property} if an upper-bounded upward monotone net (i.e.\ a sequence $(x_\lambda)$ such that for all $\lambda\le\mu$, $x_\lambda\le x_{\mu}$, and such that there exists $y$ such that for all $\lambda$, $x_\lambda\le y$) has a supremum, and tends to it topologically.
        \item $X$ has the \newterm{downward Levi property} if a lower-bounded downward monotone net has an infimum, and tends to it topologically.
    \end{itemize}
\end{definition}

\begin{proposition}\label{limitclosedorder}
    Let $X$ be a topological space equipped with a partial order, and suppose that the order relation is a separately closed (hence jointly closed) subset of $X\times X$. 
    Then if an upward monotone net converges, it converges to it supremum.
    
    Similarly, if a downward monotone net converges, it converges to its infimum.
\end{proposition}

\begin{proof}
    We will prove the upward case, the downward case is analogous. 
    
    Let $(x_\lambda)$ be an upward monotone net, and let $x$ be its limit. Since the net is increasing, for $\lambda\le\mu$ we have $x_\lambda\le x_\mu$. Since the order is closed, taking the limit in $\mu$ we get that $x_\lambda\le x$ for all $\lambda$. Therefore $x$ is an upper bound. 
    Now suppose that for some $y$, $x_\lambda\le y$ for all $\lambda$. Then, again since the order is closed, taking the limit in $\lambda$ gives $x\le y$.
    Therefore $x$ is the least upper bound.
\end{proof}

\begin{definition}
    Let $\cat{C}$ be a dagger category with a topological enrichment.
    We say that $\cat{C}$ has the \newterm{idempotent Levi property} if for every object $X$, the poset of dagger idempotents on $X$ has the upward and downward Levi properties. 
\end{definition}

Let's spell this out in detail. Let $(e_\lambda:X\to X)$ be a net of dagger idempotents on $X$, with dagger splittings $(A_\lambda,\iota_\lambda,\pi_\lambda)$.
Suppose first that the net is increasing: for all $\lambda\le\mu$, $e_\lambda\le e_{\mu}$, i.e.\ $e_\lambda\circ e_\mu = e_\mu\circ e_\lambda = e_\lambda$. It is automatically bounded above. The upward Levi property says then that the net $e_\lambda$ has a supremum $e$, and that $e_\lambda\to e$ topologically. 
Now suppose instead that the net is decreasing: for all $\lambda\le\mu$, $e_\lambda\ge e_{\mu}$ (that is, $e_\lambda\circ e_\mu = e_\mu\circ e_\lambda = e_\mu$). Suppose moreover it is bounded below. The downward Levi property says then that the net $e_\lambda$ has an infimum $e$, and that $e_\lambda\to e$ topologically. 
In $\cat{EKrn}$, we always have a lower bound, the constant kernel. In $\cat{Hilb}$, we can take the zero subspace.

In both $\cat{EKrn}$ and $\cat{Hilb}$, we know that the suprema and infima exist. We will now show topological convergence.

\subsection{The idempotent Levi property of Hilb}
\label{sec_levi_hilb}

In the category of Hilbert spaces, categorical limits often give rise to topological limits of norms (see \cite{dimeglio2024hilbert} for more on this). We can formalize this idea by saying that, in $\cat{Hilb}$, the idempotent Levi property holds.

\begin{proposition}[upward Levi property for $\cat{Hilb}$]\label{up_levi_hilb}
    Let $X$ be a Hilbert space. Let $(A_\lambda)_{\lambda\in\Lambda}$ and $A$ be closed subspaces of $X$, with corresponding orthogonal projectors $(e_\lambda)_{\lambda\in\Lambda}$ and $e$. Suppose moreover that for all $\lambda\le\mu$, $A_\lambda\subseteq A_{\mu}$, and that 
    \[
    A = \mathrm{cl} \left( \bigcup_\lambda A_\lambda \right) .
    \]
    Then $e_\lambda\to e$ pointwise.
\end{proposition}

\begin{lemma}\label{orthogonality}
    Let $X$ be Hilbert space, let $B\subseteq X$ be a closed subset, and let $e_B:X\to X$ be the orthogonal projection onto $B$. 
    Then for every $x\in X$, the projection $e_B(x)$ is the closest point on $B$ to $x$, i.e.
    \[
    d\big( x, e_B(x) \big) \;=\; \inf_{b\in B} d(x, b) .
    \]
\end{lemma}
\begin{proof}[Proof of \Cref{orthogonality}]
    Since $e_B$ is an orthogonal projection (see \Cref{orthogonal_projection}), we have that for all $b\in B$, the vectors $x-e_B(x)$ and $b-e_B(x)$ are orthogonal. 
    Therefore, for all $b\in B$,
    \begin{align*}
    \| x - b \|^2 \;&=\; \left\| \big(x-e_B(x)\big) - \big(b-e_B(x)\big) \right\|^2 \\
    &=\; \left\| x-e_B(x) \right\|^2 + \left\| b-e_B(x) \right\|^2 \\
    &\ge\; \left\| x-e_B(x) \right\|^2 .
    \end{align*}
    So, for all $b\in B$, $d\big( x, e_B(x) \big) \;=\; d(x, b)$.
\end{proof}

\begin{proof}[Proof of \Cref{{up_levi_hilb}}]
    Given $x\in A$, note that $e(x)\in A$. 
    Now for all $\lambda$, by \Cref{orthogonality} (setting $B=A_\lambda$ and replacing $x$ by $e(x)$),
\begin{align*}
    d\big( e(x), e_\lambda(x) \big) &= d\big( e(x), e_\lambda(e(x)) \big) \\
    &= \inf_{a_\lambda\in A_\lambda} d\big(e(x), a_\lambda\big) ,
\end{align*}
which tends to zero since $e(x)$ is in the closure of the $A_\lambda$. 
\end{proof}

\begin{proposition}[downward Levi property for $\cat{Hilb}$]\label{down_levi_hilb}
    Let $X$ be a Hilbert space. Let $(A_\lambda)_{\lambda\in\Lambda}$ and $A$ be closed subspaces of $X$, with corresponding orthogonal projectors $(e_\lambda)_{\lambda\in\Lambda}$ and $e$. Suppose moreover that for all $\lambda\le\mu$, $A_\lambda\supseteq A_{\mu}$, and that 
    \[
    A = \bigcap_\lambda A_\lambda .
    \]
    Then $e_\lambda\to e$ pointwise.
\end{proposition}

We prove the statement by means of the following lemma, which we learned from Matthew Di Meglio \cite{dimeglio2024hilbert}.

\begin{lemma}\label{colim_matt}
    Let $(A_\lambda)_{\lambda\in\Lambda}$ be a filtered diagram of Banach spaces and surjective, linear 1-Lipschitz maps $f_{\lambda,\mu}:A_\lambda\to A_\mu$.
    Then the colimit in $\cat{Ban}_{\le 1}$ is the space $A$ given by forming the quotient of any of the $A_\lambda$, under the seminorm
    \[
    a\longmapsto \lim_{\mu\ge\lambda} \| \pi_{\lambda,\mu}(a) \|_{A_\mu} ,
    \]
    with colimiting cocone formed by the quotient maps $q_\lambda:A_\lambda\to A$.
\end{lemma}

By ``quotienting under a seminorm'', here, we mean forming the universal quotient vector space on which the seminorm induces a norm (similarly to how we obtain $L^p$ spaces from $\mathcal{L}^p$ spaces).

\begin{proof}[Proof of \Cref{colim_matt}]
    First of all, the limit in the definition of the norm converges, since it is a nonincreasing sequence (the $\pi_{\lambda,\mu}$ are 1-Lipschitz) bounded below by zero. Being a limit of seminorms, it is itself a seminorm.
    Moreover, by filteredness, this seminorm does not depend on the choice of $\lambda$ we start with.

    Consider now a Banach space $B$ and 1-Lipschitz linear maps $c_\lambda:A_\lambda\to B$ for all $\lambda$ forming a cocone, i.e.~such that for all $\mu\ge\lambda$, the outer triangle in the diagram below commutes.
    \[
    \begin{tikzcd}
        A_\lambda \ar{dr}[swap]{q_\lambda} \ar{drr}{c_\lambda} \ar{dd}[swap]{\pi_{\lambda,\mu}} \\
        & A \ar[dashed]{r} & B \\
        A_\mu \ar{ur}{q_\mu} \ar{urr}[swap]{c_\mu}
    \end{tikzcd}
    \]
    Define now a map $c:A\to B$ as follows. If we construct $A$ as a quotient of $A_\lambda$, its elements are equivalence classes $[a_\lambda]$, with $a_\lambda\in A_\lambda$, quotiented by the seminorm. 
    Define now $c([a_\lambda])=c_\lambda(a_\lambda)$. 

    To show that this map is well defined on equivalence classes, suppose that $a_\lambda$ and $b_\lambda$ are in the same equivalence class, which means that 
    \[
    \lim_{\mu\ge\lambda} \| \pi_{\lambda,\mu}(a_\lambda-b_\lambda) \|_{A_\mu} \;=\; 0 .
    \]
    Then for every $\mu\ge\lambda$,
    \begin{align*}
        \|c_\lambda(a_\lambda) - c_\lambda(b_\lambda) \|_{B} \;&=\; \|c_\lambda(a_\lambda - b_\lambda) \|_{B} \\
        &=\; \|c_\mu(\pi_{\lambda,\mu}(a_\lambda - b_\lambda)) \|_{B} \\
        &\le\; \|\pi_{\lambda,\mu}(a_\lambda - b_\lambda) \|_{A_\mu} ,
    \end{align*}
    and so 
    \[
    \|c_\lambda(a_\lambda) - c_\lambda(b_\lambda) \|_{B} \;\le\; \lim_{\mu\ge\lambda} \| \pi_{\lambda,\mu}(a_\lambda-b_\lambda) \|_{A_\mu} \;=\; 0 ,
    \]
    which means that $c_\lambda(a_\lambda) = c_\lambda(b_\lambda)$, i.e.~the map $c$ is well defined on equivalence classes.
    A similar calculation shows that it is 1-Lipschitz, and linearity is easily checked.

    By construction of the map $c$, we have $c\circ q_\lambda=c_\lambda$.
    Uniqueness of the map, and hence the universal property, follows from the fact that the quotient map $q_\lambda$ is epimorphic.

    To show that this quotient construction does not depend on the choice of $\lambda$, notice first of all that for all $\mu\ge\lambda$, defining $c$ in terms of $\mu$ instead of $\lambda$ gives the same quotient (since the $\mu'\ge\mu$ form a subnet which converges to the same limit).
    Now by directedness, given $\lambda,\lambda'\in\Lambda$ we can find $\mu\in\Lambda$ such that $\mu\ge \lambda,\lambda'$. 
\end{proof}

\begin{proof}[Proof of \Cref{down_levi_hilb}]
    First of all, by \Cref{hilbdownp} we know that $A$ is the colimit in $\cat{Hilb}_{\le 1}$ of the spaces $A_\lambda$ and their projections.
    Let $x\in X$. Notice that if $\pi:X\to A$ is the projection (this time considered with full image, unlike $e:X\to X$), we have that $\pi(e(x))=\pi(x)$, and so, by \Cref{colim_matt} (taking $A_\lambda=X$), we have that $x$ and $e(x)$ are in the same equivalence class induced by the limiting norm:
    \[
    \lim_{\mu} \| \pi_{\mu} (x - e(x) ) \| \;=\; 0 .
    \]
    Now since $A\subseteq A_\mu$, $e_\mu\circ e = e$, and using that the inclusions $\iota_\mu:A_\mu\to X$ are isometries,
    \begin{align*}
        \lim_\mu \| e_\mu(x) - e(x) \| \;&=\; \lim_\mu \| e_\mu(x) - e_\mu(e(x)) \| \\
        &=\; \lim_\mu \| e_\mu( x - e(x) ) \| \\
        &=\; \lim_\mu \| \iota_\mu(\pi_\mu( x - e(x) ) ) \| \\
        &=\; \lim_\mu \| \pi_\mu( x - e(x) ) \| \\
        &=\; 0 . \qedhere
    \end{align*}
\end{proof}

Note that the proof uses \Cref{hilbdownp}, which holds for Hilbert spaces, but not for general Banach spaces. Indeed, for general Banach spaces the statement fails, as the following example shows.

\begin{example}\label{non_lim_ban}
    Let $X=\ell^\infty(\N)$.
    Consider the decreasing sequence of closed subsets,
    \[
    A_i \;=\; \{f\in\ell^\infty(\N) : \forall x\le i, f(x)=0 \} .
    \]
    These subspaces admit projectors, given by
    \[
    e_i(f)(x) \;=\; \begin{cases}
        0 & x \le i ; \\
        f(x) & x > i .
    \end{cases}
    \]
    We have that $\bigcap_i A_i=\{0\}$, However, $e_i(f)$ does not tend to zero in general. Indeed, for $f=1$ identically,
    $$
    \| e_i(f) \| \;=\; \sup_x |e_i(f)(x)| \;=\; 1 
    $$
    for all $i\in\N$.

    Indeed, while $\bigcap_i A_i=\{0\}$, it is not true that $\{0\}$ is the colimit of the projections: 
    by \Cref{colim_matt}, the colimit would be given by the quotient under the seminorm $f\longmapsto \lim_{i} \| e_i(f) \|$.
    By the calculation above, however, for $f=1$ the resulting seminorm gives $1$, not zero.
\end{example}

\subsection{The idempotent Levi property of EKrn and Levy's theorems}
\label{sec_levi_krn}

The idempotent Levi property holds in $\cat{EKrn}$ with our topological enrichment, and this can be considered a categorical formalization of Levy's (not Levi's) upward and downward theorems in mean. 

\begin{theorem}\label{levi_krn}
    The category $\cat{EKrn}$ has the idempotent Levi property.
\end{theorem}

\begin{proof}
    Let $(X,\mathcal{A},p)$ be essentially standard Borel, and let $(e_\lambda)_{\lambda\in\Lambda}$ be a net of almost surely idempotent kernels.
    We know that this net is bounded above and below.
    Also, by \Cref{hilb_homeo}, we have that $e_\lambda$ tends (either in the one-sided or two-sided topology) to an idempotent $e$ if and only if the operators $e_\lambda^*:L^2(X,\mathcal{A},p)\to L^2(X,\mathcal{A},p)$ tend to $e^*$ in $\cat{Hilb}$.
    \begin{itemize}
        \item If the net is increasing, we know it has a supremum in $\cat{EKrn}$, given by the join sigma-algebra. Call this supremum $e_\infty$. 
        By \Cref{preserves_optima}(i), $e_\infty^*:L^2(X,\mathcal{A},p)\to L^2(X,\mathcal{A},p)$ is the supremum of the $e_\lambda^*$ in $\cat{Hilb}$ as well.
        Now by the upward Levi property for $\cat{Hilb}$ (\Cref{up_levi_hilb}), $e_\lambda^*$ tends to $e_\infty^*$.
        Therefore $e_\lambda\to e_\infty$.

        \item Similarly, if the net is decreasing, we know it has an infimum in $\cat{EKrn}$, given by the meet sigma-algebra (intersection of the null-set completions). Call this infimum $e_\infty$. 
        By \Cref{preserves_optima}(ii), $e_\infty^*:L^2(X,\mathcal{A},p)\to L^2(X,\mathcal{A},p)$ is the infimum of the $e_\lambda^*$ in $\cat{Hilb}$ as well.
        Now by the downward Levi property for $\cat{Hilb}$ (\Cref{down_levi_hilb}), $e_\lambda^*$ tends to $e_\infty^*$.
        Therefore $e_\lambda\to e_\infty$. \qedhere
    \end{itemize}
\end{proof}

As corollaries, by looking at the actions of idempotents on random variables, we get Levy's upward and downward theorems, for convergence in mean:

\begin{corollary}[Levy's upward theorem]\label{Levy_upward}
    Let $(X,\mathcal{A},p)$ be essentially standard Borel, let $(\mathcal{B}_\lambda)_{\lambda\in\Lambda}$ be an increasing filtration, and let $f\in L^n(X,\mathcal{A},p)$.
    Denote the join $\bigvee_{\lambda\in\Lambda} \mathcal{B}_\lambda$ by $\mathcal{B}_\infty$. Then $\E[f|\mathcal{B}_\lambda]$ converges to $\E[f|\mathcal{B}_\infty]$ in $L^n$.
\end{corollary}

\begin{proof}
First of all, the increasing filtration $(\mathcal{B}_\lambda)_{\lambda\in \Lambda}$ induces an increasing net of dagger idempotents $(e_\lambda:X\to X)_{\lambda\in \Lambda}$. Since the join of the filtration is given by $\mathcal{B}_\infty$, the join of the dagger idempotents is given by the dagger idempotent $e_\infty:X\to X$ induced by $\mathcal{B}_\infty$.

Because $\cat{EKrn}$ satisfies the idempotent Levi property (\Cref{levi_krn}), it follows that the net $(e_\lambda)_{\lambda\in \Lambda}$ converges topologically to $e_\infty$ in $\cat{EKrn}((X,\mathcal{A},p),(X,\mathcal{A},p))$, using either the one-sided or two-sided topology (they agree on idempotents). Since $L^n$ is an enriched functor (\Cref{Ln_enriched}), we see that $(e^*_\lambda)_{\lambda\in \Lambda}$ converges to $e^*_\infty$ in $\cat{Ban}(L^n(X,\mathcal{A},p),(X,\mathcal{A},p))$. This exactly means that \[ \mathbb{E}[f\mid \mathcal{B}_\lambda]=e^*_\lambda f\to e^*_\infty f=\mathbb{E}[f\mid B_\infty]\] in $L^n$ for every $f\in L^n(X,\mathcal{A},p)$.
\end{proof}

\begin{corollary}[Levy's downward theorem]
    Let $(X,\mathcal{A},p)$ be essentially standard Borel, let $(\mathcal{B}_\lambda)_{\lambda\in\Lambda}$ be a decreasing filtration, and let $f\in L^n(X,\mathcal{A},p)$.
    Denote the meet $\bigwedge_{\lambda\in\Lambda} \mathcal{B}_\lambda$ (intersection of the null-set completions) by $\mathcal{B}_\infty$. Then $\E[f|\mathcal{B}_\lambda]$ converges to $\E[f|\mathcal{B}_\infty]$ in $L^n$.
\end{corollary}

The proof is analogous to the one of the previous corollary.

We can now combine these theorems with the results at the end of \Cref{sec_martingales}, to obtain the following \emph{martingale convergence theorems} in mean.

\begin{corollary}[Martingale convergence theorem in mean]\label{martin_up}
    Let $(X,\mathcal{A},p)$ be essentially standard Borel, let $(\mathcal{B}_\lambda)_{\lambda\in\Lambda}$ be an increasing filtration, and let $(f_\lambda)_{\lambda\in\Lambda}$ be a martingale where either
    \begin{itemize}
        \item the $f_\lambda$ admit a uniform $L^n$ bound for $n\ge 2$; or
        \item there exists $f\in L^1(X,\mathcal{A},p)$ such that for eventually all $\lambda$, $f_\lambda=\E[f|\mathcal{B}_\lambda]$.
    \end{itemize}
    Denote the join $\bigvee_{\lambda\in\Lambda} \mathcal{B}_\lambda$ by $\mathcal{B}_\infty$. Then there exists $f_\infty\in L^n(X,\mathcal{B}_\infty,p)$ such that
    \begin{itemize}
        \item $f_\lambda=\E[f_\infty|\mathcal{B}_\lambda]$ for all $\lambda$;
        \item $f_\lambda$ converges to $f_\infty$ in $L^n$.
    \end{itemize}
\end{corollary}

\begin{proof}
Let us first consider the first case. The increasing filtration induces a cofiltered diagram $\mathcal{B}:\Lambda\to \cat{EKrn}$. By applying the dagger functor and then the $L^2$ functor we obtain a filtered diagram $D$: \[\Lambda\xrightarrow{\mathcal{B}}\cat{EKrn}\xrightarrow{\dagger} \cat{EKrn}^\text{op}\xrightarrow{L^2}\cat{Hilb}_{\leq 1}.\] This means that $D(\lambda)=L^2(X,\mathcal{B}_\lambda,p)$ and $D(\lambda<\mu)=(\pi_{\lambda,\mu}^\dagger)^*$, where $\pi_{\lambda,\mu}: (X,\mathcal{B}_\mu,p)\to (X,\mathcal{B}_\lambda,p)$ is the identity on sets. By \Cref{Levy_upward} the limit of $D$ is given by $L^2(X,\mathcal{B}_\infty,p)$ with projection maps $(\pi_\lambda^\dagger)^*$ for every $\lambda\in \Lambda$.

Now for every $\lambda\in\Lambda$ and for $n\ge 2$, $\|f_\lambda\|_{L^2}\le \|f_\lambda\|_{L^n}$, and so the uniform $L^n$ bound (that we have by hypothesis) also gives us a uniform $L^2$ bound. In other words, we can define a real number $r:=\sup_\lambda\mathbb{E}[f_\lambda^2]$. For $\lambda\in \Lambda$, the assignment $r\mapsto f_\lambda$ defines a map $\varphi_\lambda:\mathbb{R}\to L^2(X,\mathcal{B}_\lambda,p)$ in $\cat{Hilb}_{\leq 1}$. Because $(f_\lambda)_{\lambda\in\Lambda}$ is a martingale we have that $\mathbb{R}$ together with the maps $(\varphi_\lambda)_{\lambda\in \Lambda}$ form a cone over the diagram $L^2$. Therefore there exists a unique map $\varphi_\infty:\mathbb{R}\to L^2(X,\mathcal{B}_\infty,p)$ in $\cat{Hilb}_{\leq 1}$ such that $(\pi_\lambda^\dagger)^*\varphi_\infty=\varphi_\lambda$ for every $\lambda\in \Lambda$. If we define $f_\infty:=\varphi_\infty(r)$, then \[\mathbb{E}[f_\infty\mid \mathcal{B}_\lambda]=(\pi_\lambda^\dagger)^*\varphi_\infty(r)=\varphi_\lambda(r)=f_\lambda.\]
By \Cref{Levy_upward}, it now follows that $f_\lambda\to f_\infty$ in $L^n$. 

The second case follows from Corollary \ref{marti_hilb}, from the tower property of conditional expectation (taking $f_\infty=\mathbb{E}[f\mid \mathcal{B}_\infty]$), and again by \Cref{Levy_upward}.
\end{proof}

\begin{corollary}[Backward martingale convergence theorem in mean]\label{martin_down}
    Let $(X,\mathcal{A},p)$ be essentially standard Borel, let $(\mathcal{B}_\lambda)_{\lambda\in\Lambda}$ be a decreasing filtration, and let $(f_\lambda)_{\lambda\in\Lambda}$ be a martingale with $f_\lambda\in L^n(X,\mathcal{B}_\lambda,p)$ for $n\ge 1$.
    Denote the meet $\bigwedge_{\lambda\in\Lambda} \mathcal{B}_\lambda$ (intersection of the null-set completions) by $\mathcal{B}_\infty$. Then there exists $f_\infty\in L^n(X,\mathcal{B}_\infty,p)$ such that
    \begin{itemize}
        \item $f_\infty=\E[f_\lambda|\mathcal{B}_\infty]$ for all $\lambda$;
        \item $f_\lambda$ converges to $f_\infty$ in $L^n$.
    \end{itemize}
\end{corollary}

The proof is analogous to the one of the second case of \Cref{martin_up}. 

This version of the backward martingale convergence theorem, for decreasing filtrations indexed by arbitrary nets, seems to be new.

\section{Vector-valued martingales}
\label{sec_V}

Our abstract formalism allows to talk about martingales in a very general way, including vector-values ones, and gives us a convergence result for all such classes of martingales. Vector-valued martingales and their convergence can be used to describe approximations of stochastic processes such as Brownian motion. 

Here is the definition in general.

\begin{definition}
    Let $F:\cat{EKrn}^\op\to\cat{Ban}$ be any functor. 
    \begin{itemize}
        \item An \newterm{$F$-valued random variable} on $(X,\mathcal{A},p)$ is an element of the Banach space $F(X,\mathcal{A},p)$. (Equivalently, a linear map $\R\to F(X,\mathcal{A},p)$).
        \item An \newterm{$F$-valued martingale} on $(X,\mathcal{A},p)$ consists of an increasing filtration $(\mathcal{B}_\lambda)_{\lambda\in\Lambda}$ on $(X,\mathcal{A},p)$, together with a family of $F$-valued random variables $(f_\lambda)_{\lambda\in\Lambda}$ such that for all $\lambda\le\mu$, $f_\lambda= F(\pi^+_{\mu,\lambda})(f_\mu)$.
        (Equivalently, with a cone in $\cat{Ban}$ on the $F(X,\mathcal{B}_\lambda,p)$ with tip $\R$.)
        \item An \newterm{$F$-valued backward martingale} is defined analogously, but with a decreasing filtration.
    \end{itemize}
\end{definition}

Beside the ordinary RV functors of \Cref{RV}, we will also give the example of vector-valued random variables in the Bochner (strong) sense.
One could generalize this even further, replacing the category $\cat{Ban}$ with any other category, but we will not do that here.

Thanks to how we have set up the formalism, whenever the functor $F$ is enriched we automatically have a version of the Levy theorems:

\begin{theorem}\label{general_levy}
    Let $F:\cat{EKrn}^\op\to\cat{Ban}$ be an enriched functor, let $(\mathcal{B}_\lambda)_{\lambda\in\Lambda}$ be an increasing filtration on a standard Borel space $(X,\mathcal{A},p)$, and let $f\in F(X,\mathcal{A},p)$.
    We have that 
    \begin{enumerate}
        \item The $F$-images of the retracts 
        \[
        \begin{tikzcd}
            (X,\mathcal{A},p) \ar[shift left]{r}{\pi_\lambda} & (X,\mathcal{B}_\lambda,p) \ar[shift left]{l}{\pi^+_\lambda}
        \end{tikzcd}
        \]
        are retracts 
        \[
        \begin{tikzcd}
            F(X,\mathcal{A},p) \ar[leftarrow, shift left]{r}{F(\pi_\lambda)} & F(X,\mathcal{B}_\lambda,p) \ar[leftarrow, shift left]{l}{F(\pi^+_\lambda)}
        \end{tikzcd}
        \]
        in $\cat{Ban}$, where $F(\pi_\lambda)$ is a projection and $F(\pi^+_\lambda)$ an inclusion, and the order of split idempotents is preserved;
        \item If we denote by $f_\lambda$ the element $F(\pi^+_\lambda)(f)$, we have that $(f_\lambda)_{\lambda\in\Lambda}$ is an $F$-valued martingale;
        \item The $F$-image of the supremum 
        \[
        \begin{tikzcd}
            (X,\mathcal{A},p) \ar[shift left]{r}{\pi_\infty} & (X,\mathcal{B}_\infty,p) \ar[shift left]{l}{\pi^+_\infty}
        \end{tikzcd}
        \]
        is the supremum in the split idempotent order on $F(X,\mathcal{A},p)$ in $\cat{Ban}$;
        \item If we denote by $f_\infty$ the element $F(\pi^+_\infty)(f)$, we have that $f_\lambda$ tends to $f_\infty$ topologically (in the norm of $F(X,\mathcal{A},p)$).
    \end{enumerate}

    The analogous statement is true for decreasing filtrations and backward martingales.
\end{theorem}

\begin{proof}
    We will prove the statement for increasing filtrations and forward martingales, the backward case is analogous.
    \begin{enumerate}
        \item By (contravariant) functoriality, 
        \[
        F(\pi^+_\lambda)\circ F(\pi_\lambda) \;=\; F(\pi_\lambda\circ \pi^+_\lambda) \;=\; F(\id_{(X,\mathcal{B}_\lambda,p)}) \;=\; \id_{F(X,\mathcal{B}_\lambda,p)} ,
        \]
        so that we have a retract.
        To see that the order is preserved, denote by $e_\lambda$ the idempotent induced by $\mathcal{B}_\lambda$. If $e_\lambda\le e_\mu$, i.e.~$e_\lambda\circ e_\mu=e_\mu\circ e_\lambda=e_\lambda$, again by functoriality we have that
        \[
        F(e_\lambda)\circ F(e_\mu) \;=\; F(e_\mu\circ e_\lambda) \;=\; F(e_\lambda) ,
        \]
        and analogously $F(e_\mu)\circ F(e_\lambda)=F(e_\lambda)$, so that $F(e_\lambda)\le F(e_\mu)$. 
        \item Once again by functoriality, for every $\lambda\le\mu$, 
        \[
        F(\pi^+_{\mu,\lambda})(f_\mu) \;=\; F(\pi^+_{\mu,\lambda})\circ F(\pi^+_{\mu}) (f) \;=\; F(\pi^+_{\mu}\circ \pi^+_{\mu,\lambda}) \;=\; F(\pi^+_\lambda)(f) \;=\; f_\lambda .
        \]
        \item Denote by $e_\infty$ the idempotent induced by $\mathcal{B}_\infty$. 
        By \Cref{levi_krn}, $e_\lambda\to e_\infty$ topologically in $\cat{EKrn}\big( (X,\mathcal{A},p), (X,\mathcal{A},p) \big)$.
        Since the functor $F$ is enriched, then $F(e_\lambda)\to F(e_\infty)$ topologically in $\cat{Ban}\big( F(X,\mathcal{A},p), F(X,\mathcal{A},p) \big)$. 
        Now by \Cref{closedorder} and \Cref{limitclosedorder} it follows that $F(e_\infty)$ is a supremum of the $F(e_\lambda)$.
        \item From the point above, we have that $f_\lambda = F(e_\lambda) (f) \to F(e_\infty) (f) = f_\infty$ topologically. \qedhere
    \end{enumerate}
\end{proof}

Let's now apply this to vector-valued martingales. 

\subsection{Bochner integrals and conditional expectations}
\label{sec_bochner}

In what follows we will consider random variables with values in a Banach space $V$.
The case where $V$ is separable is particularly simple, but we will not assume it here. 
(Note that a separable Banach space is in particular a complete separable metric space, so that with its Borel sigma-algebra it is a standard Borel space.) 

Where possible, we will follow the terminology and conventions of \cite[Appendix~E]{cohn}.

\begin{definition}
    Let $(X,\mathcal{A})$ be a measurable space, and let $V$ be a Banach space.
    \begin{itemize}
        \item A \newterm{$V$-valued random variable} is an $\mathcal{A}$-measurable function $f:X\to V$, where $V$ is taken with the Borel sigma-algebra.
        \item A $V$-valued RV $f$ on $X$ is called \newterm{strongly measurable} if it is measurable and its image $f(X)\subseteq V$ is separable, and \emph{almost surely strongly measurable} if it is almost surely equal to a strongly measurable RV.
        \item A $V$-valued RV $f$ on $X$ is called \newterm{simple} if it is in the form
    \[
    f(x) \;=\; \sum_{i=1}^n 1_{A_i}(x)\cdot v_i ,
    \]
    where $v_i\in V$ and $A_i\in\mathcal{A}$ for all $i=1,\dots,n$.
    \end{itemize}
\end{definition}
Every simple RV is strongly measurable. Also, if $V$ is separable, measurability and strong measurability coincide.
Just as ordinarily measurable RVs, strongly measurable RVs are closed under linear combinations \cite[Proposition~E.3]{cohn} and under pointwise limits \cite[Proposition~E.1]{cohn}.

\begin{definition}
    In analogy with the real case, we call an a.s.~strongly measurable $V$-valued random variable $f$ \newterm{Bochner-integrable} if and only if 
    \[
    \int_X \|f(x)\|_V \,p(dx) \;<\; \infty ,
    \]
    and \newterm{Bochner-$n$-integrable}, for $n\ge 1$, if and only if 
    \[
    \int_X (\|f(x)\|_V)^n \,p(dx) \;<\; \infty .
    \]
\end{definition}

We denote the set of a.s.~equivalence classes of (a.s.~strongly measurable) Bochner-$n$-integrable random variables on $(X,\mathcal{A},p)$ by $L^n(X,\mathcal{A},p;V)$. 
These spaces are known in the literature under the name of \emph{Bochner spaces}. 
They are Banach spaces with norm
\[
\|f\|_{L^n(X,\mathcal{A},p;V)} \;=\; \left( \int_X (\|f(x)\|_V)^n \,p(dx) \right)^{1/n} ,
\]
analogously to usual $L^n$ spaces.
Similarly, we denote by $L^\infty(X,\mathcal{A},p;V)$ the set of equivalence classes of almost surely bounded $V$-valued random variables, with the essential supremum norm. Just as in the real-valued case, we have that for $m\ge n$, $L^m(X,\mathcal{A},p;V)\subseteq L^n(X,\mathcal{A},p;V)$.

It can be shown that if a RV $f$ is Bochner-integrable, then there exists a sequence $(f_n)$ of simple $V$-valued RVs 
\[
f_n(x) \coloneqq \sum_{i=1}^{k_n} 1_{A_{n,i}}(x)\cdot v_{n,i}
\]
tending pointwise to $f$ and such that for all $x$ and $n$, $\|f_n(x)\|_V\le\|f(x)\|_V$ \cite[Proposition~E.2]{cohn}.
In that case, the \newterm{Bochner integral} of $f$ is defined by 
\[
\int_X f(x)\, p(dx) \;\coloneqq\; \lim_{n\to\infty} \int_X f_n(x)\,p(dx) \;\coloneqq\; \lim_{n\to\infty} \sum_{i=1}^{k_n} p(A_{n,i}) \cdot v_{n,i} ,
\]
is finite, and is independent of the choice of $(f_n)$  \cite[p.~399]{cohn}.
One can show that Bochner integration is linear in $f$, and a version of the dominated convergence theorem holds.

We now want to construct functors $L^n(-;V):\cat{GKrn}^\op\to\cat{Ban}$ analogous to the $L^n$ functors defined in \Cref{sec_RV}:
\begin{itemize}
    \item On objects, we map a probability space $(X,\mathcal{A},p)$ to the Banach space $L^n(X,\mathcal{A},p;V)$;
    \item On morphisms, given a measure-preserving kernel $k:(X,\mathcal{A},p)\to(Y,\mathcal{B},q)$, we get a bounded linear map $k^*:L^n(Y,\mathcal{B},q;V)\to L^n(X,\mathcal{A},p;V)$ acting on random variables $g\in L^n (Y,\mathcal{B},q;V)$ by
    \begin{equation}\label{defkstar}
    k^*g (x) \coloneqq \int_Y g(y)\,k(dy|x) .
    \end{equation}
\end{itemize}

Here is the precise statement, the analogous of \Cref{RV}.

\begin{proposition}\label{RVV}
    Let $k:(X,\mathcal{A},p)\to(Y,\mathcal{B},q)$ be a measure-preserving kernel, and let $g\in L^n(Y,\mathcal{B},q;V)$, for $n$ finite or infinite. Then the assignment 
    \[
    x\longmapsto k^*g (x) \coloneqq \int_Y g(y)\,k(dy|x) 
    \]
    is a well-defined element of $L^n(X,\mathcal{A},p;V)$. 

    Moreover, the assignment $k^*:L^n(Y,\mathcal{B},q;V)\to L^n(X,\mathcal{A},p;V)$ given by $g\mapsto k^*g$ is linear and 1-Lipschitz, so that $L^n$ is a functor $\cat{GKrn}^\op\to\cat{Ban}$.
\end{proposition}

\begin{lemma}\label{stronglysmult}
    Let $h:(X,\mathcal{A})\to\R$ be a measurable function, and let $v\in V$.
    Then the function 
    \[
    \begin{tikzcd}[row sep=0]
        (X,\mathcal{A}) \ar{r}{h\cdot v} & V \\
        x \ar[mapsto]{r} & h(x)\cdot v
    \end{tikzcd}
    \]
    is strongly measurable.
\end{lemma}
\begin{proof}[Proof of \Cref{stronglysmult}]
    Let $(h_n)$ be a sequence of simple (real-valued) functions 
    \[
    h_n(x) \coloneqq \sum_{i=1}^{k_n} \lambda_{n,i}\,1_{A_{n,i}}(x) 
    \]
    tending pointwise to $h$. Since the scalar multiplication 
    \[
    \begin{tikzcd}[row sep=0]
        \R\times V \ar{r} & V \\
        (\lambda,v) \ar[mapsto]{r} & \lambda\cdot v
    \end{tikzcd}
    \]
    is continuous, we have that the functions
    \[
    x \longmapsto h_n(x)\cdot v \;=\; \sum_{i=1}^{k_n} \lambda_{n,i}\,1_{A_{n,i}}(x) \cdot v
    \]
    tend pointwise to $h\cdot v$. Since strongly measurable functions are closed under linear combinations and pointwise limits, we only have to show that for each $n$ and $i$, the function $x\mapsto 1_{A_{n,i}}(x) \cdot v$ is strongly measurable. 
    But this is the case: the function is measurable (since the set $A_{n,i}$ is), and it assumes at most only two values ($0$ and $v$).  
\end{proof}

\begin{proof}[Proof of \Cref{RVV} for finite $n$]
Since $k$ is measure-preserving, by approximation via simple (real-valued) functions,
\begin{equation}\label{meas_pres}
    \int_X \int_Y (\|g(y)\|_V)^n \,k(dy|x)\,p(dx) \;=\; \int_Y (\|g(y)\|_V)^n\,q(dy) \;<\;\infty ,
\end{equation}
and so the integrals
\[
\int_Y (\|g(y)\|_V)^n \,k(dy|x) 
\]
are finite for $p$-almost all $x\in X$. Even if $n>1$, again by Jensen's inequality we have that 
\[
\left( \int_Y \|g(y)\|_V \,k(dy|x) \right)^n \;\le\; \int_Y (\|g(y)\|_V)^n \,k(dy|x) \;<\; \infty .
\]
Therefore the Bochner integral
\begin{equation}\label{defkstar}
x\longmapsto \int_Y g(y) \,k(dy|x) ,
\end{equation}
is defined for $p$-almost all $x\in X$, and the assignment specifies a unique function $X\to V$ up to a null set.
To see that this function is a.s.~strongly measurable, let $(g_n)_{n=1}^\infty$ be a sequence of simple $V$-valued RVs 
\[
g_n(y) \coloneqq \sum_{i=1}^{k_n} 1_{B_{n,i}}(y)\cdot v_{n,i}
\]
tending pointwise to $g$ and such that $\|g_n(y)\|_V\le\|g(y)\|_V$. (The sequence exists since $g$ is Bochner-integrable.) 
Hence for almost all $x$,
\[
\int_Y g(y) \,k(dy|x) \;=\; \lim_{n\to\infty} \int_Y g_n(y)\,k(dy|x) 
\;=\; \lim_{n\to\infty} \sum_{i=1}^{k_n} k(B_{n,i}|x) \cdot v_{n,i} ,
\]
and so by \Cref{stronglysmult} the assignment \eqref{defkstar} is a pointwise limit of strongly measurable functions.
Therefore the assignment \eqref{defkstar} gives a well-defined element of $L^n(X,\mathcal{A},p;V)$. 
We denote this element by $k^*g$.

In other words, given $k:(X,\mathcal{A},p)\to(Y,\mathcal{B},q)$, we get a well-defined function $k^*:L^n(Y,\mathcal{B},q;V)\to L^n(X,\mathcal{A},p;V)$ given by $g\mapsto k^*g$.  This function is linear by linearity of integration.
To see that it is 1-Lipschitz, once again by Jensen,
\begin{align*}
(\| k^*g \|_{L^n(X,\mathcal{A},p;V)})^n \;&=\; \int_X \left(\left\| \int_Y g(y) \,k(dy|x) \right\|_V\right)^n \, p(dx) \\
&\le\; \int_X \left( \int_Y \|g(y)\|_V \,k(dy|x) \right)^n \, p(dx) \\
&\le\; \int_X \int_Y (\|g(y)\|_V)^n \,k(dy|x)\,p(dx) \\
&=\; \int_Y (\|g(y)\|_V)^n\,q(dy) \\
&=\; (\|g\|_{L^n(Y,\mathcal{B},q;V)})^n . \qedhere
\end{align*}
\end{proof}

\begin{proof}[Proof of \Cref{RVV} for infinite $n$]
Let $g\in L^\infty(Y,\mathcal{B},q;V)$. By definition there exists $B\in\mathcal{B}$ of $q$-measure one such that $g$ restricted to $B$ is bounded. Let $K$ be an upper bound for the norm $\|g|_B\|_V$.
Since $k$ is measure-preserving, 
\begin{align*}
    \int_X \int_Y \|g(y)\|_V \,k(dy|x)\,p(dx) \;=\; \int_Y \|g(y)\|_V\,q(dy) \;=\; \int_B \|g(y)\|_V\,q(dy) \;\le\; K .
\end{align*}
Therefore the integrals
\[
\int_Y \|g(y)\|_V \,k(dy|x) 
\]
are finite for $p$-almost all $x\in X$, and so the Bochner integral
\[
x\longmapsto \int_Y g(y) \,k(dy|x) ,
\]
is defined for $p$-almost all $x\in X$. Denote this value by $k^*g(x)$. As in the case of finite $n$, this is strongly measurable in $X$.

Now by \Cref{preserveone}, there exists $A\in\mathcal{A}$ of $p$-measure one such that for all $x\in A$, $k(B|x)=1$.
Therefore, for all $x\in A$,
\[
\|k^*g(x)\|_V \;=\; \left\| \int_Y g(y) \,k(dy|x) \right\|_V \;\le\; \int_Y \|g(y)\|_V \,k(dy|x) \;=\;\int_B \|g(y)\|_V \,k(dy|x) \;\le\; K .
\]
In other words, the restriction of $k^*g$ to $A$ is also bounded by $K$. So we get a well-defined function $k^*:L^\infty(Y,\mathcal{B},q;V)\to L^\infty(X,\mathcal{A},p;V)$. Moreover, since every essential upper bound for $\|g\|_V$ is an essential upper bound for $\|k^*g\|_V$, we have that 
\[
\|k^*g\|_{L^\infty} \;=\; \mathrm{ess}\sup \|k^*g\|_V \;\le\; \mathrm{ess}\sup \|g\|_V \;=\; \|g\|_{L^\infty} . \qedhere
\]
\end{proof}

Just as in the real-valued case, the inclusions $L^n(X,\mathcal{A},p;V)\subseteq L^m(X,\mathcal{A},p;V)$ for $m\le n$ form natural transformations.

As an example of the functorial action of $L^n(-;V)$, we can look at \emph{$V$-valued conditional expectations}. 
Indeed, let $(X,\mathcal{A},p)$ be a standard Borel space, and let $\mathcal{B}\subseteq\mathcal{A}$ be a sub-sigma-algebra. 
The disintegration kernel $\pi^+:(X,\mathcal{B},p)\to(X,\mathcal{A},p)$ given by the conditional expectations
\[
\pi^+(A|x) \;=\; \P[A|\mathcal{B}](x) 
\]
induces a linear, 1-Lipschitz map $(\pi^+)^*:L^n(X,\mathcal{A},p;V)\to L^n(X,\mathcal{B},p;V)$
which can be considered a generalization of conditional expectation to vector-valued random variables.
Because of that, let's use the same notation as for real-valued random-variables: we write
\[
\begin{tikzcd}[row sep=0]
L^n(X,\mathcal{A},p;V) \ar{r}{(\pi^+)^*} & L^n(X,\mathcal{B},p;V) \\
f \ar[mapsto]{r} & \E[f|\mathcal{B}] .
\end{tikzcd}
\]
Note that, just as in the real-valued case:
\begin{itemize}
    \item $\E[f|\mathcal{B}]$ is (strongly) $\mathcal{B}$-measurable;
    \item For every $B\in\mathcal{B}$ we have that 
    \[
    \int_B \E[f|\mathcal{B}](x) \,p(dx) \;=\; \int_B f(x) \,p(dx) ,
    \]
    and $\E[f|\mathcal{B}]$ is determined almost surely uniquely by these properties.
\end{itemize}
Also, in this notation, a $V$-valued martingale looks just like a real-valued one: a $V$-valued martingale on $(X,\mathcal{A},p)$ with filtration $(B_\lambda)$ is a family of $V$-valued, strongly $\mathcal{B}_\lambda$-measurable RVs $(f_\lambda)$ such that for all $\lambda\le\mu$, $f_\lambda=\E[f_\mu|\mathcal{B}_\lambda]$.

\subsection{Topological enrichment and convergence}
\label{sec_bochnerenrichment}

Let's now turn to topological enrichment and convergence.

\begin{proposition}
    For every Banach space $V$ and for each $n$, the functor $L^n(-;V):\cat{GKrn}^\op\to\cat{Ban}$ is topologically enriched.
\end{proposition}
\begin{proof}
    Let $(k_\lambda)_{\lambda\in\Lambda}$ be a net of measure-preserving kernels $(X,\mathcal{A},p)\to(Y,\mathcal{B},q)$ tending to $k$.
    That is, for every $B\in\mathcal{B}$,
    \[
    \int_X |k(B|x) - k_\lambda(B|x)|\,p(dx) \to 0.
    \]
    We have to show that the induced maps $k_\lambda^*:L^n(Y,\mathcal{B},q;V)\to L^n(X,\mathcal{A},p;V)$ tend to $k^*$,
    meaning that for all $g\in L^n(Y,\mathcal{B},p;V)$, 
    \[
    \| k^*g - k_\lambda^*g \|_{L^n(X,\mathcal{A},q;V)} \to 0 ,
    \]
    i.e.
    \[
    \int_X \|k^*g (x) - k_\lambda^*g(x) \|_V \,p(dx) \to 0.
    \]
    By \Cref{densitylemma} it suffices to test this convergence on a dense subset of $L^n(Y,\mathcal{B},q;V)$, and we take the set simple $V$-valued RVs, i.e.~those in the following form,
    \[
    g(y) \;=\; \sum_{i=1}^k 1_{B_i}(x)\cdot v_i
    \]
    for $v_i\in V$ and $B_i\in\mathcal{B}$.
    Now 
    \[
        k_\lambda^*g(x) \;=\; \int_Y g(y)\,k(dy|x) \;=\; \sum_{i=1}^k k_\lambda(B_i|x)\cdot v_i ,
    \]
    and the same is true for $k$, so that 
    \begin{align*}
    \int_X \|k^*g (x) - k_\lambda^*g(x) \|_V\, p(dx) \;&=\; \int_X \left\|  \sum_{i=1}^k \big( k(B_i|x) - k_\lambda(B_i|x) \big)\cdot v_i \right\|_V \,p(dx) \\
    &\le\; \int_X \left|  \sum_{i=1}^k k(B_i|x) - k_\lambda(B_i|x) \right|\, \| v_i \|_V \,p(dx) \\
    &\le\; \sum_{i=1}^k \left( \int_X |k(B_i|x) - k_\lambda(B_i|x)|\, p(dx) \right) \, \|v_i\|_V ,
    \end{align*}
    but the term in the brackets tends to zero by assumption. 
    Hence $k_\lambda^*\to k^*g$. 
\end{proof}

Because of this, \Cref{general_levy} can be applied, and it gives us the following result.

\begin{corollary}\label{marti_bochner}
    Let $(X,\mathcal{A},p)$ be a standard Borel probability space, and let $V$ be a Banach space. 

    \begin{enumerate}
        \item Let $(\mathcal{B}_\lambda)_{\lambda\in\Lambda}$ be an increasing filtration and let by $\mathcal{B}_\infty$ its supremum (join of sigma-algebras).
    Given $f\in L^n(X,\mathcal{A},p,V)$, denote by $f_\lambda$ and $f_\infty$ the ($V$-valued) conditional expectations $\E[f|\mathcal{B}_\lambda]$ and $\E[f|\mathcal{B}_\infty]$.
    Then $(f_\lambda)$ is a $V$-valued martingale, and $f_\lambda$ tends to $f_\infty$ in mean, meaning that 
    \[
    \int_X (\| f(x) - f_\lambda(x) \|_V)^n \, p(dx) \to 0 .
    \]
        \item Let $(\mathcal{B}_\lambda)_{\lambda\in\Lambda}$ be a decreasing filtration and let by $\mathcal{B}_\infty$ its infimum (intersection of the null-set-completions).
    Given $f\in L^n(X,\mathcal{A},p,V)$, denote by $f_\lambda$ and $f_\infty$ the ($V$-valued) conditional expectations $\E[f|\mathcal{B}_\lambda]$ and $\E[f|\mathcal{B}_\infty]$.
    Then $(f_\lambda)$ is a $V$-valued backward martingale, and $f_\lambda$ tends to $f_\infty$ in mean, as above.
    \end{enumerate}
\end{corollary}

\appendix

\section{Topologically enriched categories}
\label{sec_top}

The category $\cat{Top}$ is not cartesian closed. It is however still \emph{monoidal} closed, where as tensor product we take the topology of \emph{separate} continuity, instead of joint continuity. 

\subsection{The tensor product}

\begin{definition}
    Let $X$ and $Y$ be topological spaces.
    We call a subset $U\subseteq X\times Y$ \newterm{separately open} if and only if 
    \begin{itemize}
        \item For each $x\in X$, the set $U_x\subseteq Y$ given by
        \[
        U_x \coloneqq \{ y\in Y : (x,y) \in U \}
        \]
        is an open subset of $Y$;
        \item For each $y\in Y$, the set $U_y\subseteq X$ given by
        \[
        U_y \coloneqq \{ x\in X : (x,y) \in U \}
        \]
        is an open subset of $X$.
    \end{itemize}
\end{definition}

\begin{proposition}
    Let $X$ and $Y$ be topological spaces. The separately open subsets of $X\times Y$ form a topology.
\end{proposition}
\begin{proof}
    First of all, the empty set and the set $X\times Y$, seen as subsets of $X\times Y$, are separately open. 
    Let now $U$ and $V$ be separately open. For each $x\in X$ have that 
    \begin{align*}
    (U\cap V)_x &= \{ y\in Y : (x,y) \in U\cap V \} \\
        &= \{ y\in Y : (x,y) \in U, (x,y) \in V \} \\
        &= \{ y\in Y : y\in U_x, y\in V_x \} \\
        &= U_x\cap V_x ,
    \end{align*}
    which is open. The same is true for $(U\cap V)_y$ for each $y\in Y$.
    Given a family $(U_i)_{i\in I}$ of separately open subsets of $X\times Y$, for $x\in X$,
    \begin{align*}
        \left( \bigcup_i U_i \right)_x &= \left\{ y\in Y : (x,y) \in \bigcup_i U_i \right\} \\
        &= \left\{ y\in Y : (x,y) \in U_i \mbox{ for some }i\in I \right\} \\
        &= \bigcup_i U_i \left\{ y\in Y : (x,y) \in U_i \right\} \\
        &= \bigcup_i (U_i)_x ,
    \end{align*}
    which is open. Once again, the same is true for $\left( \bigcup_i U_i \right)_y$.
\end{proof}

\begin{definition}
    We call the topology on $X\times Y$ given by the separately open subsets the \newterm{topology of separate continuity} or the \newterm{tensor product topology}. We denote the resulting topological space by $X\otimes Y$.
\end{definition}

Note that this makes $\cat{Top}$ a monoidal category, with the one-point space as monoidal unit (so it's semicartesian monoidal).

\begin{proposition}
    Let $X$, $Y$ and $Z$ be topological spaces, and let $f:X\times Y\to Z$ be a function. We have that $f$ is separately continuous in $X$ and $Y$ if and only if it is continuous for the topology of separate continuity.
\end{proposition}

For this proof, and for later convenience, given a function $f:X\times Y\to Z$, define 
\begin{itemize}
    \item For each $x\in X$, the function $f_x:Y\to Z$ given by $f_x(y)=f(x,y)$;
    \item For each $y\in Y$, the function $f_y:X\to Z$ given by $f_y(x)=f(x,y)$.
\end{itemize}

\begin{proof}
    The function $f$ is continuous as a function $f:X\otimes Y\to Z$ if and only if for every open $V\subseteq Z$, the preimage 
    \[
    f^{-1}(V) = \{(x,y)\in X\times Y : f(x,y)\in V\}
    \]
    is separately open. This means that 
    \begin{itemize}
        \item For every $x\in X$, the set 
        \begin{align*}
        (f^{-1}(V))_x &= \{y\in Y: (x,y)\in f^{-1}(V)\} \\
        &= \{y\in Y: f(x,y)\in V\} \\
        &= \{y\in Y: f_x(y)\in V\} \\
        &= f_x^{-1}(V)
        \end{align*}
        is open. This means exactly that $f_x:Y\to Z$ is continuous;
        \item For every $x\in X$, the set 
        \begin{align*}
        (f^{-1}(V))_y &= f_y^{-1}(V)
        \end{align*}
        is open. This means exactly that $f_y:X\to Z$ is continuous.
    \end{itemize}
    Now $f$ is separately continuous if and only if for all $x$ and $y$, the functions $f_x:Y\to Z$ and $f_y:X\to Z$ are continuous.
\end{proof}

\subsection{The internal hom}

\begin{definition}
    Let $X$ and $Y$ be topological spaces. Denote by $[X,Y]$, and call it \newterm{internal hom}, the set of continuous functions $X\to Y$, equipped with the \newterm{topology of pointwise convergence}, generated by the (sub-basic) sets
    \[
    S(x,V) \coloneqq \{ f:X\to Y : f(x)\in V \}
    \]
    for each $x\in X$ and each open $V\subseteq Y$. 
\end{definition}

\begin{proposition}
    Let $X$ and $Y$ be topological spaces. Let $(f_\alpha)_{\alpha\in\Lambda}$ be a net (or a sequence) in $[X,Y]$. We have that $f_\alpha\to f$ if and only if for every $x\in X$, $f_\alpha(x)\to f(x)$.
\end{proposition}

\begin{proof}
    We have that $f_\alpha\to f$ if and only if every sub-basic open neighborhood $S(x,V)$ of $f$ eventually contains $f_\alpha$. That is, if for each $x\in X$ and for each open $V\subseteq Y$ with $f(x)\in V$, there exists a $\beta\in\Lambda$ such that for all $\alpha\ge\beta$, $f_\alpha(x)\in V$. 
    This is exactly saying that for each $x\in X$, the net $(f_\alpha(x))_{\alpha\in\Lambda}$ converges to $f(x)$. 
\end{proof}

\begin{proposition}
    Let $X$, $Y$ and $Z$ be topological spaces. We have a natural bijection
    \[
    \begin{tikzcd}[row sep=0]
        \cat{Top}(X\otimes Y, Z) \ar{r}{\sharp}[swap]{\cong} & \cat{Top}(X,{[Y,Z]}) \\
        f \ar[mapsto]{r} & \big( x\xmapsto{f^\sharp} f_x \big)
    \end{tikzcd}
    \]
    induced by the corresponding one of $\cat{Set}$.
\end{proposition}

Therefore $\cat{Top}$ with this tensor product and internal hom is closed monoidal.

\begin{proof}
    Note that the hom-tensor adjunction mapping $f\mapsto f^\sharp$ for $\cat{Set}$ is a natural bijection.
    Therefore it suffices to prove that $f:X\times Y\to Z$ is separately continuous if and only if 
    \begin{itemize}
        \item for every $x\in X$, $f^\sharp(x)=f_x$ is a continuous function $Y\to Z$;
        \item the map $f^\sharp:X\to[Y,Z]$ is continuous.
    \end{itemize}
    Now for each $x\in X$, the function $f_x$ is continuous if and only if for each open $V\subseteq Z$, the set $f_x^{-1}(V)=(f^{-1}(V))_x$ is open.
    Also, $f^\sharp:X\to[Y,Z]$ is continuous if and only if for each $y\in Y$ and each open $V\subseteq Z$, the set 
    \begin{align*}
    (f^\sharp)^{-1}(S(y,V)) &= \{ x\in X : f_x \in S(y,V) \} \\
        &= \{ x\in X : f_x(y) \in V \} \\
        &= \{ x\in X : f(x,y) \in V \} \\
        &= (f^{-1}(V))_y
    \end{align*}
    is open.
    The two conditions together are therefore equivalent to $f$ being separately continuous.
\end{proof}

\subsection{Categories enriched in Top}
\label{sec_topcat}

Consider $\cat{Top}$ with the closed monoidal structure of the previous section. 
Instantiating the usual definition of enriched category, we get:

\begin{definition}
    A \newterm{category enriched in $\cat{Top}$}, or \newterm{topologically enriched category}, amounts to 
    \begin{itemize}
        \item An ordinary category $\cat{C}$;
        \item On each hom-set $\cat{C}(X,Y)$, a topology, in such a way that each composition assignment
        \[
        \begin{tikzcd}[row sep=0]
            \cat{C}(X,Y) \times \cat{C}(Y,Z) \ar{r} & \cat{C}(X,Z) \\
            (f,g) \ar[mapsto]{r} & g\circ f
        \end{tikzcd}
        \]
        is separately continuous. 
    \end{itemize}
\end{definition}

The second condition is equivalent to the following two conditions:
\begin{itemize}
    \item Given a net (or sequence) $(f_\alpha:X\to Y)_{\alpha\in\Lambda}$ tending to $f:X\to Y$, and a morphism $g:Y\to Z$, we have that $g\circ f_\alpha\to g\circ f$, i.e.~postcomposition is a continuous operation;
    \item Given a morphism $f:X\to Y$ and a net (or sequence) $(g_\alpha:Y\to Z)_{\alpha\in\Lambda}$ tending to $g:Y\to Z$, we have that $g_\alpha\circ f \to g\circ f$, i.e.~precomposition is a continuous operation.
\end{itemize}

\newpage
\bibliographystyle{alpha}
\bibliography{markov}

\end{document}

%% file: markov2.tikzstyles
\tikzstyle{morphism}=[fill=white, draw=black, shape=rectangle]
\tikzstyle{medium box}=[fill=white, draw=black, shape=rectangle, minimum width=0.7cm, minimum height=0.7cm]
\tikzstyle{large morphism}=[fill=white, draw=black, shape=rectangle, minimum width=1.7cm, minimum height=1cm]
\tikzstyle{bn}=[fill=black, draw=black, shape=circle, inner sep=1.5pt]
\tikzstyle{state}=[fill=white, draw=black, regular polygon, regular polygon sides=3, minimum width=0.8cm, shape border rotate=180, inner sep=0pt]
\tikzstyle{medium state}=[fill=white, draw=black, regular polygon, regular polygon sides=3, minimum width=1.3cm, inner sep=0pt, shape border rotate=180]
\tikzstyle{large state}=[fill=white, draw=black, regular polygon, regular polygon sides=3, minimum width=2.2cm, shape border rotate=180, inner sep=0pt]
\tikzstyle{wide state}=[fill=white, draw=black, shape=isosceles triangle, minimum width=0.8cm, shape border rotate=270, inner sep=1.4pt, minimum height=0.5cm, isosceles triangle apex angle=80]
\tikzstyle{wn}=[fill=white, draw=black, shape=circle, inner sep=1.5pt]
\tikzstyle{blue morphism}=[fill=white, draw={rgb,255: red,15; green,0; blue,150}, shape=rectangle, text={rgb,255: red,15; green,0; blue,150}, tikzit category=blue]
\tikzstyle{blue state}=[fill=white, draw={rgb,255: red,15; green,0; blue,150}, shape=circle, regular polygon, regular polygon sides=3, minimum width=0.8cm, shape border rotate=180, inner sep=0pt, text={rgb,255: red,15; green,0; blue,150}, tikzit category=blue]
\tikzstyle{blue node}=[fill={rgb,255: red,15; green,0; blue,150}, draw={rgb,255: red,15; green,0; blue,150}, shape=circle, tikzit category=blue, inner sep=1.5pt]
\tikzstyle{blue}=[text={rgb,255: red,15; green,0; blue,150}, tikzit draw={rgb,255: red,191; green,191; blue,191}, tikzit category=blue, tikzit fill=white, inner sep=0mm]
\tikzstyle{blue wide state}=[fill=white, draw={rgb,255: red,15; green,0; blue,150}, text={rgb,255: red,15; green,0; blue,150}, shape=isosceles triangle, minimum width=0.8cm, shape border rotate=270, inner sep=1.4pt, minimum height=0.5cm, isosceles triangle apex angle=80]
\tikzstyle{red node}=[fill={rgb,255: red,150; green,0; blue,2}, draw={rgb,255: red,150; green,0; blue,2}, shape=circle, inner sep=1.5pt]
\tikzstyle{Purple node}=[fill={rgb,255: red,120; green,0; blue,120}, draw={rgb,255: red,120; green,0; blue,120}, text={rgb,255: red,120; green,0; blue,120}, shape=circle, inner sep=1.5pt]
\tikzstyle{red}=[text={rgb,255: red,150; green,0; blue,2}, inner sep=0mm, tikzit fill=white, tikzit draw={rgb,255: red,191; green,191; blue,191}]
\tikzstyle{purple}=[text={rgb,255: red,150; green,0; blue,150}, inner sep=0mm, tikzit fill=white, tikzit draw={rgb,255: red,191; green,191; blue,191}]
\tikzstyle{white morphism}=[fill=white, draw=white, shape=rectangle, tikzit draw={rgb,255: red,139; green,139; blue,139}]
\tikzstyle{leak morphism}=[fill=white, draw={rgb,255: red,120; green,0; blue,85}, shape=rectangle, text={rgb,255: red,120; green,0; blue,85}, tikzit category=leak]
\tikzstyle{leak}=[text={rgb,255: red,120; green,0; blue,85}, inner sep=0mm, tikzit fill=white, tikzit draw={rgb,255: red,191; green,191; blue,191}, tikzit category=leak]
\tikzstyle{leak node}=[fill={rgb,255: red,120; green,0; blue,85}, draw={rgb,255: red,120; green,0; blue,85}, shape=circle, inner sep=1.5pt, tikzit category=leak]

\tikzstyle{arrow}=[->]
\tikzstyle{dashed box}=[-, dashed]
\tikzstyle{blue line}=[-, draw={rgb,255: red,15; green,0; blue,150}, tikzit category=blue]
\tikzstyle{red arrow}=[-, draw={rgb,255: red,150; green,0; blue,2}, tikzit category=red]
\tikzstyle{purple line}=[draw={rgb,255: red,120; green,0; blue,120}, >=stealth, shorten <=2pt, shorten >=2pt, -]
\tikzstyle{protected purple line}=[draw={rgb,255: red,120; green,0; blue,120}, >=stealth, shorten <=2pt, shorten >=2pt, preaction={line width=1.8pt, white, draw}, -]
\tikzstyle{mapsto}=[{|->}]
\tikzstyle{double wire}=[-, double]
\tikzstyle{protected}=[-, preaction={line width=1.8pt,white,draw}]
\tikzstyle{leak arrow}=[-, line join=round, decorate, decoration={snake, segment length=4, amplitude=0.75, pre=curveto, post=curveto, pre length=1pt, post length=1pt}]
\tikzstyle{protected leak arrow}=[-, line join=round, decorate, decoration={snake, segment length=4, amplitude=0.75, pre=curveto, post=curveto, pre length=1pt, post length=1pt}, preaction={line width=1.8pt, white, draw}]
\tikzstyle{hollow arrow}=[-, very thin, white, preaction={line width=0.7pt,draw={rgb,255: red,120; green,0; blue,85}}, tikzit category=leak, tikzit draw={rgb,255: red,150; green,0; blue,120}]
\tikzstyle{protected hollow arrow}=[-, very thin, white, preaction={line width=0.7pt,draw={rgb,255: red,120; green,0; blue,85},preaction={line width=2.1pt,white,draw}}, tikzit category=leak, tikzit draw={rgb,255: red,150; green,0; blue,120}]
\tikzstyle{curly brace}=[-, decorate, decoration={brace,amplitude=5pt}]

%% file: cache.tex
\usetikzlibrary{external}
\tikzsetexternalprefix{cache/}
\usepackage{etoolbox}
\AtBeginEnvironment{tikzcd}{\tikzexternaldisable}
\AtEndEnvironment{tikzcd}{\tikzexternalenable}
\makeatletter
\renewcommand{\todo}[2][]{\tikzexternaldisable\@todo[#1]{#2}\tikzexternalenable}
\makeatother